\documentclass[reqno]{amsart}
\usepackage{amsmath,amssymb}
\usepackage{amsmath,amssymb,cite}
\usepackage{graphics,epsfig,color}
\usepackage[mathscr]{euscript}
\usepackage{xcolor}
\usepackage{enumitem}
\usepackage{a4wide}
\usepackage{tikz}

\theoremstyle{plain}
\newtheorem{theorem}{Theorem}[section]
\newtheorem{proposition}[theorem]{Proposition}
\newtheorem{lemma}[theorem]{Lemma}
\newtheorem{corollary}[theorem]{Corollary}

\theoremstyle{definition}
\newtheorem{definition}[theorem]{Definition}

\newtheorem{remark}[theorem]{Remark}

\makeatletter
\@addtoreset{equation}{section}
\makeatother

\newcommand{\mc}[1]{{\mathcal #1}}
\newcommand{\mb}[1]{{\mathbf #1}}
\newcommand{\mf}[1]{{\mathfrak #1}}

\newcommand{\bb}[1]{{\mathbb #1}}
\newcommand{\ms}[1]{{\mathscr #1}}
\newcommand{\mss}[1]{{\textsf #1}}
\newcommand{\mtt}[1]{{\mathtt #1}}

\newcounter{as}[section]

\title[One-dimensional parabolic equations asymptotics]{From
one-dimensional diffusion processes metastable behaviour to parabolic
equations asymptotics.  }

\author{Claudio Landim, and Christian Maura}

\address{IMPA, Estrada Dona Castorina 110, J. Botanico, 22460 Rio de
Janeiro, Brazil and Univ. Rouen Normandie, CNRS,
LMRS UMR 6085,  F-76000 Rouen, France. \\
e-mail: \texttt{landim@impa.br} }

\address{IMPA, Estrada Dona Castorina 110, J. Botanico, 22460 Rio de
Janeiro, Brazil.
e-mail: \texttt{christian.maura@impa.br} }



\begin{document}
 \maketitle

\begin{abstract}
Consider the one-dimensional elliptic operator given by
\begin{equation*}
(\ms L_\epsilon f)(x)  \;=\;
\mss b (x) \, f'(x)  \,+\,
\epsilon\, \mss a (x)\, f''(x) \;,
\end{equation*}
where the drift $\mss b\colon \bb R \to \bb R$ and the diffusion
coefficient $\mss a\colon \bb R \to \bb R$ are periodic $C^1(\bb R)$
functions satisfying further conditions, and $\epsilon>0$.  Consider
the initial-valued problem
\begin{equation*}
\left\{ \begin{aligned} &
\partial_{t}\,u_{\epsilon}\,=\,\mathscr{L}_{\epsilon}\,u_{\epsilon}\;,\\
& u_{\epsilon}(0,\,\cdot)=u_{0}(\cdot)\;,
\end{aligned}
\right.\end{equation*} for some bounded continuous function
$u_{0}$. We prove the existence of time-scales
$\theta_{\epsilon}^{(1)},\,\dots,\,\theta_{\epsilon}^{(\mathfrak{q})}$
such that $\theta_{\epsilon}^{(1)}\to\infty$,
$\theta_{\epsilon}^{(p+1)}/\theta_{\epsilon}^{(p)}\to\infty$,
$1\le p\le\mathfrak{q}-1$, probability measures $p(x,\cdot)$, $x\in
\bb R$, and kernels $R_{t}^{(p)}(m_j,m_k)$, where $\{m_j:j\in\bb Z\}$
represents the set of stable equilibrium of the ODE $\dot{\mtt x}(t) =
\mss b(\mtt x(t))$ such that 
\begin{equation*}
\lim_{\epsilon\to0} u_{\epsilon}(t\theta_{\epsilon}^{(p)}, x)
\;=\;\sum_{j,k\in \bb Z} p(x,m_j)\, R_{t}^{(p)} (m_j,m_k) \,u_{0}(m_k)\;,
\end{equation*}
for all $t>0$ and $x\in \bb R$. The solution $u_{\epsilon}$ asymptotic
behavior description is completed by the characterisation of its
behaviour in the intermediate time-scales $\varrho_{\epsilon}$ such
that $\varrho_{\epsilon}/\theta_{\epsilon}^{(p)}\to\infty$,
$\varrho_{\epsilon}/\theta_{\epsilon}^{(p+1)}\to0$ for some
$0\le p\le\mathfrak{q}$, where $\theta_{\epsilon}^{(0)}=1$,
$\theta_{\epsilon}^{(\mathfrak{q}+1)}=+\infty$.

The proof relies on the analysis of the diffusion $X_\epsilon(\cdot)$
induced by the generator $\ms L_\epsilon$ based on the resolvent
approach to metastability introduced in \cite{lms}.
\end{abstract}

\section{Introduction}
\label{sec0}

Fix $\epsilon>0$, and consider the one-dimensional elliptic operator
$\ms L_\epsilon$ given by
\begin{equation*}
{\color{blue}\ms L_\epsilon f } \;:=\;
\mss b \, \partial_x  f  \,+\,
\epsilon\, \mss a\, \partial^2_x f \;.
\end{equation*}
Denote by $C^k(\bb R)$, $k\in\bb N$, the functions
$f\colon \bb R \to\bb R$ with $k$ continuous derivatives.  We assume
that the drift $\mss b\colon \bb R \to \bb R$ and the diffusion
coefficient $\mss a\colon \bb R \to \bb R$ are periodic with, say,
period $1$. We suppose that $\mss a(\cdot)$, $\mss b(\cdot)$ are of
class $C^1(\bb R)$, and that their first derivatives are Lipschitz
continuous. We assume furthermore that the diffusion coefficient
$\mss a(\cdot)$ is strictly positive: $\mss a(x) \ge c_0 >0$, and that
in the interval $[0,1)$ the drift $\mss b$ vanishes at a finite number
of points represented by $x_j$, $1\le j\le q$, and that
$\mss b'(x_j) \neq 0$. This later condition and the periodicity of
$\mss b(\cdot)$ imply that the number of such points has to be even:
$q=2N$ for some $N\ge 1$.

Denote by $\mc M$, $\mc W \subset \bb R$ the stable, unstable
equilibrium points of the ODE $\dot x(t) = \mss b(x(t))$,
\begin{equation}
\label{76}
{\color{blue} \mc M} \,:=\,  \{ x \in \bb R:
\mss b(x ) = 0 \;,\; \mss b'(x) < 0\}
\;, \quad
{\color{blue} \mc W} \,:=\,  \{ x \in \bb R:
\mss b(x ) = 0 \;,\;  \mss b'(x ) > 0\} \;.
\end{equation}
Enumerate the elements in $\mc M$ and $\mc W$ as
$\mc M = \{m_k : k\in \bb Z \}$, $\mc W = \{\sigma_k : k\in \bb Z \}$,
and assume, without loss of generality, that
\begin{equation}
\label{68}
\sigma_k \,<\,  m_k \,<\, \sigma_{k+1}\;, \quad 
0 < \sigma_0 < m_0 < \dots < \sigma_{N-1} < m_{N-1} < 1\;.
\end{equation}

Fix a bounded, continuous function $u_0 : \bb R \to \bb R$, and
consider the initial-valued problem
\begin{equation}
\label{51}
\left\{
\begin{aligned}
& \partial_ t u_\epsilon\,=\, \ms L_\epsilon u_\epsilon\;, \\
& u_\epsilon (0, \cdot) = u_0(\cdot) \;.
\end{aligned}
\right.
\end{equation}
In this article, we investigate the asymptotic behavior of the
solution $u_\epsilon(t,x)$ of this equation.

The first main result of the article states that there exists a time-scale
$\theta^{(1)}_{\epsilon}$, $\theta^{(1)}_{\epsilon} \to \infty$, such that
\begin{equation*}
\lim_{\epsilon\to0}
u_{\epsilon}(\varrho_{\epsilon}, x)\;=\;
u_{0}(m_k) 
\end{equation*}
for all $k\in \bb Z$, $x \in (\sigma_k, \sigma_{k+1})$ and time-scale
$\varrho_{\epsilon}$ such that $\varrho_{\epsilon} \to +\infty$,
$\varrho_{\epsilon} / \theta^{(1)}_{\epsilon} \to 0$. Moreover, there
exists a Markov semigroup $(p_{t}:t\ge0)$ defined on
${\color{blue} \ms {M}_{0}} := \{\{m\}: m\in \mc M\}$ such that
\begin{equation*}
\lim_{\epsilon\to0}u_{\epsilon}(s\, \theta^{(1)}_{\epsilon}, x)
\;=\;\sum_{k\in \bb Z}
p_{s}(\{m_j\}, \{m_k\})\, u_{0}(m_k)
\end{equation*}
for all $s>0$, $j\in \bb Z$, $x \in (\sigma_j, \sigma_{j+1})$.  Hence,
the solution does not change until the time-scale
$\theta^{(1)}_{\epsilon}$, when it start to evolve according to the
semigroup $(p_t: t\ge 0)$.  This analysis is extended to longer
time-scales and provides a full description of the solution
$u_\epsilon$ asymptotic behavior.

We present an inductive procedure which provides

\begin{itemize}
\item[(a)] A sequence of time-scales $\theta^{(p)}_{\epsilon}$,
$1\le p\le \mf q$, such that
$\theta^{(p+1)}_{\epsilon} / \theta^{(p)}_{\epsilon} \to \infty$,
$1\le p <\mf q$, and $\theta^{(1)}_\epsilon = \theta_\epsilon$.

\item[(b)] A sequence of partitions
$\{\ms M_p(j), j \in \bb Z\} \cup \{\ms T_p\}$ of $\mc M$,
$1\le p\le \mf q$, where $\ms M_1(j) = \{m_j\}$, $j\in \bb Z$,
$\ms T_1= \varnothing$.

\item[(c)] A sequence of non-degenerate $\ms S_p$-valued
continuous-time Markov semigroups $(p^{(p)}_t : t\ge 0)$. Here,
$\ms S_p = \{\ms M_p(j): j\in \bb Z\}$,
$p^{(1)}_{s}(\cdot, \cdot) = p_{s}(\cdot, \cdot)$, and non-degenerate
means that there exists $k\neq j\in \bb Z$ such that
$p^{(p)}_t(\ms M_p(j), \ms M_p(k))>0$ for some (and therefore all)
$t>0$.
\end{itemize}

The partitions are constructed inductively as follows. Assume that the
construction has been carried out up to layer $p$. Denote by
$\{\mf R^{(p)}_k : k\in K\}$ the closed irreducible classes of the
Markov chain induced by the semigroup $(p^{(p)}_t : t\ge 0)$. Then, up
to relabelling to order the sets $\ms M_{p+1}(j)$,
\begin{equation*}
\ms M_{p+1}(k) \,=\, \bigcup_{i: \ms M_{p}(i) \in \mf R^{(p)}_k} \ms
M_p(i)\;. 
\end{equation*}
The construction ends after a finite number of steps: the Markov chain
induced by the semigroup $(p^{(\mf q)}_t : t\ge 0)$ has only one closed
irreducible class or only transient states.

The previous construction determines a tree with $\mf q$
generations. Each element of the tree is a subset of $\mc M$.  The
leaves are the sets $\{m_j\}$, $j\in \bb Z$. Each parent is the union
of its children and corresponds to a closed irreducible class of a
Markov chain. The root represents the unique closed irreducible class
of the last Markov chain or the union of its transient states.

The second main result of the article states that 
\begin{equation}
\label{107}
\lim_{\epsilon\to0}\,
u_{\epsilon}(\theta_{\epsilon}^{(p)}t\,,\,x) \ =\
\sum_{j\in\bb Z}\,
\mtt h_p (x , \ms M_p(j))\, \sum_{k\in\bb Z}
p_{t}^{(p)}(\ms M_p(j), \ms M_p(k))\,
\sum_{m'\in \ms M_p(k)}
\frac{\pi(m')}{\pi(\ms M_p(k))}
\,u_{0}(m')\ .
\end{equation}
for all $1\le p\le \mf q$, $t>0$, $x\in \bb R$.  In this formula,
$\mtt h_p (x , \ms M_p(j))$ has to be understood as the (asymptotic)
probability that the diffusion induced by the generator
$\ms L_\epsilon$ and starting from $x$ hits the set
$\cup_{i\in \bb Z} \ms M_p(i)$ at $\ms M_p(j)$, and
$\pi(m')/\pi(\ms M_p(k))$ as the proportion of time this diffusion
spends close to $m'$ when it visits the set $\ms M_p(k)$.
We provide in \eqref{74} and \eqref{104} explicit formulas for
$\mtt h_p (x , \cdot)$, $\pi(\cdot)$, respectively, as functions of
$\mss a(\cdot)$, $\mss b(\cdot)$.

For any intermediate scale $\varrho_\epsilon$ such that
$\varrho_\epsilon / \theta^{(p)}_{\epsilon} \to +\infty$,
$\varrho_\epsilon / \theta^{(p+1)}_{\epsilon} \to 0$, $1\le p< \mf
q$,
\begin{equation}
\label{107b}
\lim_{\epsilon\to0}\,
u_{\epsilon}(\varrho_\epsilon \,,\,x) \ =\
\sum_{k\in\bb Z}\,
\mtt h_p (x , \ms M_{p+1}(k))\, 
\sum_{m'\in \ms M_{p+1}(k)}
\frac{\pi(m')}{\pi(\ms M_{p+1}(k))}
\,u_{0}(m')\ .
\end{equation}
for all $x\in \bb R$. The right-hand side can be obtained from the
right-hand side of \eqref{107} by letting $t\to\infty$.  A similar
result can be derived for a time-scale $\varrho_\epsilon$ such that
$\varrho_\epsilon / \theta^{(\mf q)}_{\epsilon} \to +\infty$, but it
requires the initial condition $u_0(\cdot)$ to be periodic.

\subsection*{Resolvent approach to metastability}

The proof is based on the stochastic representation of the solution
$u_\epsilon(t, \cdot)$ to \eqref{51} in terms of the diffusion
$X_\epsilon(\cdot)$ induced by the generator $\ms L_\epsilon$.

The Freidlin and Wentzell \cite{fw} large deviations theory for
diffusions permit to characterise the asymptotic behaviour of the
solution $u_\epsilon(t, \cdot)$ in the intermediate time-scales
$(\varrho_\epsilon :\epsilon >0)$ such that
$\theta^{(p)}_{\epsilon} \prec \varrho_\epsilon \prec
\theta^{(p+1)}_{\epsilon}$, $1\le p<\mf q$, \cite{fk10a, fk10b,
kt16}. Here and below, for two positive sequences $\lambda_{\epsilon}$
and $\mu_\epsilon$, write
${\color{blue}\lambda_{\epsilon}\prec\mu_{\epsilon}}$ or
${\color{blue} \mu_{\epsilon} \succ \lambda_{\epsilon}}$ if
$\lambda_{\epsilon}/\mu_{\epsilon}\to0$ as $\epsilon\rightarrow0$.
Similar results were obtained in \cite{is15, is17} with purely
analytical methods and in \cite{Miclo2} with stochastic analysis
tools.

More recently, sharp estimates for the transition time between two
wells, the so-called Eyring-Kramers formula, have been obtained for
diffusion processes \cite{BEGK, LMS2, LS-22} based on potential theory
\cite{l-review}. These results leaded to the description of the
evolution of small perturbations of dynamical systems between the
deepest wells, the so-called metastable behavior \cite{RS,
LS-22b}. Similar results were obtained with analytical techniques,
\cite{LM, BLS} and references therein.

We follow here the resolvent approach to metastability \cite{llm,
lms}.  It consists in showing first that the solution of a resolvent
equation is asymptotically constant on each well. More precisely,
denote by $\ms E(\ms M_p(k))$, $k\in \bb Z$, the wells of the
diffusion $X_\epsilon(\cdot)$ observed at the time-scale
$\theta_{\epsilon}^{(p)}$. Fix $\lambda>0$, a function
$\mf g \colon \bb Z\to \bb R$, and let $\phi_\epsilon(\cdot)$ be the
solution of the resolvent equation
\begin{equation}
\label{41}
(\lambda - \theta_{\epsilon}^{(p)}\,
\ms L_{\epsilon})\, \phi_{p,\epsilon} \,=\, G \,,
\end{equation}
where
\begin{equation}
\label{55}
G (x) \,=\, \sum_{k\in \bb Z} \mf g (k)\,
\chi_{_{\ms E (\ms M_{p}(k))}} (x)\;,
\end{equation}
and $\color{blue} \chi_{_{\ms A}}$, $\ms A \subset \bb R$, represents
the indicator function of the set $\ms A$.  Theorem \ref{mt4} asserts
that on each well $\ms E(\ms M_p(k))$ $\phi_\epsilon(\cdot)$
converges, as $\epsilon\to 0$, to $\mf f(k)$, where $\mf f$ is the
solution of the reduced resolvent equation
\begin{equation*}
(\lambda - \mf L_p) \, \mf f = \mf g\,,    
\end{equation*}
where $\mf L_p$ is the generator of a $\bb Z$-valued Markov chain $\mf
X_p$. 

From this result, one deduces in Theorem \ref{mt3} that probability
of the event $\{X_\epsilon(t \theta_{\epsilon}^{(p)}) \in \ms E (\ms
M_{p}(k)) \}$  converges to the probability that $\mf X_p(t) =k$ for
all $k\in \bb Z$, $t>0$. From this convergence and the stochastic
representation of the solution of linear parabolic equation, one
completes the proof of \eqref{107}, \eqref{107b}.

% E^p ou E ?

This approach has been successfully applied in \cite{lls1, lls2} to
describe the asymptotic behavior at the critical time-scales of the
linear parabolic equations solution of 
\begin{equation*}
\partial_t u_\epsilon \, = \ms L^V_\epsilon u_\epsilon \;,
\quad \ms L^V_\epsilon f \,=\,
\epsilon \, \Delta f \,-\,
(\nabla V + \ell) \, \nabla f\,,
\end{equation*}
where $V$ is a Morse potential and $\ell$ a divergence free field
orthogonal to $V$. The conditions on the field $\ell$ are necessary
and sufficient for the measure $Z_\epsilon^{-1} \exp\{- V(x)/\epsilon\}
\, dx$ to be stationary for the diffusion process induced by the generator
$\ms L^V_\epsilon$ \cite{LS-22}.

In this article, we investigate the one-dimensional version of this
problem, allowing the diffusion coefficient $\mss a(\cdot)$ and the
drift $\mss b(\cdot)$ to be any periodic $C^1$ function whose
derivative is Lipschitz continuous. One of the main obstacles lies in
the fact that the stationary state is not explicitly known, preventing
computation of capacities. In \cite{ls}, for instance, assuming that
$\mss a(\cdot)$ is constant, based on an explicit formula for the
quasi-potential derived in \cite{fg1}, the evolution of the diffusion
among the deepest wells was obtained from the computation of the
capacities.

It turns out that the proof does not rely deeply on the periodicity of
the parameters $\mss a(\cdot)$, $\mss b(\cdot)$. It is plausible that
the random coefficients case \cite{BR86, Chel15} can be handled with
some adaptations, or their discrete versions \cite{BF08}. The
degenerate case in which the first $2n$, $n\ge 1$, derivatives of the
drift $\mss b(\cdot)$ vanish at the points where $\mss b(\cdot)$
vanishes seems also possible to tackle.

The article is organised as follows. The main results are presented in
Section \ref{sec7}. In Section \ref{sec1}, we introduce the
one-dimensional diffusion process $X_\epsilon(\cdot)$ and state the
results describing its asymptotic behavior at the critical
time-scales. In Section \ref{sec2}, we present some estimates on
hitting times needed throughout the paper. This section can be skipped
at a first reading. In Sections \ref{sec6} and \ref{sec5}, we derive
the asymptotic behavior of the resolvent equation solution in the
critical time-scales (Theorem \ref{mt4}). In Section \ref{sec8}, we
describe the metastable behavior of the diffusion at the critical
time-scales, and we establish the finite-dimensional distributions
convergence. In Section \ref{sec9}, we prove Theorems \ref{mt1} and
\ref{mt2}. In the first appendix, we present properties of the
hierarchical structure introduced in Section \ref{sec7}, and used
throughout the paper. In Appendix \ref{sec-ap1}, we present some
properties of elliptic equations weak solutions needed in the previous
sections.

\section{Notation and Results}
\label{sec7}

In this section, we state the main results of the article.  Denote by
$S\colon \bb R\to \bb R$ the potential defined by
\begin{equation}
\label{34}
S(x) \;=\; -\, \int_0^x \frac{\mss b(z)}{\mss a(z)}   \, dz\;.
\end{equation}
Let $\ms M_{1}(k)$, $k\in \bb Z$, be the sets
defined by
\begin{equation*}
{\color{blue}\ms M_{1}(k)} \,:=\, \{m_{k}\} \;, \;\;
\text{and set}\;\;
{\color{blue}\ms S_{1}}:= \{\ms M_{1}(k) : k\in \bb Z\}\;.
\end{equation*} 
The first time-scale encodes the size of the minimal energy barrier
between $S$-local minima. Let
\begin{equation}
\label{f01}
{\color{blue} h^{+}_{k}}  := S(\sigma_{k+1}) - S(m_{k}),
\hspace{10mm}
{\color{blue}  h^{-}_{k}} := S(\sigma_{k}) - S(m_{k}),
\end{equation}
which can be regarded as escape barriers on the right and left of the
set $\ms M_{1}(k)$, respectively. Clearly, as $\mss a(\cdot)$, $\mss
b(\cdot)$ are periodic functions,
\begin{equation}
\label{70b}
h^\pm_{k+N} \,=\, h^\pm _k\quad\text{for all}\;\;
k\in \bb Z\;. 
\end{equation}
For each $k\in \bb Z$ define the height (of the escape barrier) of
$\ms M_{1}(k)$ as the minimum of its right and left escape
barriers. Let $\mf h_1$ be the smallest height, and
$\theta_{\epsilon}^{(1)}$ the associated first time-scale:
\begin{equation}
\label{69b}
{\color{blue}   h_k}  \,:=\, \min\{h^{-}_{k}, h^{+}_{k}\}
\;, \; k\in \bb Z\;, \qquad
{\color{blue}\mf h_1}  \,:=\, \min\{h_k : k\in \bb Z\}\;,
\qquad
{\color{blue}\theta_{\epsilon}^{(1)}} \,:=\,
e^{\mf h_1/\epsilon}\;.
\end{equation}
By \eqref{70b}, $\mf h_1 = \min\{h_k : 0 \le k<N \}$, so that
$\mf h_1>0$.

For $k\in\bb Z$, let
\begin{equation}
\label{04}
\begin{gathered}
{\color{blue}\pi_{1}(k) \, :=\, \pi(\{m_k\})} \,=\,
\frac{1}{\mss a(m_k)}\, \sqrt{\frac{2\pi}{S''(m_k)}}
\,=\, \sqrt{\frac{-\, 2\pi}{\mss b'(m_k)\, \mss a(m_k)}} \;,
\\
{\color{blue} \sigma_1(k,k+1)}
\, :=\,  \sqrt{\frac{2\pi}{- S''(\sigma_{k+1})}}
\,=\, \sqrt{\frac{2\pi \, \mss a(\sigma_{k+1}) }{\mss b'(\sigma_{k+1})}}
\;\cdot
\end{gathered}
\end{equation}
Here, the indices $k$, $k+1$ of $\sigma_1(k,k+1)$ indicate that we
consider the $S$-local maximum between the local minima $m_k$ and
$m_{k+1}$.  Let $\color{blue} \sigma_1(k+1,k) = \sigma_1(k,k+1)$.
Define the jump rates
\begin{equation}
\label{73}
{\color{blue} R_1 (\ms M_{1}(k),\ms M_{1}(k\pm 1))} \; := \;
\frac{1}{\pi_{1}(k)\, \sigma_1(k,k\pm 1) }
\; \mb 1\{ h^{\pm }_{k} = \mf h_1\} \;.
\end{equation}
Therefore, the jump rate from $\ms M_{1}(k)$ to $\ms M_{1}(k\pm 1)$
vanishes if $h^{\pm }_{k} > \mf h_1$. [By definition
$\mf h_1 \le h^{\pm }_{j}$ for all $j\in\bb Z$.]  Denote by
$\color{blue} \bb X_1 (\cdot)$ the $\ms S_{1}$-valued Markov chain
induced by the generator $\mathbb{L}_1$ defined as
\begin{equation}
\label{72b}
{\color{blue} (\bb L_1 f)(\ms M_{1}(k))} \,:=\, 
\sum_{a = \pm 1} R_1(\ms M_{1}(k), \ms M_{1}(k+a))\,
\big[\, f(\ms M_{1}(k+a)) - f(\ms M_{1}(k))\, \big]\;, \quad k\in \bb
Z\;. 
\end{equation}

The first main result of this article reads as follows.  Denote by
$\color{blue} \mc D(m)$, $m\in\mc M$, the domain of attraction of $m$
for the ODE $\dot x(t) = \mss b(x(t))$:
$\mc D(m_k) = (\sigma_k, \sigma_{k+1})$, $k\in\bb Z$. For
$x\in \mc D(m_j)$, let
${\color{blue} \mtt h(x, \ms M_1(k))} = \delta_{j,k}$, where
$\color{blue} \delta_{j,k}$ represents Kronecker's delta, while for
$x=\sigma_j$,
$\mtt h(x, \ms M_1(k)) = (1/2) \delta_{k,j} + (1/2) \delta_{k,j-1}$.

\begin{theorem}
\label{mt1}
Fix a bounded continuous function $u_0\colon \bb R \to \bb R$.  Denote
by $u_\epsilon$ the solution of the parabolic equation
\eqref{51}. Then,
\begin{equation*}
\lim_{\epsilon\to 0} u_{\epsilon}(t\,\theta_{\epsilon}^{(1)},x)
\;=\; \sum_{j\in \bb Z} \mtt h(x, \ms M_1(j))\, \sum_{k\in \bb Z}
p_{t}(\ms M_1(j), \ms M_1(k))\,u_{0}(m_k)\;,
\end{equation*}
for all $t>0$, $x\in \bb R$. In this formula,
$p_t(\cdot, \cdot)$ represents the transition probability of the
Markov chain $\bb X_1(\cdot)$.  Moreover, for any sequence
$1 \prec \varrho_{\epsilon}\prec\theta_{\epsilon}^{(1)}$
\begin{equation*}
\lim_{\epsilon\to0}u_{\epsilon}(\varrho_{\epsilon}, x)
\;=\;u_{0}(m) \;,
\end{equation*}
for all $x\in\mathcal{D}(m)$, and for any sequence
$\epsilon^{-1} \prec \varrho_{\epsilon}\prec\theta_{\epsilon}^{(1)}$
\begin{equation*}
\lim_{\epsilon\to0}u_{\epsilon}(\varrho_{\epsilon}, \sigma_j)
\;=\; (1/2) \, u_{0}(m_{j-1}) \,+\, (1/2) \, u_{0}(m_j) \;,
\end{equation*}
for all $j\in \bb Z$.

\end{theorem}

\subsection*{Longer time-scales}

In this subsection, we present a recursive method to derive the
higher-order metastable structure of the process $X_\epsilon(\cdot)$.
At each iteration $q\geq 1$, the inductive procedure yields a triplet
\begin{equation}
\label{meta_trip}
{\color{blue}
\Gamma_{q}:= (\mc P_{q} \,,\, \mf h_{q} \,,\, \bb L_q)},   
\end{equation}
where $\mc P_{q} := \ms S_q \cup \{ \ms T_q\}$ is a partition of the
set $\mc M$,
${\color{blue} \ms S_q} := \{ \ms M_{q}(j) : j\in \bb Z\}$,
$\mf h_q>0$ is an energy barrier, and $\bb L_q$ is the generator of a
$\ms S_q$-valued Markov chain, denoted by $\bb X_q(\cdot)$.

In this construction, two subsets $A$, $B$ of $\bb R$ are said to be
equivalent, $\color{blue} A\sim B$, if $B=A+k$ for some $k\in \bb Z$,
where $\color{blue} A+k = \{x+k : x\in A\}$.

The first layer of this iterative construction has been carried out in
the previous section. The components of the triplet $\Gamma_1$ are
$\mc P_1 = \ms S_1 = \{ \{m_j\} : j\in \bb Z\}$,
$\ms T_1=\varnothing$, $\mf h_1$ defined in \eqref{69b}, the Markov
chain $\bb X_1(\cdot)$ induced by the generator $\bb L_1$ defined in
\eqref{72b}.

The recursive constructions of the triplets can be divided in three
stages.  We start showing that the first layer $\Gamma_1$ fullfils
conditions $\mc P_1$---$\mc P_{10}$ introduced below. Then, assuming
that the layer $\Gamma_p$ satisfy these conditions and that there is
one or more $\bb X_p$-closed irreducible class, we define a new layer
$\Gamma_{p+1}$. To complete the construction, we prove that the new
layer satisfy properties $\mc P_1$---$\mc P_{10}$, and that the number
of $\bb X_{p+1}$-closed irreducible classes, up to equivalences, is
strictly smaller than the $\bb X_p$-closed irreducible classes. This
last property ensures that the recursive procedure ends after a finite
number of steps.

We start with the properties of the first layer.  Recall the
definition of the $\Gamma_1$-components given in the previous
subsection.  It is clear that the following conditions are satisfied
for $q=1$.

\begin{enumerate}
\item[[$\mc P_1(q)$\!\!\!]] For every $k\in \bb Z$,
$\ms M_{q}(k)\neq \varnothing$ and $\ms M_{q}(k) \subset \mc M$.
\end{enumerate}

Two subsets $C$, $D$ of $\bb R$ are said to be {\it well ordered}, if
$C\cap D = \varnothing$ and $x<y$ or $x>y$ for every $x\in C$ and
$y\in D$. In the first case this relation is represented by
$\color{blue} C<D$, while in the second one it is denoted by
$\color{blue} C>D$. 

\begin{enumerate}

\item[[$\mc P_2(q)$\!\!\!]]  $\ms M_{q}(j) < \ms M_{q}(k)$ for $j<k$,
and $m_{\mss j_q} \in \ms M_q(0)$, where
\begin{equation}
\label{24b}
{\color{blue} \mss j_q } \,:=\, \min\{k\ge 0 : m_k \in \ms M_q\}\;,
\quad
{\color{blue}\ms M_{q}} := \bigcup_{k\in \bb Z}\ms M_{q}(k)\;.
\end{equation}
\end{enumerate}

The integer $\mss j_q$ is introduced to enumerate the sets
$\ms M_q(i)$, $i\in\bb Z$. The set $\ms M_q(0)$ will be the one which
contains $m_{\mss j_q}$.  Recall the equivalence relation $\sim$
introduced above, and denote by $\color{blue} \mf u_q\in \bb N$ the
number of equivalent classes of the set $\ms S_q$, so that
$\mf u_1=N$.

\begin{enumerate}

\item[[$\mc P_3(q)$\!\!\!]] 
$\ms M_q(j + \mf u_q) = \ms M_q(j) + 1$ for all $j\in \bb Z$.
In particular, $\ms M_q(j + \mf u_q) \sim \ms M_q(j)$.

\end{enumerate}

Let $C$ be a non empty subset of $\bb R$. We say that the elements in
$C$ have the same \emph{depth} if $S(x) = S(y)$ for all $x$,
$y \in C$. For a set $A\subset \bb R$ where all elements share the
same depth, let $S(A)$ denote the common value of $S$ at $A$:
\begin{equation}
\label{16}
{\color{blue} S(A)} \,:=\,  S(x),\;\; \forall\;x\in A.
\end{equation}

Since the sets $\ms M_{1}(k)$ are singletons,

\begin{enumerate}

\item[[$\mc P_4(q)$\!\!\!]] For all $k\in \bb Z$, the local minima in
$\ms M_{q}(k)$ have the same \emph{depth}.

\end{enumerate}

Given two subsets $A$, $B$ of $\bb R$ such that $A<B$ or $B<A$, define
the \emph{barrier} between $A$ and $B$ as
\begin{equation}
\label{09}
{\color{blue} \Lambda (A,B)}  \,:= \,
\sup_{x\in [x^{+}_{A}, x^{-}_{B}]}S(x)
\;\;\text{if}\;\;
A<B\quad\text{and}\quad 
{\color{blue}  \Lambda (A,B) }  \,:= \, \Lambda  (B,A)\;\;\text{if}\;\; B<A\;.
\end{equation}
In this formula, $x^{+}_{A}$ is the rightmost element of the closure of $A$, and
$x^{-}_{B}$ is the leftmost element on the closure of $B$:
$x^{+}_{A}:= \sup\{x\in A\}$, $x^{-}_{B}:= \inf\{x\in B\}$.

Condition $\mc P_5(1)$ below is empty as the sets $\ms M_{1}(k)$ are
singletons. Conditions $\mc P_6(1)$, $\mc P_7(1)$ follow from the
definition of the height $\mf h_1$ and the jump rates $R_1$.
Postulate $\mc P_5(q)$ states that the energetic barrier between the
elements of a set $\ms M_{q}(k)$ is bounded by $\mf h_{q-1}$, and
postulate $\mc P_6(q)$ that the energetic barrier between the set
$\ms M_{q}(k)$ and its neighbours is larger than $\mf h_q$. Postulate
$\mc P_7(q)$ asserts that the Markov chain $\bb X_q(\cdot)$ jumps only
to nearest neighbour sets, and that the jump is possible only in the
energetic barrier is equal to $\mf h_q$.

\begin{enumerate}

\item[[$\mc P_5(q)$\!\!\!]]  For every $k\in \bb Z$ for which
$\ms M_{q}(k)$ has two or more elements, and for every
$m',m''\in \ms M_{q}(k)$
\begin{equation}
\label{48}
\Lambda (\{m'\},\{m''\}) - S(\ms M_{q}(k))\, \le\, \mf h_{q-1} \;.
\end{equation}

\item[[$\mc P_6(q)$\!\!\!]] For every $k\in \bb Z$,
$\displaystyle \Lambda (\ms M_{q}(k),\ms M_{q}(k\pm
1))-S(\ms M_{q}(k)) \ge \mf h_{q}.$

\item[[$\mc P_7(q)$\!\!\!]] For every $k$, $l \in \bb Z$,
$R_q(\ms M_{q}(k),\ms M_{q}(l))>0$ if, and only if, $l = k \pm 1$
and
$\displaystyle \Lambda (\ms M_{q}(k), \ms M_{q}(k\pm 1))-S(\ms
M_{q}(k)) = \mf h_{q}$.

\end{enumerate}

We prove below that the next two conditions are fulfilled for $q=1$.

\begin{enumerate}

\item[[$\mc P_8(q)$\!\!\!]] Let $a=\pm 1$. If
$R_q(\ms M_{q}(k),\ms M_{q}(k+a))>0$ for some $k\in \bb Z$, then
\begin{equation*}
S(\ms M_{q}(k))\geq S(\ms M_{q}(k+a)).
\end{equation*}
Moreover, if $R_q(\ms M_{q}(k+a),\ms M_{q}(k)) = 0$, then the
previous inequality is strict.

\item[[$\mc P_9(q)$\!\!\!]]
The jump rates are $\mf u_q$-periodic:
$R_q(\ms M_{q}(k+ \mf u_q),\ms M_{q}(k+ \mf u_q \pm 1)) = R_q(\ms
M_{q}(k),\ms M_{q}(k\pm 1))$ for all $k\in \bb Z$.

\end{enumerate}

Condition $\mc P_8(1)$ is easy to check. Assume that
$R_{1}(\ms M_{1}(k),\ms M_{1}(k + 1))>0$ for some $k\in \bb Z$. The
same argument applies to $k-1$. By \eqref{f01}, \eqref{73},
$S(m_k) = S(\sigma_{k}) - h^+_k = S(\sigma_{k}) - \mf h_1$.  On the
other hand, by the definition \eqref{69b} of $\mf h_1$ and \eqref{f01},
$S(\sigma_{k}) - \mf h_1 \ge S(\sigma_{k}) - h^-_{k+1} = S(m_{k+1})$,
so that $S(m_k) \ge S(m_{k+1})$. If
$R_{1}(\ms M_{1}(k + 1),\ms M_{1}(k)) = 0$, then by \eqref{73},
$h^-_{k+1} > \mf h_1$ and the previous inequality is strict, which
proves the postulate $\mc P_8(1)$.

Postulate $\mc P_9(1)$ follows from the definition \eqref{73} of the
jump rates $R_1$ and from the $1$-periodicity of the functions
$\mss a(\cdot)$, $\mss b(\cdot)$.

Denote by $\color{blue} \mf n_q$ the number of $\bb X_q$-closed
irreducible classes, up to equivalence. The proof of postulate $\mc
P_{10}(1)$ is presented in Appendix \ref{sec4}. 

\begin{enumerate}

\item[[$\mc P_{10}(q)$\!\!\!]] Let $\mf n_0=N$, $\mf n_{q} < \mf n_{q-1}$. 

\end{enumerate}

\subsection*{Recursive procedure}

We may now start the recursive procedure, see Figure \ref{fig-1f}
below. Fix $p\geq 1$ and assume that the construction of the triplets
$\Gamma_{q}$, $1\le q\le p$, defined in \eqref{meta_trip}, has been
carried out, and that each $\Gamma_q$ satisfy the properties
$\mc P_1 (q)$ -- $\mc P_{10}(q)$ introduced above. In this subsection,
we present a recursive scheme which provides a triplet $\Gamma_{p+1} $
satisfying the same properties.

Denote by $\color{blue} \mf R_p(k)$, $k\in K$, where $\color{blue} K$
is a subset of $\bb N$, the closed irreducible classes of the Markov
chain $\bb X_{p}(\cdot)$, and by $\color{blue} \mf T_p$ the transient
states.  If $K$ is empty or a singleton, the construction is
completed.  Otherwise, one additional layer is formed.

If $K$ is empty, $\mf n_p=0$. If $K$ is a singleton, that is, if there
is only one closed irreducible class, by postulate $\mc P_7(p)$, this
irreducible class has to be connected.  By postulate $\ms P_9(p)$, it
cannot be bounded or semi-infinite. Hence, if there is only one
irreducibles class, it must be the all set $\ms S_p$. In this case, we
say that $\mf n_p=0$.

Assume, from now, that $K$ has at least two elements. By postulate
$\ms P_9(p)$, if $\bb X_{p}(\cdot)$ has two distinct closed
irreducibles classes, it has a countable number of them, and $\mf
n_p\ge 1$. Hence, the construction ends if $\mf n_p=0$ and it
continues if $\mf n_p\ge 1$.

Since $\bb X_{p}(\cdot)$ takes value in $\ms S_p$, there exist
disjoint subsets $I_k$ of $\bb Z$ such that
$\mf R_p(k) = \{\ms M_p(i) : i \in I_k\}$. Let
\begin{equation}
\label{24}
{\color{blue}\ms M^*_{p+1}(k)} \,:=\, \bigcup_{i\colon \ms M_p(i)  \in \mf
R_p(k)} \ms M_p(i) \;, \quad
{\color{blue}\ms M_{p+1}} \,:=\, \bigcup_{k\in K} \ms M^*_{p+1}(k)
\;.
\end{equation}
Hence, $\ms M_{p+1}$ contains all the $S$-local minima $m$ belonging
to some $\bb X_p$-recurrent state $\ms M_{p}(k)$.

Let ${\color{blue}\ms S_{p+1}} := \{ \ms M^*_{p+1}(k) : k\in K\}$.
The number of equivalent classes of $\ms S_{p+1}$, denoted by
$\mf u_{p+1}$, is, by construction, the number of $\bb X_p$-closed
irreducible classes, up to equivalence, which is represented by
$\mf n_p$. Thus $\mf u_{p+1} = \mf n_p\ge 1$.

\medskip\noindent ($\alpha$){\it \ Properties of the sets
$\ms M^*_{p+1}(i)$}.  By postulate $\mc P_1(p)$ and by definition of
the sets $\ms M^*_{p+1}(j)$, $j\in K$, $\ms M_{p+1}$, each set
$\ms M^*_{p+1}(j)$ is not empty, and $\ms M_{p+1} \subset \mc M$. 

By postulates $\mc P_2(p)$ and $\mc P_7(p)$, the sets
$\ms M^*_{p+1}(k)$ are ordered:
\begin{equation}
\label{78b}
\text{For $j\neq k \in K$ either}\;\;
\ms M^*_{p+1}(j) \,<\, \ms M^*_{p+1}(k) \;\;\text{or}\;\;
\ms M^*_{p+1}(k) \,<\, \ms M^*_{p+1}(j)\;.
\end{equation}
By postulates $\mc P_3(p)$, $\mc P_9(p)$, if $\mf R_p(k)$ is a
$\bb X_p$-recurrent class, then so is $1+\mf R_p(k)$. More precisely,
\begin{equation}
\label{75b}
\text{if $\{\ms M_p(j) : j\in I\}$ is a $\bb X_p$-recurrent
class for some set $I\subset K$, then $\{1+\ms M_p(j) : j\in I\}$}
\end{equation}
is also a $\bb X_p$-recurrent class.

\smallskip\noindent {\sl Claim A}: For each set
$\ms M^*_{p+1}(k)$, $k\in K$, there exist sets $\ms M^*_{p+1}(j)$,
$\ms M^*_{p+1}(j')$, $j$, $j'\in K$, such that
$\ms M^*_{p+1}(j) \,<\, \ms M^*_{p+1}(k) \,<\, \ms M^*_{p+1}(j')$.  In
particular, the set $K$ is countably infinite, and each element
$\ms M^*_{p+1}(k)$ of $\ms S_{p+1}$ is bounded.  \smallskip

Indeed, fix a set $\ms M^*_{p+1}(k)$, $k\in K$.  Since
$\mf u_{p+1}\ge 2$, the set $\ms S_{p+1}$ contains at least two
elements. Let $\ms M^*_{p+1}(j)$, $j\neq k$, be another element of
$\ms S_{p+1}$. In view of \eqref{78b}, assume without loss of
generality that $\ms M^*_{p+1}(j) < \ms M^*_{p+1}(k)$. Choose $\ell$
large enough so that $\ms M^*_{p+1}(k) \not> \ell + \ms M^*_{p+1}(j)$
[this means that $\ell + \ms M^*_{p+1}(j)$ contains an element larger
than an element of $\ms M^*_{p+1}(k)$]. By \eqref{75b},
$\ell + \ms M^*_{p+1}(j) = \ms M^*_{p+1}(j')$ for some $j'\in K$.
Since the sets are ordered, and
$\ms M^*_{p+1}(k) \not> \ms M^*_{p+1}(j')$,
$\ms M^*_{p+1}(k) < \ms M^*_{p+1}(j')$, so that
$\ms M^*_{p+1}(j) < \ms M^*_{p+1}(k) < \ms M^*_{p+1}(j')$, as claimed.

\smallskip\noindent {\sl Claim B}: $\ms M_{p+1} \cap \{m_0,
\dots, m_{N-1}\} \not = \varnothing$. 
\smallskip

Since, by hypothesis, the set
$\ms S_{p+1} = \{ \ms M^*_{p+1}(k) : k\in K\}$ has two or more
equivalent classes, there exists a non-empty set $\ms M^*_{p+1}(k)$,
$k\in K$. Let $j\in \bb Z$ such that
$[\,\ms M^*_{p+1}(k) + j\,] \cap [0,1) \neq \varnothing$.  By
\eqref{75b}, $\ms M^*_{p+1}(k) +j = \ms M^*_{p+1}(k')$ for some
$k'\in K$. Thus, $\ms M^*_{p+1}(k') \cap [0,1) \neq
\varnothing$. Since $\ms M^*_{p+1}(k') \subset \ms M_{p+1}$ and, as
already proved, $\ms M_{p+1} \subset \mc M$,
$\ms M_{p+1} \cap \{m_0, \dots, m_{N-1}\} \not = \varnothing$, as
claimed.

\medskip\noindent ($\beta$){\it \ Reordering the sets
$\ms M^*_{p+1}(i)$}.  Recall the definition of $\mss j_{p+1}$
introduced \eqref{24b}.  By Claim B,
$\mss j_{p+1} \in \{0, \dots, N-1\}$.  Denote by $\ms M_{p+1}(0)$ the
sets $\ms M^*_{p+1}(k)$, $k\in K$, which contains $m_{\mss
j_{p+1}}$. By \eqref{78b} and Claim A, we may relabel the
sets $\ms M^*_{p+1}(k)$, $k\in K$, to obtain an ordered family
$\ms M_{p+1}(i)$, $i\in \bb Z$ satisfying postulate $\mc
P_2(p+1)$. This family is indexed by $\bb Z$ because, by Claim
A, for each element there is a smaller and larger one. 

Since, as observed in Step ($\alpha)$, each set $\ms M^*_{p+1}(k)$,
$k\in K$, is not empty, the same property holds for the sets
$\ms M_{p+1}(i)$, $i\in \bb Z$. Moreover,
$\ms M_{p+1}(i) \subset \ms M_{p+1} \subset \mc M$, showing that
postulate $\mc P_1(p+1)$ holds.  Hence, up to this point, we proved
that postulates $\mc P_1(p+1)$, $\mc P_2(p+1)$ are in force.  The next
result, as well as the following two, is proved in Appendix
\ref{sec4}.

\begin{proposition}
\label{l12b}
Fix $p\geq 1$ and assume that the construction of the triplets
$\Gamma_{q}$, $1\le q\le p$, defined in \eqref{meta_trip}, has been
carried out, and that each $\Gamma_q$ satisfy the properties
$\mc P_1 (q)$ --- $\mc P_{10}(q)$. Assume, furthermore, that
$\mf u_{p+1}\ge 2$. Then, postulates $\mc P_1(p+1)$---$\mc P_4(p+1)$
are in force. 
\end{proposition}

\medskip\noindent ($\gamma$){\it \ Construction of the \emph{height}
$\mf h_{p+1}$}.  By Proposition \ref{l12b} and postulates
$\mc P_2(p+1)$, $\mc P_4(p+1)$, the sets $\ms M_{p+1}(k)$,
$k\in\bb Z$, are ordered and the elements of each set have the same
depth. We may therefore define the heights of the left and right
escape barriers of the set $\ms M_{p+1}(k)$ as
\begin{equation}
\label{47b}
\begin{gathered}
{\color{blue} h^{p+1,\pm}_{k} }  \,:=\,
\Lambda\, \big (\ms M_{p+1}(k\pm 1) \,,\,
\ms M_{p+1}(k)\big)-S(\ms M_{p+1}(k))\;, \quad
k\in \bb Z \;.
\end{gathered}
\end{equation}
Define the height $h^{p+1}_{k}$ of the escape barrier of
$\ms M_{p+1}(k)$ as the minimum of its left and right heights, and let
$\mf h_{p+1}$ be the smallest $(p+1)$-order height:
\begin{equation}
\label{38b}
{\color{blue} h^{p+1}_{k}} := h^{p+1,-}_{k}\wedge h^{p+1,+}_{k} \;, \quad 
{\color{blue}\mf h_{p+1}} := \min
\big\{ \, h^{p+1}_{k} : k\in \bb Z\, \big\}\;.
\end{equation}
By postulates $\mc P_3(p+1)$, and the $1$-periodicity of the functions
$\mss a(\cdot)$, $\mss b(\cdot)$
\begin{equation}
\label{77}
h^{p+1}_{k+\mf u_{p+1}} \,=\, h^{p+1}_{k}\;\;\text{for all}\;\;
k\,\in\, \bb Z\;.
\end{equation}
Hence, the heights $h^{p+1}_{k}$, $k\in\bb Z$, take at most
$\mf u_{p+1}$ different values.  Since $\mf u_{p+1} = \mf n_p$, and,
by postulate $\mc P_{10}(p)$, $\mf n_p < N$,
$\mf h^{p+1} \in (0,\infty)$, and $\mf h^{p+1} = h^{p+1}_{k}$ for some
$k\in \bb Z$.

\begin{proposition}
\label{l11b}
Under the hypotheses of Proposition \ref{l12b}, postulates
$\mc P_5(p+1)$, $\mc P_6(p+1)$ are in force. Moreover,
$\mf h_{p+1}>\mf h_p$.
\end{proposition}

\medskip\noindent ($\delta$){\it \ Construction of the Markov chain
$\bb X_{(p+1)}(\cdot)$}.  Denote by $m^+_{p+1,i}$, $m^-_{p+1,i}$,
$i\in \bb Z$, the rightmost, leftmost element of $\ms M_{p+1}(i)$,
respectively:
\begin{equation}
\label{49b}
{\color{blue} m^-_{p+1,i}}
\,:=\,  \min\{m\in \ms M_{p+1}(i)\}\;, \hspace{5mm}
{\color{blue} m^+_{p+1,i}}
\,:=\,  \max\{m\in \ms M_{p+1}(i)\}.
\end{equation}
Let $\ms W^{(p+1)}_{k,k+1}$ be the set of the absolute maximum of
$S$ on the intervals $[m^+_{p+1, k}, m^-_{p+1, k+1}]$:
\begin{equation}\label{max_br}
{\color{blue}\ms W^{(p+1)}_{k,k+1} } \; := \;
\operatorname*{arg\,max}_{x\;\in\;
[m^+_{p+1, k} \,,\,  m^-_{p+1, k+1}]} S(x) \;.
\end{equation}
Clearly, the elements of $\ms W^{(p+1)}_{k,k+1}$ have the same
height and
\begin{equation*}
S(\ms W^{(p+1)}_{k,k+1}) \,=\,  \Lambda
\big(\ms M_{p+1}(k), \ms M_{p+1}(k+1)\big) \;.
\end{equation*}

Let the weight $\pi_{p+1}(k)$ of the set of local minima
$\ms M_{p+1}(k)$ and the weight $\sigma_{p+1}(k,k+1)$ of the set of
maxima $\ms W^{(p+1)}_{k,k+1}$, $k\in \bb Z$, be defined as
\begin{equation}
\label{01}
{\color{blue}\pi_{p+1}(k)} \,:=\,
\sum_{m\in \ms M_{p+1}(k)}
\frac{1}{\mss a (m)}
\sqrt{\frac{2\pi}{S''(m)}}\,, \quad
{\color{blue}\sigma_{p+1}(k,k+1)} \,:= \,
\sum_{ \sigma\in \ms W^{(p+1)}_{k,k+1}}
\sqrt{\frac{2\pi}{-\, S''(\sigma)}}\;\cdot
\end{equation}
Define the $(p+1)$-th-level jump rates as
\begin{equation}
\label{50}
\begin{gathered}
{\color{blue}R_{p+1}(\ms M_{p+1}(k),\ms M_{p+1}(k+1))} \; :=
\; \frac{1}{\pi_{p+1}(k) \, \sigma_{p+1}(k,k+1)}\;
\mb 1\{\, h^{p+1,+}_{k} \,=\,  \mf h_{p+1} \,\} \;,
\\
{\color{blue}R_{p+1}(\ms M_{p+1}(k),\ms M_{p+1}(k-1))} \; :=
\; \frac{1}{\pi_{p+1}(k)\, \sigma_{p+1}(k-1,k)}
\; \mb 1\{\,  h^{p+1,-}_{k} \,=\,  \mf h_{p+1} \,\} \;.
\end{gathered}
\end{equation}
Let ${\color{blue}\bb X_{p+1}(\cdot)}$ be the $\ms S_{p+1}-$valued
Markov chain induced by the generator $\bb L_{p+1}$ given by
\begin{equation}
\label{83}
{\color{blue}(\bb L_{p+1}f)\, (\ms M_{p+1}(k))}
\;=\; \sum_{a = \pm 1} R_{p+1}(\ms M_{p+1}(k), \ms
M_{p+1}(k+ a)) \,
[\, f(\ms M_{p+1}(k+a)) - f(\ms M_{p+1}(k)) \,]\;,
\end{equation}
for any $f\colon \ms S_{p+1} \to \bb R$ and
$\ms M_{p+1}(k)\in \ms S_{p+1}$. 

\begin{proposition}
\label{l10b}
Under the hypotheses of Proposition \ref{l12b}, postulates
$\mc P_7(p+1)$---$\mc P_{10}(p+1)$ are in force. Moreover,
there exists $k\in \bb Z$ such that
\begin{equation}
\label{79}
R_{p+1}(\ms M_{p+1}(k),\ms M_{p+1}(k+1))
\,+\, R_{p+1}(\ms M_{p+1}(k),\ms M_{p+1}(k-1)) \,>\, 0\;.
\end{equation}
\end{proposition}

\noindent{\it Conclusion:}
According to Propositions \ref{l12b}--\ref{l10b}, if triplets
$\Gamma_1, \dots, \Gamma_p$, $p\ge 1$, have been constructed
fulfilling conditions $\mc P_1$---$\mc P_{10}$, and if
$\mf u_{p+1} = \mf n_{p}\ge 2$, then a new triplet $\Gamma_{p+1}$ can
be constructed fulfilling conditions $\mc
P_1(p+1)$---$\mc P_{10}(p+1)$. Moreover, $\mf u_1=N$, the sequence
$(\mf u_q :1\le q\le p+1)$ is strictly decreasing, and the sequence
$(\mf h_q :1\le q\le p+1)$ is strictly increasing.

Since the sequence $\mf u_q$ starts from $N$ and strictly decreases,
there exists $p \ge 1$ such that $\mf u_{p+1} = 0$ (ore, equivalently,
$\mf n_p=0$). Let
${\color{blue} \mf q} := \min\{ p\ge 1 : \mf u_{p+1} =0\}$. The
iterative construction ends with the triplet $\Gamma_{\mf q}$ and the
recursive method produced the triplets $\Gamma_r$, $1\le r\le \mf q$.
The last Markov chain $\bb X_{\mf q}(\cdot)$ has only transient states
or $\ms S_{\mf q}$ is a closed irreducible class. We describe further
this last layer in the next subsection.

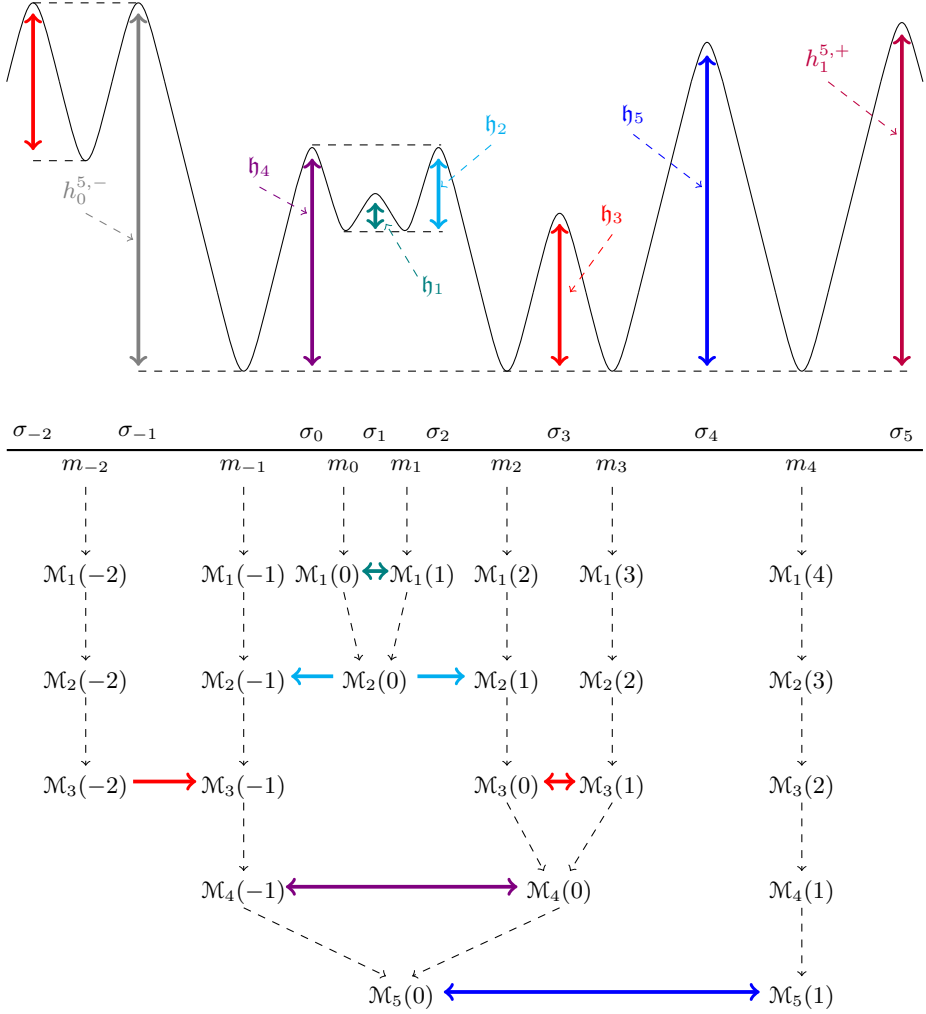
\begin{figure}
\centering
\begin{tikzpicture}[scale=0.7]
\draw [rounded corners] (0,3) .. controls (0.5,5) .. (1,3);
\draw[rounded corners] (1,3) .. controls (1.5,1) .. (2,3);
\draw[rounded corners] (2,3) .. controls (2.5,5) .. (3,3);  
\draw[rounded corners] (3,3) .. controls (3.5, 1) .. (4, -1);
\draw[rounded corners] (4,-1) .. controls (4.5, -3) .. (5,-1);
\draw[rounded corners] (5,-1) -- (5.5,1);
\draw[rounded corners] (5.5,1) .. controls (5.8,2) .. (6.1, 1);
\draw[rounded corners] (6.1, 1) .. controls (6.4, 0) .. (6.7, 0.5);
\draw[rounded corners] (6.7, 0.5) .. controls (7, 1) .. (7.3, 0.5);
\draw[rounded corners] (7.3, 0.5) .. controls (7.6, 0) .. (7.9, 1);
\draw[rounded corners] (7.9, 1) .. controls (8.2, 2) .. (8.5, 1);
\draw[rounded corners] (8.5, 1) -- (9, -1);
\draw[rounded corners] (9,-1) .. controls (9.5, -3) .. (10, -1);
\draw[rounded corners] (10,-1) .. controls (10.5, 1) .. (11, -1);
\draw[rounded corners] (11,-1) .. controls (11.5,-3) .. (12, -1);
\draw[rounded corners] (12,-1) .. controls (12.5, 1) .. (13, 3);
\draw[rounded corners] (13,3) .. controls (13.3, 4) .. (13.6, 3);
\draw[rounded corners] (13.6,3) .. controls (14.1,1) .. (14.6,-1);
\draw[rounded corners] (14.6,-1) .. controls (15.1,-3).. (15.6,-1);
\draw[rounded corners] (15.6,-1) .. controls (16.1,1) .. (16.6,3);
\draw[rounded corners](16.6,3) .. controls (17,4.5) .. (17.4,3);

\draw[to-to, line width=0.5mm, blue](13.3,3.5)--(13.3,-2.4);
%\fill(13.6,1)node[below,blue, font=\tiny]{$\mf h_5$};
\draw[to-to, line width = 0.5mm, teal](7,0.7)--(7,0.2);
%\fill(7.2,0.5)node[below,font=\tiny, teal]{$\mf h_1$};
\draw[to-to, line width = 0.5mm, cyan](8.2,1.54)--(8.2,0.2);
%\fill(8.4,1)node[below,cyan, font=\tiny]{$\mf h_2$};
\draw[to-to, line width=0.5mm, violet](5.8,1.54)--(5.8,-2.4);
%\fill(6,-1)node[below,violet,font=\tiny]{$\mf h_4$};
\draw[to-to, line width=0.5mm, red](10.5,0.3)--(10.5,-2.4);
%\fill(10.7,-1)node[below,red, font=\tiny]{$\mf h_3$};
\draw[to-to, line width = 0.5mm, gray](2.5,4.3)--(2.5,-2.4);
%\fill(2.9,1)node[below,gray,font=\tiny]{$h_0^{5,-}$};
\draw[to-to, line width=0.5mm, purple](17,3.9) -- (17,-2.4);

\draw[to-to,line width=0.5mm, red](0.5,4.3)--(0.5,1.7);

\draw[-to, dashed, gray](1.5,0.5)--(2.4,0); \fill(1.5,0.5)node[above,
gray, font=\small]{$h_0^{5,-}$}; 

\draw[-to, dashed, violet](4.8,1)--(5.7,0.5);
\fill(4.8,1)node[above,violet,font=\small]{$\mf h_4$};

\draw[-to, dashed, teal](7.8,-0.7)--(7.2,0.3);
\fill(8.1,-0.5)node[below, teal, font=\small]{$\mf h_1$};

\draw[-to, dashed, cyan](9.2,1.8)--(8.25,1);
\fill(9.3,1.8)node[above, cyan, font=\small]{$\mf h_2$};

\draw[-to, dashed, red](11.5,0.1)--(10.7,-1);
\fill(11.5,0.1)node[above,red,font=\small]{$\mf h_3$};

\draw[-to, dashed, blue](11.9,2)--(13.2,1);
\fill(11.9,2)node[above, blue, font=\small]{$\mf h_5$};

\draw[-to, dashed, purple](15.6,3)--(16.9,2);
\fill(15.6,3)node[above,purple,font=\small]{$h_1^{5,+}$};

\draw[dashed,black](0.5,4.5)--(2.5,4.5);
\draw[dashed, black](5.8,1.8)--(8.25,1.8);
\draw[dashed, black](2.5,-2.5)--(17.1,-2.5);
\draw[dashed, black](0.5,1.5)--(1.5,1.5);
\draw[dashed, black](6.4,0.15)--(8.28,0.15);

\draw[solid, thick, black](0,-4)--(17.4,-4);

\fill(6.4,-4)node[below, font=\small]{$m_0$};
\fill(7.6,-4)node[below, font=\small]{$m_1$};
\fill(9.5,-4)node[below, font=\small]{$m_2$};
\fill(11.5,-4)node[below, font=\small]{$m_3$};
\fill(15.1,-4)node[below, font=\small]{$m_4$};
\fill(4.5,-4)node[below, font=\small]{$m_{-1}$};
\fill(1.5,-4)node[below, font=\small]{$m_{-2}$};

\fill(0.5,-4)node[above, font=\small]{$\sigma_{-2}$};
\fill(2.5,-4)node[above, font=\small]{$\sigma_{-1}$};
\fill(5.8,-4)node[above, font=\small]{$\sigma_0$};
\fill(7,-4)node[above, font=\small]{$\sigma_1$};
\fill(8.2,-4)node[above, font=\small]{$\sigma_2$};
\fill(10.5,-4)node[above, font=\small]{$\sigma_3$};
\fill(13.3,-4)node[above, font=\small]{$\sigma_4$};
\fill(17,-4)node[above, font=\small]{$\sigma_5$};

\draw[-to, dashed](6.4,-4.7)--(6.4,-6);
\fill(6.1,-6)node[below, font=\small]{$\ms M_1(0)$};
\draw[-to, dashed](7.6,-4.7)--(7.6,-6);
\fill(7.9,-6)node[below, font=\small]{$\ms M_1(1)$};

\draw[to-to, teal, line width=0.5mm] (6.75,-6.3) -- (7.25,-6.3); 

\draw[-to, dashed](1.5,-4.7)--(1.5,-6);
\fill(1.5,-6)node[below, font=\small]{$\ms M_1(-2)$};
\draw[-to, dashed](4.5,-4.7)--(4.5,-6);
\fill(4.5,-6)node[below, font=\small]{$\ms M_1(-1)$};
\draw[-to, dashed](9.5,-4.7)--(9.5,-6);
\fill(9.5,-6)node[below, font=\small]{$\ms M_1(2)$};
\draw[-to, dashed](11.5,-4.7)--(11.5,-6);
\fill(11.5,-6)node[below, font=\small]{$\ms M_1(3)$};
\draw[-to, dashed](15.1,-4.7)--(15.1,-6);
\fill(15.1,-6)node[below, font=\small]{$\ms M_1(4)$};

\draw[-to, dashed](6.4,-6.7)--(6.7,-8);
\draw[-to, dashed](7.6,-6.7)--(7.3,-8);
\fill(7,-8)node[below, font=\small]{$\ms M_2(0)$};

\draw[-to, cyan, line width=0.5mm](7.8,-8.3)--(8.7,-8.3);
\draw[-to, cyan, line width=0.5mm](6.2,-8.3)--(5.4,-8.3);

\draw[-to, dashed](1.5,-6.7)--(1.5,-8);
\fill(1.5,-8)node[below, font=\small]{$\ms M_2(-2)$};
\draw[-to, dashed](4.5,-6.7)--(4.5,-8);
\fill(4.5,-8)node[below, font=\small]{$\ms M_2(-1)$};
\draw[-to, dashed](9.5,-6.7)--(9.5,-8);
\fill(9.5,-8)node[below, font=\small]{$\ms M_2(1)$};
\draw[-to, dashed](11.5,-6.7)--(11.5,-8);
\fill(11.5,-8)node[below, font=\small]{$\ms M_2(2)$};
\draw[-to, dashed](15.1,-6.7)--(15.1,-8);
\fill(15.1,-8)node[below, font=\small]{$\ms M_2(3)$};

\draw[-to, dashed](1.5,-8.7)--(1.5,-10);
\fill(1.5,-10)node[below, font=\small]{$\ms M_3(-2)$};

\draw[-to, red, line width=0.5mm](2.4,-10.3)--(3.6,-10.3);

\draw[-to, dashed](4.5,-8.7)--(4.5,-10);
\fill(4.5,-10)node[below, font=\small]{$\ms M_3(-1)$};
\draw[-to, dashed](9.5,-8.7)--(9.5,-10);
\fill(9.5,-10)node[below,font=\small]{$\ms M_3(0)$};

\draw[to-to, red, line width=0.5mm](10.2,-10.3)--(10.8,-10.3);

\draw[-to, dashed](11.5, -8.7)--(11.5,-10);
\fill(11.5,-10)node[below, font=\small]{$\ms M_3(1)$};
\draw[-to, dashed](15.1,-8.7)--(15.1,-10);
\fill(15.1,-10)node[below, font=\small]{$\ms M_3(2)$};

\draw[-to, dashed](4.5,-10.7)--(4.5,-12);
\fill(4.5,-12)node[below, font=\small]{$\ms M_4(-1)$};
\draw[-to, dashed](9.5,-10.7)--(10.2,-12);
\draw[-to, dashed](11.5,-10.7)--(10.7,-12);
\fill(10.5,-12)node[below, font=\small]{$\ms M_4(0)$};

\draw[to-to, violet, line width=0.5mm](5.3,-12.3)--(9.7,-12.3);

\draw[-to, dashed](15.1,-10.7)--(15.1,-12);
\fill(15.1,-12)node[below, font=\small]{$\ms M_4(1)$};

\draw[-to, dashed](4.5,-12.7)--(7.2,-14);
\draw[-to,dashed](10.5,-12.7)--(7.7,-14);
\fill(7.5,-14)node[below,font=\small]{$\ms M_5(0)$};
\draw[-to, dashed](15.1,-12.7)--(15.1,-14);
\fill(15.1,-14)node[below,font=\small]{$\ms M_5(1)$};
\draw[to-to, blue, line width=0.5mm](8.3,-14.3)--(14.3,-14.3);

\end{tikzpicture}
\caption{Here we represent a hierarchy structure with 5 different
levels. For the level one, the states $\ms M_1(k)$, $-2\leq k\leq 4$
are given by $\ms M_1(k) = \{m_k\}$. Moreover, $\ms M_1(0)$ and
$\ms M_1(1)$ belong to the same $\bb X_1(\cdot)-$recurrent class,
while the remaining states that appear in the figure are
absorbing. For the level 2, we have $\ms M_2(0) = \{m_0, m_1\}$ and
$\ms M_2(k) = \{m_k\}$, for $k\in \{-2,-1,2,3,4\}$. At this level,
$\ms M_2(0)$ is a $\bb X_2(\cdot)-$transient state, while the
remaining states are absorbing. Turning to level 3, we see the first
relabeling. Here $\mss j_3=2$ and the metastable states of level 3
that appear in the figure are given by $\ms M_3(-2)=\{m_{-2}\}$,
$\ms M_3(-1)=\{m_{-1}\}$, $\ms M_3(0) = \{m_2\}$, $\ms M_3(1)=\{m_3\}$
and $\ms M_3(2)=\{m_4\}$. At this level, $\ms M_3(-2)$ is
$\bb X_3(\cdot)-$transient, $\ms M_3(-1)$ and $\ms M_3(2)$ are
absorbing, and $\{\ms M_3(0), \ms M_3(1)\}$ is a recurrent class. At
level 4 we have the states $\ms M_4(-1) = \{m_{-1}\}$,
$\ms M_4(0) = \{m_2, m_3\}$ and $\ms M_4(1) = \{m_4\}$. Here
$\mss j_4 = 2$. At this level, the states $\ms M_4(-1)$ and
$\ms M_4(1)$ are absorbing, while the class
$\{\ms M_4(0), \ms M_4(1)\}$ is $\bb X_4(\cdot)-$recurrent. Finally,
at the level 5, the only states appearing in the figure are
$\ms M_5(0) = \{m_{-1},\; m_2,\; m_3\}$ and $\ms M_5(1) = \{m_4\}$,
while $\mss j_5=2$. Finally, as $\mf h_5 < h_0^{5,-}$, then
$R_5(\ms M_5(0), \ms M_5(1))>0,$ while
$R_5(\ms M_5(0), \ms M_)5(-1)) = 0$. Similarly, as
$\mf h_5<h_1^{5,+}$, then $R_5(\ms M_5(1), \ms M_5(0))>0,$ while
$R_5(\ms M_5(1), \ms M_5(2)) = 0$.}
\label{fig-1f}
\end{figure}

\subsection*{The parabolic equation}

The statement of the second main result of this article requires some
more notation. For $m\in\mathcal{M}$ and finite subsets $\ms M'$ of
$\mc M$, let the weights $\pi(m)$, $\pi_{p+1}(\ms M')$ be given by
\begin{equation}
\label{104}
{\color{blue}\pi(m_k)} \,:=\, \pi_1(k)\;, \quad
{\color{blue}\pi (\ms M')} \, :=\,
\sum_{m\in \ms M'}\pi(m)\;.
\end{equation}
Set also, for $m\in\mc M$, ${\color{blue}\pi(m)} := \pi(m_k)$ if $m=m_k$.

For $2\le p\le \mf q$, $x\in \bb R$, let $\mtt l_p(x)$, $\mtt r_p(x)$
the indices of the first set $\ms M_p(j)$ to the left, right of $x$,
respectively:
\begin{equation}
\label{85}
{\color{blue} \mtt l_p(x)} \,:=\, 
\max\{ j\in \bb Z : \exists\; m\in \ms M_p(j) \,,\,
m \le  x\,\}\;, \quad
{\color{blue}  \mtt r_p(x)} 
\,:=\, \min\{ k\in \bb Z : \exists\; m\in \ms M_p(k) \,,\,
m \ge x\,\}\;.
\end{equation}
Mind that either $\mtt r_p(x) = \mtt l_p(x)$ or
$\mtt r_p(x) = \mtt l_p(x) +1$.  

Suppose that $\mtt l_p(x) < \mtt r_p(x)$. In this case,
$m^+_{p,\mtt l_p(x)} < x < m^-_{p,\mtt r_p(x)}$, where the local
minima $m^\pm_{p,k}$ have been introduced in \eqref{49b}.  Let
$m_{\mtt l} = m^+_{p,\mtt l_p(x)}$, $m_{\mtt r} = m^-_{p,\mtt r_p(x)}$
to simplify notation.
Denote by $\mf S_p(x)$ the finite set where $S(\cdot)$ attains its global
maximum in the interval $[m_{\mtt l} , m_{\mtt r}]$:
\begin{equation*}
{\color{blue} \mf S_p(x)}  \,:=\, \big\{\sigma \in (m_{\mtt l} , m_{\mtt r})
\cap \mc W : S(\sigma) = \max_{y\in
[m_{\mtt l} , m_{\mtt r}]} S(y) \big\}\;,
\end{equation*}
and let
${\color{blue} \mf S^-_p(x)} : = \{\sigma\in \mf S_p(x) : \sigma <
x\}$. Mind that $\mf S^-_p(x)$ may be empty.

Let $\mtt h_p (x, \cdot)$, $x\in \bb R$, be the
probability measure on $\ms S_p$ given by
\begin{equation}
\label{74}
\begin{gathered}
{\color{blue} \mtt h_p (x, \ms M_p(k)) }
\,=\, \delta_{k,\mtt l_p(x)} \quad
\text{if $\mtt l_p(x) = \mtt r_p(x)$ }\;,
\\
{\color{blue} \mtt h_p (x, \ms M_p(\mtt r_p(x)) )}  \,:=\, 
1\,-\, \mtt h_p (x, \ms M_p(\mtt l_p(x)) ) \,=\,
\mf w_p(x)
\quad
\text{if $\mtt l_p(x) < \mtt r_p(x)$}\;.
\end{gathered}
\end{equation}
where
\begin{equation}
\label{92}
\begin{gathered}
{\color{blue}\mf w_p(x)}  \,:=\,
\frac{\sum_{\sigma \in \mf S^-_p(x)}  1/\sqrt{- S''(\sigma)}}
{\sum_{\sigma \in \mf S_p(x)} 1/\sqrt{- S''(\sigma)}} \;\;
\text{if}\;\; x\not\in \mf S_p(x)
\\
\text{and}\;\;
{\color{blue}\mf w_p(x)}  \,:=\,
\frac{1/2 \sqrt{- S''(x)}
\,+\, \sum_{\sigma \in \mf S^-_p(x)}  1/\sqrt{- S''(\sigma)}}
{\sum_{\sigma \in \mf S_p(x)} 1/\sqrt{- S''(\sigma)}} \;\;
\text{if}\;\; x \in \mf S_p(x)\;.
\end{gathered}
\end{equation}
We give in Remark \ref{rm1} an interpretation for
$\mtt h_p (x, \cdot)$ in terms of the diffusion on $\bb R$ induced by
the generator $\ms L_\epsilon$.

For $1\le p\le \mf q$, denote by $p_{t}^{(p)}(\cdot,\,\cdot)$, $t>0$,
the transition probability of the Markov chain $\bb X_p(\cdot)$.  Let
\textcolor{blue}{$\theta_{\epsilon}^{(0)}\equiv1$} and
\textcolor{blue}{$\theta_{\epsilon}^{(\mathfrak{q}+1)}\equiv\infty$}
for convenience.

\begin{theorem}
\label{mt2}
Fix a continuous function $u_0\colon \bb R \to \bb R$.  Denote by
$u_\epsilon$ the solution of the parabolic equation \eqref{51}.

\begin{enumerate}
\item[{\rm (a)}] For all $1\le p\le \mf q$, $x\in \bb R$ and $t>0$,
\begin{equation*}
\lim_{\epsilon\to0}\,
u_{\epsilon}(\theta_{\epsilon}^{(p)}t\,,\,x) \ =\
\sum_{k\in\bb Z}\,
\mtt h_p (x , \ms M_p(k))\, \sum_{\ell\in\bb Z}
p_{t}^{(p)}(\ms M_p(k), \ms M_p(\ell))\,
\sum_{m'\in \ms M_p(\ell)}
\frac{\pi(m')}{\pi(\ms M_p(\ell))}
\,u_{0}(m')\ .
\end{equation*}

\item[{\rm (b)}] For all $1\le p< \mf q$, $x\in \bb R$, and sequence
$(\varrho_{\epsilon})_{\epsilon>0}$ such that
$\theta_{\epsilon}^{(p)}\,\prec\,\varrho_{\epsilon}\,\prec\,
\theta_{\epsilon}^{(p+1)}$,
\begin{equation*}
\lim_{\epsilon\to0}\,u_{\epsilon}(\varrho_{\epsilon} \,,\,x)
\ =\ \sum_{k\in\bb Z}
\, \mtt h_{p+1} (x, \ms M_{p+1}(k) )
\sum_{m'\in \ms M_{p+1}(k)}
\frac{\pi(m')}{\pi(\ms M_{p+1}(k))} \,u_{0}(m')\;.
\end{equation*}

\item[{\rm (c)}] Fix $x\in \bb R$, and a sequence
$(\varrho_{\epsilon})_{\epsilon>0}$ such that
$\epsilon^{-1} \prec \varrho_{\epsilon} \prec
\theta_{\epsilon}^{(1)}$.  Denote by $m_{\mtt l}(x)$, $m_{\mtt r}(x)$
the $S$-local minimum to the left, right of $x$, respectively, and by
$\sigma$ the $S$-local maximum between $m_{\mtt l}(x)$ and
$m_{\mtt r}(x)$, Then,
\begin{equation*}
\lim_{\epsilon\to 0} u_{\epsilon}(\varrho_{\epsilon} \,,\,x)  \, =\,
\left\{
\begin{aligned}
& u_0(m_{\mtt l}(x)) \;\;\text{if}\;\; x<\sigma\,,
\\
& u_0(m_{\mtt r}(x)) \;\;\text{if}\;\; x>\sigma\,,
\\
& (1/2) \,[\, u_0(m_{\mtt l}(x)) + u_0(m_{\mtt r}(x))\,]
\;\;\text{if}\;\; x=\sigma\,.
\end{aligned}
\right.
\end{equation*}
\end{enumerate}
\end{theorem}

\begin{remark}
Assertion (a) for $p=1$ is Theorem \ref{mt1}. The claim for
$1\le p\le \mf q$ has to be understood as follows. Starting from $x$, the
set $\ms M_p(k)$ is reached with probability
$\mtt h_p (x , \ms M_p(k))$.  From $\ms M_p(k)$ it moves to
$\ms M_p(\ell)$ at time $t$ with probability
$p_{t}^{(p)}(\ms M_p(k), \ms M_p(\ell))$. Being at $\ms M_p(\ell)$,
its position among the elements of this set is distributed according
to the probability measure $\pi(\cdot)/ \pi(\ms M_p(\ell))$.

Assertion (b) has a similar interpretation. In the time-scale
$\varrho_\epsilon$, starting from $x$, the set $\ms M_p(k)$ is reached
with probability $\mtt h_p (x , \ms M_p(k))$. Being at $\ms M_p(k)$,
its position among the states of this set is distributed according to
the probability measure $\pi(\cdot)/ \pi(\ms M_p(k))$.

It is not difficult to show that the limit, as $t\to\infty$, of item
(a) right-hand side for $1\le p <\mf q$, converges to item (b)
right-hand side with $p+1$.
\end{remark}

\begin{remark}
In part (c) of the theorem, the condition
$\varrho_{\epsilon} \succ \epsilon^{-1}$ is imposed to include the
point $x=\sigma$. If one assumes that $x\neq \sigma$, then it is
enough to require that $\varrho_{\epsilon} \succ 1$.
\end{remark}

\subsection*{The last layer in the hierarchical construction.}

By construction, either all states of the Markov chain
$\bb X_{\mf q}(\cdot)$ are transient, or $\ms S_{\mf q}$ is an
irreducibles class. 

\smallskip
\noindent{\bf Claim 2.A:} Suppose that
$R_{\mf q} (\ms M_{\mf q} (j), \ms M_{\mf q} (j\pm 1)) = 0$ for some
$j\in\bb Z$. Then,
$R_{\mf q} (\ms M_{\mf q} (k), \ms M_{\mf q} (k\mp 1)) > 0$ for all
$k\in\bb Z$.

Suppose, without loss of generality, that
$R_{\mf q} (\ms M_{\mf q} (j), \ms M_{\mf q} (j-1))= 0$ for some
$j\in \bb Z$. In this case,
$R_{\mf q} (\ms M_{\mf q} (j), \ms M_{\mf q} (j+1))> 0$.  Indeed, if
this rate vanishes, as $\bb X_{\mf q}(\cdot)$ jumps only to nearest
neighbour sites, $\ms M_{\mf q} (j)$ would be an absorbing state.
This contradicts the fact that $\bb X_{\mf q}(\cdot)$ has no bounded
irreducibles class.

If $R_{\mf q} (\ms M_{\mf q} (k), \ms M_{\mf q} (k+1)) = 0$ for some
$k>j$, the Markov chain $\bb X_{\mf q}(\cdot)$ is confined to the set
$\{\ms M_{\mf q} (j), \dots, \ms M_{\mf q} (k)\}$. It has therefore, a
closed irreducible class contained in this set. This contradicts once
more the fact that $\bb X_{\mf q}(\cdot)$ has no bounded irreducibles
class.  Hence,
$R_{\mf q} (\ms M_{\mf q} (k), \ms M_{\mf q} (k+1)) > 0$ for all
$k>j$. By postulate $\mc P_9(\mf q)$,
$R_{\mf q} (\ms M_{\mf q} (k'), \ms M_{\mf q} (k'+1))> 0$ for all
$k'\in\bb Z$, as claimed. \smallskip

\smallskip
\noindent{\bf Claim 2.B:} Suppose that $S(1)< S(0)$. Then,
$R_{\mf q} (\ms M_{\mf q} (j), \ms M_{\mf q} (j-1))= 0$ for some
$j\in \bb Z$.

By construction, $\mf u_{\mf q} \ge 1$, and by postulate
$\mc P_3(\mf q)$,
$\ms M_{\mf q} (j + \mf u_{\mf q}) = \ms M_{\mf q} (j) +1$, $j\in \bb Z$.
Then, by postulate $\mc P_4(\mf q)$, which states that $S(\cdot)$ is
constant at each set $\ms M_{\mf q} (k)$,
\begin{equation}
\label{101}
S(\ms M_{\mf q} (\mf u_{\mf q})) \,<\, S(\ms M_{\mf q} (0))\,.
\end{equation}
It follows from this bound that $R_{\mf q} (\ms M_{\mf q} (j),  \ms
M_{\mf q} (j-1))= 0$ for some $0< j \le \mf u_{\mf q}$. Indeed, if all
rates were positive, by postulate $\mc P_8(\mf q)$, $S(\ms M_{\mf q}
(0)) \le S(\ms M_{\mf q} (\mf u_{\mf q}))$, in contradiction with
\eqref{101}. 

\smallskip
\noindent{\bf Claim 2.C:} Suppose that $S(1)= S(0)$. Then,
$R_{\mf q} (\ms M_{\mf q} (j), \ms M_{\mf q} (j\pm 1))> 0$ for all
$j\in \bb Z$. Moreover, the measure $\pi(\ms M_{\mf q} (j))$ satisfies
the detailed balance conditions.

By the argument leading to \eqref{101},
$S(\ms M_{\mf q} (\mf u_{\mf q})) = S(\ms M_{\mf q} (0))$.

Suppose, by contradiction, that
$R_{\mf q} (\ms M_{\mf q} (j), \ms M_{\mf q} (j- 1)) = 0$ for some
$j\in \bb Z$. Then, by Claim 2.A,
$R_{\mf q} (\ms M_{\mf q} (k), \ms M_{\mf q} (k+ 1)) > 0$ for all
$k\in \bb Z$. In particular,
$R_{\mf q} (\ms M_{\mf q} (j-1), \ms M_{\mf q} (j)) > 0$. Hence, by
postulate $\mc P_8(\mf q)$,
$S(\ms M_{\mf q} (j)) < S(\ms M_{\mf q} (j-1))$.

On the other hand, by postulate $\mc P_8(\mf q)$, since
$R_{\mf q} (\ms M_{\mf q} (k), \ms M_{\mf q} (k+ 1)) > 0$ for all
$k\in \bb Z$,
$S(\ms M_{\mf q} (k))$ is a non-increasing sequence. As
$S(\ms M_{\mf q} (\mf u_{\mf q})) = S(\ms M_{\mf q} (0))$, it is
constant. This contradicts the inequality $S(\ms M_{\mf q} (j)) <
S(\ms M_{\mf q} (j-1))$ obtained in the previous paragraph.

The measure $\pi(\ms M_{\mf q} (j))$ satisfies the detailed balance
conditions in view of the definition \eqref{50} of the jump rates
$R_{\mf q}$.
\smallskip

Recall from postulates $\mc P_3(\mf q)$, $\mc P_9(\mf q)$ that
$\ms M_{\mf q} (j+\mf u_{\mf q}) = \ms M_{\mf q} (j) +1$,
$j\in \bb Z$, and that the Markov chain $\bb X_{\mf q} (\cdot)$ jump
rates are $\mf u_{\mf q}$-periodic. It is therefore natural to
consider the evolution of $\bb X_{\mf q} (\cdot)$ among the
equivalence classes.

Fix a positive integer $\ell \ge 1$, and let
${\color{blue} \ms S_{\mf q, \ell} }:= \{\ms M_{\mf q, \ell}(j) : 0\le
j <\ell \mf u_{\mf q} \}$.  Denote by
$\color{blue} \bb X_{\mf q, \ell} (\cdot)$ the Markov chain
$\bb X_{\mf q} (\cdot)$ projected on the sets $\ms S_{\mf q,
\ell}$. Hence, $\bb X_{\mf q, \ell} (\cdot)$ is the
$\ms S_{\mf q, \ell}$-valued Markov chain whose jump rates are given
by \eqref{50}, where the sum is taken modulo $\ell\, \mf u_{\mf q}$.

By Claim 2.C, if $S(0)=S(1)$, the probability measure
$\pi_{\mf q, \ell}$ defined on
$\ms S_{\mf q, \ell}$ by
\begin{equation}
\label{102}
{\color{blue} \pi_{\mf q, \ell} (\ms M_{\mf q} (j))}
\,:=\, \frac{1}{\ell}\, \frac{\pi(\ms M_{\mf q} (j)) }{\sum_{0\le i<
\mf u_{\mf q}} \pi(\ms M_{\mf q} (i))}\;, \quad 0\le j< \ell\, \mf
u_{\mf q}\,.
\end{equation}
is reversible for the Markov chain $\bb X_{\mf q, \ell} (\cdot)$.

If $S(0) \neq S(1)$, there is no explicit formula for the stationary
state, which exists and is unique because the chain
$\bb X_{\mf q, \ell} (\cdot)$ is irreducible. We denote the
stationary state by $\pi_{\mf q, \ell}$, as well.

\begin{theorem}
\label{mt5}
Fix a continuous function $u_0\colon \bb R \to \bb R$, $\ell$-periodic
for some $\ell\ge 1$, and a sequence
$(\varrho_{\epsilon})_{\epsilon>0}$ such that
$\varrho_{\epsilon} \,\succ\, \theta_{\epsilon}^{(\mathfrak{q})}$.
Denote by $u_\epsilon$ the solution of the parabolic equation
\eqref{51}.  Then,
\begin{equation*}
\lim_{\epsilon\to0}\,u_{\epsilon}(\varrho_{\epsilon} , x)
\,=\, \sum_{k=0}^{\ell \mf u_{\mf q}- 1}
\pi_{\mf q, \ell} (\ms M_{\mf q}(k)))\,
\sum_{m \in \ms M_{\mf q} (k)}
\frac{\pi(m)}{\pi(\ms M_{\mf q}(k))}
\, u_{0}(m)
\end{equation*}
for all $x\in \bb R$.
\end{theorem}

\section{Sketch of the proof}
\label{sec1}

The proof of Theorems \ref{mt1} and \ref{mt2} is based on the
stochastic representation of the solutions of the parabolic equation
\eqref{51}. In particular, all quantities appearing in the statements
of these results have a simple stochastic interpretation in terms of a
diffusion process.

Consider the one-dimensional diffusion process 
\begin{equation}
\label{03}
dX_\epsilon (t) \;=\; \mss b(X_\epsilon (t))\, dt \;+\;
\sqrt{2\, \epsilon \, \mss a (X_\epsilon (t))} \,
dW_t\;,
\end{equation}
where $W_t$ is the Brownian motion on $\bb R$.

\begin{remark}
Let $\color{blue} \bb T=[0,1)$ be the one-dimensional torus of length
one, and let $\Pi\colon \bb R\to \bb T$ be the projection given by
$\color{blue} \Pi(x) = x - \lfloor x\rfloor$, $\lfloor x\rfloor$ being
the integer part of $x$.  
If the integral of $\mss b/\mss a$ on the interval $[0,1]$ vanishes,
$S(1)=0$, then the potential $S(\cdot)$ is periodic, and one can
define a probability measure $\mu_\epsilon$ on the torus $\bb T$ by
\begin{equation*}
\mu_\epsilon(d\theta) \,=\, \frac{1}{Z_\epsilon}\,
\frac{1}{\mss a(\theta)}\,
e^{-S(\theta)/\epsilon}\, d\theta\;,
\end{equation*}
where $Z_\epsilon$ is the normalization constant which turns
$\mu_\epsilon$ into a probability measure.  A simple computation shows
that the projected diffusion
${\color{blue} \mtt X_\epsilon(t)}: = \Pi( X_\epsilon(t))$ is
reversible with respect to the measure $\mu_\epsilon$. In particular,
$\mu_\epsilon$ is a stationary state for the diffusion
$\mtt X_\epsilon(\cdot)$.  Note, however, that we did not assume in
the previous section that $S(\cdot)$ is a periodic function.
\end{remark}

Denote by $\color{blue} N(x,r)$, $x\in \bb R$, $r>0$, the connected
component of the set $\{y\in \bb R : S(y) < S(x) + r\}$ which contains
$x$, and by $\color{blue} B(x,r)$ the open ball of radius $r$ centered
at $x$. Let Fix $\color{blue} 0<r_0<\mf h_1$, and let $\ms E(m)$,
$m\in \mc M$, be the well given by
\begin{equation}
\label{71}
{\color{blue}\ms E(m)} \,:=\, N(m ,r_0)\,\cap\, B(m ,r_0)\;.
\end{equation}
For each $m\in \mc M$ we say that $\ms E(m)$ is a well.  By definition
of $r_0$, $\mc E(m) \cap \mc E(m')=\varnothing$ if $m\neq m'$. In
particular, by \eqref{68}, the wells are ordered so that
$\mc E(m_j) < \mc E(m_k)$ for $j<k$. For $1\le p\le \mf q$, let
$\ms E(\ms M_{p}(k))$, $k\in \bb Z$, be the wells defined by
\begin{equation}
\label{52}
{\color{blue}\ms E (\ms M_{p}(k))}
\,:=\, \bigcup_{m\in \ms M_{p}(k)}\ms E(m)\;,
\end{equation}

Denote by $\color{blue} C(\bb R_{+}, \ms F)$, $\ms F$ a metric space,
the space of continuous trajectories with the locally uniform topology
and its Borel $\sigma$-algebra, and by
$\color{blue} D(\bb R_{+}, \ms F)$ the space of right-continuous
trajectories with left limits endowed with the Skorohod topology and
its Borel $\sigma$-algebra.  Let $\color{blue}\bb P^\epsilon_x$,
$\color{blue}\bb P^{p,\epsilon}_x$, $x\in \bb R$, be the probability
measure on $C(\bb R_{+}, \bb R)$ induced by the diffusion
$X_\epsilon(\cdot)$, $X_\epsilon(\theta^{(p)}_\epsilon\,\cdot)$
starting from $x$, respectively. Expectation with respect to
$\bb P^\epsilon_x$, $\bb P^{p,\epsilon}_x$ is represented by
$\color{blue}\bb E^\epsilon_x$, $\color{blue}\bb E^{p,\epsilon}_x$,
respectively.  Similarly, let $\color{blue}\mtt Q^{p}_{\nu}$,
$1\le p\le \mf q$, $\nu$ a probability measure on $\ms S_p$, be the
probability measure on $\ms D(\bb R_{+}, \ms S_p)$ induced by the
Markov chain $\bb X_p(\cdot)$ starting from $\nu$. Expectation with
respect to $\mtt Q^{p}_{\nu}$ is represented by
$\color{blue} \mtt E^{p}_{\nu}$. When $\nu$ is the Dirac measure
concentrated on a set $\ms M_p(k)$, $k\in\bb Z$, we denote
$\mtt Q^{p}_{\nu}$, $\mtt E^{p}_{\nu}$ by
$\color{blue}\mtt Q^{p}_{\ms M_p(k)}$,
$\color{blue}\mtt E^{p}_{\ms M_p(k)}$, respectively.

The next theorem is the main result of this section. It describes the
asymptotic behavior of the diffusion $X_\epsilon(\cdot)$ in terms of a
Markov chain on a countable state space. Recall the definition of the
probability measure $\mtt h_p(x, \cdot)$, $1\le p\le \mf q$,
$x\in \bb R$, on $\ms S_p$, introduced in \eqref{74}.

\begin{theorem}[Convergence of the finite-dimensional distributions]
\label{mt3}
For each $1\le p\le \mf q$, 
\begin{equation*}
\lim_{\epsilon\rightarrow0}\,
\mathbb{P}_{x}^{p,\epsilon}\Big[\,
\bigcap_{j=1}^{\mf n} \big\{\, \mtt x(t_{j})
\in \ms E (\ms M_p(k_j))\,\big\}\,\Big]
\ =\ \mtt Q^{(p)}_{\mtt h_p(x, \cdot)}
\Big[\,\bigcap_{j=1}^{n}
\big\{\, \mtt x (t_{j}) = \ms M_p( k_{j}) \,\big\}\,\Big]
\end{equation*}
for all $\mf n\ge 1$, $0<t_1< \dots <t_{\mf n}$,
$k_1, \dots, k_{\mf n} \in \bb Z$, $x\in \bb R$.
\end{theorem}

Theorems \ref{mt1} and \ref{mt2} are simple consequences of this
result. We present a proof at the end of this section. Following the
method introduced in \cite{lms}, we prove Theorem \ref{mt3}, in two
steps. We first show in Theorem \ref{mt4} that the solutions of the
resolvent equations are asymptotically constant on each well. Then,
following \cite{llm}, we deduce from this result the convergence of
the finite-dimensional distributions of $X_\epsilon$ in the
time-scales $\theta^{(p)}_\epsilon$.

For $1\le p\le \mf q$, $\lambda>0$, and a bounded function
$g\colon \ms S_{p}\to \bb R$, denote by $G \colon \bb R \to \bb R$ the
lifted function defined by \eqref{55} with
$\mf g(k) = g (\ms M_{p}(k))$, $k\in \bb Z$.  Let $\phi_{p,\epsilon}$
be the unique weak solution of the resolvent equation \eqref{41}.  We
refer to Appendix \ref{sec-ap1} for the definition of weak solutions,
and a proof of the existence and uniqueness.

\begin{theorem}
\label{mt4}
For every $1\le p\le \mf q$, $\lambda>0$ and bounded function
$g\colon \ms S_{p}\to \bb R$,
\begin{equation*} 
\lim_{\epsilon \to 0}
\sup_{k\in \bb Z} \sup_{x\in \ms E(\ms M_{p}(k))}
|\, \phi_{p,\epsilon}(x)-f (\ms M_{p}(k))\,| \,=\,  0\,,
\end{equation*}
where $f\colon \ms S_{p} \to \bb R$ is the unique bounded solution of
the reduced resolvent equation
\begin{equation}
\label{42}
(\lambda - \bb L_p)f = g\,.    
\end{equation}
\end{theorem}

\begin{remark}
\label{rm1}
Denote by $\tau (\mathscr{A})$, $\mathscr{A} \subset \bb R$,
the hitting time of the set $\mathscr{A}$:
\begin{equation}
\label{81}
{\color{blue}\tau (\ms A)} \ :=\
\inf\big\{ \, t>0\,:\, X_\epsilon (t) \in \ms A \,\big\} \;.
\end{equation}
By Corollary \ref{l26},
\begin{equation*}
\mtt h_p (x, \ms M_p(j)) \,=\,
\bb P_{x} \big[ \tau (\ms M_p(j)) \,=\,
\tau (\ms M_p)  \big] \,+\, o_\epsilon (1)
\end{equation*}
for all $2\le p\le \mf q$, $j\in \bb Z$, $x\in \bb R$. Here and below,
$\color{blue} o_\epsilon (1)$ stands for an expression which converges
to $0$ as $\epsilon\to 0$, uniformly in the parameters on which it
depends.
\end{remark}

\section{Hitting times}
\label{sec2}

In this section, we present general estimates on the hitting times of
the diffusion $X_\epsilon(\cdot)$. Denote by $\mtt x (\cdot)$ an
element of $\mc D(\bb R_+, \bb R)$, and by
$\color{blue}\bb P^{\vartheta, \epsilon}_x$, $x\in \bb R$, the
probability measure on $\ms D(\bb R_{+}, \bb R)$ induced by the
diffusion $X_\epsilon(\cdot)$ starting from $x$ and speeded-up by
$\vartheta$. Expectation with respect to
$\bb P^{\vartheta, \epsilon}_x$ is represented by
$\color{blue}\bb E^{\vartheta, \epsilon}_x$.

For every closed set $A\subset \bb T$ let
$\tau(A)$ be the first time the trajectory $\mtt x(\cdot)$ hits the
set $A$, that is
\begin{equation}
\label{hit1}
{\color{blue} \tau(A)} \,:=\,
\inf\{t\geq 0: \mtt x (t) \in A\}.
\end{equation}
If $A$ is a singleton or a pair, say $A = \{a\}$, $A = \{a, b\}$ for
some $a$, $b\in \bb R$, we write $\tau(a)$, $\tau(a,b)$ for
$\tau(\{a\})$, $\tau(\{a, b\})$, respectively. 

\begin{proposition}
\label{l04}
Let $F\colon \bb R \to \bb R$ be a bounded and measurable
function. Fix $\lambda>0$, $\vartheta > 0$, and denote by
$\psi = \psi_{\epsilon, \vartheta, \lambda, F}$ the solution of the
resolvent equation
\begin{equation*}
\{\lambda - \vartheta \, \ms L_{\epsilon}\}\, \psi = F\,.
\end{equation*}
Then,
\begin{equation}
\label{54}
{\color{blue} \Vert\psi \Vert_\infty} \,:=\,
\sup_{x\in \bb R} |\psi(x)| \,\le\, \frac{1}{\lambda}\,
\Vert F\Vert_\infty \;.
\end{equation}
Moreover,
\begin{equation*}
\Big|\, \psi(x) - \psi(a)\,
\bb P^{\epsilon}_{x}[\, \tau(a)< \tau(b)\,]
- \psi(b)\, \bb P^{\epsilon}_{x}
[\, \tau(b)< \tau(a)\,] \, \Big| \, \le \,
2\, \Vert F \Vert_\infty \;
\bb E^{\epsilon}_{x} [\, \tau(a,b) /\vartheta \, ]\;,
\end{equation*}
and
\begin{equation}
\label{08}
\big|\, \psi(x) - \psi(a)\, \big| \, \le \,
2\, \Vert F \Vert_\infty \; \Big\{\, \bb E^{\epsilon}_{x} [\,
\tau(a,b) /\vartheta \, ]
\;+\; \frac{1}{\lambda} \, 
\bb P^{\epsilon}_{x} [\, \tau(b)< \tau(a)\, ]\,\Big\}  
\end{equation}
for all $a < x < b$. 
\end{proposition}

\begin{proof}
Recall from \cite[Section 6.5]{Fri} the stochastic representation of
the solution of the resolvent equation:
\begin{equation}
\label{07}
\psi(x) = \bb
E^{\vartheta, \epsilon}_{x}\Big[\, \int_{0}^{\infty}e^{-\lambda s}\,
F ( \mtt x (s)) \; ds\,\Big]\;.
\end{equation}
Assertion \eqref{54} is a simple consequence of this formula.

Fix $a < x < b$ and let $\color{blue} \tau := \tau(a,b)$. By the
stochastic representation \eqref{07} of the resolvent equation, and
changing the measure from $\bb P^{\vartheta, \epsilon}_{x}$ to
$\bb P^{\epsilon}_{x}$, 
\begin{equation*}
\begin{aligned}
\psi(x) \, & =\, \bb E^{\epsilon}_{x}\Big[\,
\int_{0}^{\infty}e^{-\lambda s}\, F ( \mtt x (s\vartheta)) \;
ds\,\Big]
\\
\, & =\, \bb E^{\epsilon}_{x}\Big[\,
\int_{0}^{\tau/\vartheta}e^{-\lambda
s}\, F(\mtt x(s\vartheta) )\, ds\, \Big]
\,+\, \bb
E^{\epsilon}_{x}\Big[\, \int_{\tau/\vartheta}^{\infty}
\,  e^{-\lambda s}\,  F(\mtt x(s\vartheta) )\, ds\,
\Big]\, .
\end{aligned}
\end{equation*}
The absolute value of the first term is bounded by
$\Vert F \Vert_\infty \; \bb E^{\epsilon}_{x} [\, \tau/\vartheta \,
]$.  After a change of variables, by the strong Markov property and
the stochastic representation of the solution of the resolvent
equation, the second term is seen to be equal to
\begin{equation*}
\bb E^{\epsilon}_{x}\big[\,
e^{-\lambda \tau/\vartheta}\, 
\psi (\mtt x (\tau))\,\big]\;.
\end{equation*}
Decomposing this expectation according to the events
$\{\tau(a)< \tau(b) \}$ and
$\{\tau(b)< \tau(a) \}$ the previous expectation
can be written as
\begin{equation*}
\psi (a) \, \bb E^{\epsilon}_{x}\big[\,
e^{-\lambda \tau/\vartheta}
\,,\, \tau(a)< \tau(b)\, \big]
\,+\,
\psi (b) \, \bb E^{\epsilon}_{x}\big[\,
e^{-\lambda \tau/\vartheta}
\,,\, \tau(b)< \tau(a)\, \big]
\;.
\end{equation*}
Add and subtract $1$ to $\exp \{-\lambda \tau/\vartheta\}$ to
rewrite this expression as
\begin{equation}
\label{59}
\psi (a) \, \bb P^{\epsilon}_{x}\big[\,
\tau(a)< \tau(b)\, \big]
\,+\,
\psi (b) \, \bb P^{\epsilon}_{x}\big[\, 
\tau(b)< \tau(a)\, \big] \;+\; R(\epsilon)
\;,
\end{equation}
where
\begin{equation*}
\big| R(\epsilon) \big| \,\le\, \max\big\{ |\psi (a)| \,,\,
|\psi (b)| \big\}
\, \bb E^{\epsilon}_{x}\Big[\, 1 - e^{-\lambda \tau/\vartheta} \, \Big] \;.
\end{equation*}
As the absolute value of $\psi (\cdot)$ is uniformly bounded
by $(1/\lambda) \, \Vert F\Vert_\infty$, and $1 - e^{-x} \le x$ for
$x\ge 0$,
\begin{equation*}
\big| R(\epsilon) \big| \,\le\, \Vert F \Vert_\infty
\, \bb E^{\epsilon}_{x}\big[\,\tau/\vartheta \, \big] \;.
\end{equation*}
This proves the second assertion of the lemma.

To prove the last assertion of the lemma, return to \eqref{59}, and
write the sum of the first two terms as
\begin{equation*}
\psi (a) \, 
\,+\,
\big[\, \psi (b) \,-\, \psi (a) \big]
\, \bb P^{\epsilon}_{x}\big[\, 
\tau(b)< \tau(a)\, \big] \;.
\end{equation*}
To complete the proof it remains to recall the previous estimates and
that the absolute value of $\psi (\cdot)$ is uniformly bounded
by $(1/\lambda) \, \Vert F\Vert_\infty$.
\end{proof}

The next result provides bounds on the expectation of the hitting time
$\tau(a,b)$.

\begin{lemma}
\label{l01}
Given $a < b$, let $\color{blue} K = K_{a,b} = K^{^{\nearrow}}_{a,b} \wedge
K^{^\nwarrow}_{a,b}$, where
\begin{equation*}
K^{^{\nearrow}}_{a,b} \,:=\, \sup_{a \le y\le x \le b} \{\,
S(x)-S(y)\,\}\,, \quad
K^{^\nwarrow}_{a,b} \,:=\,
\sup_{a \le y\le x \le b} \{\, S(y)-S(x)\,\} \;.
\end{equation*}
Then, 
\begin{equation*}
\sup_{x\;\in\; [a,b]}\bb
E^{\epsilon}_{x}\big[\, \tau(a,b)\, \big]
\, \leq\, \mf c_0\, \epsilon^{-1}\, e^{K_{a,b}/\epsilon}
\end{equation*}
for all $\epsilon > 0$, where
${\color{blue} \mf c_0 }:= \sup_{x\in \bb R} \mss a (x)^{-1}$.
\end{lemma}

\begin{proof}
Consider the equation
\begin{equation}
\label{36}
\begin{cases} (\ms L_{\epsilon}u)(x) = 1 \;\; \text{if}\ x\in (a, b)
\\
u(a) = u'(a) = 0\, .
\end{cases}
\end{equation}
By a direct computation it can be verified that the function
$\zeta_{\epsilon}: (a, b) \to \bb R$ given by
\begin{equation}
\label{40}
\zeta_{\epsilon}(x) =
\frac{1}{\epsilon}\int_{a}^{x}e^{\frac{S(y)-S(a)}{\epsilon}}
\int_{a}^{y} \frac{1}{\mss a (z)} \, e^{\frac{S(a)-S(z)}{\epsilon}}dzdy
\end{equation}
solves the previous equation. Note that $\zeta_{\epsilon}$ is of class
$C^{2}$ on $(a, b)$ and positive. Let $\xi_{\epsilon}$ be an extension
of $\zeta_{\epsilon}$ to an element in $C^{2}(\bb R)$. In other words,
$\xi_\epsilon \in C^{2}(\bb R)$, has compact support, and satisfies
$\xi_{\epsilon}(x) = \zeta_{\epsilon}(x)$ for every $x\in (a, b)$.

Fix $x_{0}\in [a, b]$.  Since $\xi_\epsilon$ is in the domain of
the generator $\ms L_{\epsilon}$, the following expression is a
martingale:
\begin{equation*}
M_{t} \,=\, \xi_{\epsilon}(X_{\epsilon} (t)) -
\int_{0}^{t}\ms [\ms
L_{\epsilon}\xi_{\epsilon}](X_{\epsilon} (s)) \; ds\,.
\end{equation*}
Let $\tau := \tau(a,b)$. By \eqref{36} and the definition of
$\xi_{\epsilon}$
\begin{equation*}
M_{t\wedge \tau} \,=\,
\zeta_{\epsilon}(X_{\epsilon}(t\wedge \tau)) \,-\,  t\wedge
\tau, \quad\forall t\geq 0.
\end{equation*}
Therefore,
\begin{equation*}
\zeta_{\epsilon}(x_0) \; = \; \bb
E^\epsilon_{x_0}\big[M_{0}\big] \; = \; \bb
E^\epsilon_{x_0}\big[M_{t\wedge \tau}\big]
\;=\; \bb E^\epsilon_{x_{0}}\big[\zeta_{\epsilon}(\mtt x (t\wedge
\tau))\big] \;-\; \bb
E^\epsilon_{x_{0}}\big[t\wedge \tau\big]\, ,
\end{equation*}
so that
\begin{equation}
\label{10}
\bb E^\epsilon_{x_{0}} \big[ \, \tau \,\big ] 
\,=\, \lim_{t\to\infty}
\bb E^\epsilon_{x_{0}} \big[\, t \wedge \tau
\, \big]  \; = \;
\bb E^\epsilon_{x_{0}}\big[\, \zeta_{\epsilon}(
\mtt x(\tau)) \, \big] \,-\,  \zeta_{\epsilon}(x_0) \;.
\end{equation}
In particular, by the explicit expression of the positive function
$\zeta_{\epsilon}$, and a direct computation, 
\begin{equation*}
\bb E^\epsilon_{x_{0}} [\tau ]   \,\leq\,
\sup_{x\;\in\; [a, b]} \zeta_{\epsilon}(x) \;
\le \; \mf c_0 \, \epsilon^{-1}\, e^{K^{^\nearrow}_{a,b}/\epsilon} 
\end{equation*}
for all $\epsilon>0$, $x_0\in [a,b]$, where, recall,
$\mf c_0 = \sup_{x\in \bb R} [1/\mss a(x)]$, and
$K^{^\nearrow}_{a,b} = \sup_{a \le y\le x \le b} \{\, S(x)-S(y)\,\}$.

By choosing $\zeta_\epsilon (\cdot)$ as the solution of 
\begin{equation*}
\begin{cases}
(\ms L_{\epsilon}u)(x) = 1 \;\; \text{if}\ x\in (a, b)
\\
u(b) = u'(b) = 0\,,
\end{cases}
\end{equation*}
the same argument yields the alternative bound
\begin{equation*}
\sup_{x\;\in\; [a,b]}\bb
E^\epsilon_{x}\big[\, \tau(a,b)\, \big]
\, \leq\, \mf c_0 \,  \epsilon^{-1}\, e^{K^{^\nwarrow}_{a,b}/\epsilon}
\end{equation*}
for all $\epsilon > 0$, where
$K^{^\nwarrow}_{a,b}\,:=\, \sup_{a \le y\le x \le b} \{\, S(y)-S(x)\,\}$.  This
completes the proof of the lemma.
\end{proof}

The next result shows that the set $\ms M_p$ is attained in the
time-scale $\epsilon^{-1} e^{\mf h_{p-1}/\epsilon}$. It will be
improved in Corollary \ref{l40}, after we prove the metastable
behavior of the diffusion $X_\epsilon(\cdot)$ at time-scale
$\theta^{(p)}_\epsilon$.  

\begin{lemma}
\label{l33}
There exists finite constant $C_0 = C_0(\mf q, \mtt a(\cdot))$ such
that
\begin{equation*}
\sup_{x\in \bb R}
\bb E_{x}^{\epsilon} [\, \tau (\ms M_p) \, ]
\,\le\, C_0 \, \epsilon^{-1}\, e^{\mf h_{p-1}/\epsilon}
\end{equation*}
for all $1\le p\le \mf q$, provided $\mf h_0=0$.
\end{lemma}

\begin{proof}
The proof is carried by induction on $p$. By Lemma \ref{l01}, as
$\ms M_1 = \mc M$ contains all $S$-local minima,
$\bb E_{x}^{\epsilon} [\, \tau (\ms M_1) \, ] \le \mf c_0 \,
\epsilon^{-1}$ for all $x\in\bb R$.

We turn to $p\ge 2$. Bound $\tau (\ms M_p)$ by
$\tau (\ms M_{p-1}) + \tau (\ms M_p) \circ \vartheta_{\tau (\ms
M_{p-1})}$, where $(\vartheta_r : r\ge 0)$ represents the group of
translations of a trajectory. By the strong Markov property,
\begin{equation*}
\sup_{x\in \bb R}
\bb E_{x}^{\epsilon} [\, \tau (\ms M_p) \, ]
\,\le\, \sup_{x\in \bb R}
\bb E_{x}^{\epsilon} [\, \tau (\ms M_{p-1}) \, ]
\,+\, \sup_{m\in \ms M_{p-1}}
\bb E_{m}^{\epsilon} [\, \tau (\ms M_p) \, ]
\,. 
\end{equation*}
Fix $m\in \ms M_{p-1}$. If $m\in \ms M_{p}$, $\tau (\ms
M_p)=0$, and there is nothing to prove.

Assume, therefore that $m\in \ms M_{p-1}(j)$ for some $j\in \bb Z$,
and that $\ms M_{p-1}(j) \cap \ms M_p = \varnothing$. This means that
$\ms M_{p-1}(j)$ is a $\bb X_{p-1}$-transient state. By periodicity of
the jump rates, $\ms M_{p-1}(j)$ may reach a $\bb X_{p-1}$-recurrent
state in a finite number of jumps.  Therefore, by Proposition
\ref{l13}, Lemma \ref{l27}, and postulate $\mc P_7(p-1)$, if
$\ms M_{p-1}(j)$ can reach a set $\ms M_p(\ell) \subset \ms M_p$ with
jumps to the right ($j<\ell$),
$K^{^\nearrow}_{m,m^-_{p-1,\ell}} \le \mf h_{p-1}$.  Thus, by Lemma
\ref{l01},
\begin{equation*}
\bb E_{m}^{\epsilon} [\, \tau (\ms M_p) \, ] \,\le\, \mf c_0\,
\epsilon^{-1}\, e^{\mf h_{p-1}/\epsilon}\;.
\end{equation*}
A similar bound can be obtained if $\ms M_{p-1}(j)$ can reach the set
$\ms M_p(\ell) \subset \ms M_p$ with jumps to the left.  To complete
the proof, it remains to recollect all previous estimates.
\end{proof}

The next result is the main estimate needed in the proof of Corollary
\ref{l32} which states that the diffusion $X_\epsilon(\cdot)$ may not
escape from a well quickly.

\begin{lemma}
\label{l31}
Fix $b_l < b_r$ such that $S(x) < S(b_l) = S(b_r)$ for all
$b_l<x<b_r$. Let $a\in (b_l, b_r)$ such that
$S(a) = \min \{S(x) : x\in [b_l,b_r]\}$.  Then, for all $t>0$,
$\delta>0$, there exist a finite contant $C_0$, depending only on
$\mtt a(\cdot)$, and $\eta = \eta(\delta)$,
$0<\eta < \min\{b_r-a, a-b_l\}/2$ such that
\begin{equation*}
\bb P_{x}^{\epsilon}
\big [\, \tau(b_l, b_r) \le t \,\big]
\,\le\, \frac{C_0}{\eta^2} \, \big\{\, \epsilon\, t + (x-a)^2
e^{K(x)/\epsilon} \, \big\}\, e^{- [S(b_r) - S(a) - \delta]/\epsilon}
\end{equation*}
for all $x\in (b_l, b_r)$. Here,
\begin{equation*}
K(x) \,=\, K^{^\nearrow}_{a,x} \;\;\text{if $x>a$ and}\;\;
K(x) \,=\, K^{^\nwarrow}_{x,a} \;\;\text{if $x<a$} \;.
\end{equation*}
\end{lemma} 

\begin{proof}
Denote by $\zeta_\epsilon$ the solution of the equation \eqref{36} in
the interval $(b_l, b_r)$. Let $\tau = \tau(b_l, b_r)$. By the
displayed equation above \eqref{10},
\begin{equation}
\label{37}
\bb E^\epsilon_{x}\big[\,
\zeta_{\epsilon}(\mtt x (t\wedge \tau))\, \big]
\,=\,
\zeta_{\epsilon}(x) \,+\,
\bb E^\epsilon_{x}\big[\, t\wedge \tau\, \big]
\end{equation}
for all $x \in (b_l, b_r)$, $t>0$.

We estimate separately each term of the previous identity. Assume,
without loss of generality, that $a\le x\le b_r$. The second
term on the right-hand site is bounded by $t$. The first one, by the
explicit formula \eqref{40} for $\zeta_\epsilon$, is bounded by
\begin{equation*}
\frac{\mf c_0}{\epsilon}\, (x-a)^2\, e^{\sup_{a\le z\le y\le x}
[S(y)-S(z)]/\epsilon} \,=\,
\frac{\mf c_0}{\epsilon}\, (x-a)^2\, e^{K^{^\nearrow}_{a,x}/\epsilon}
\end{equation*}
where $\mf c_0$, $K^{^\nearrow}_{a,x}$ have been introduced in the
statement of Lemma \ref{l01}. .

We turn to the term on the left-hand side of \eqref{37}. Since
$\zeta_\epsilon(\cdot)$ is positive, it is bounded below by
\begin{equation*}
\bb E^\epsilon_{x}\big[\,
\zeta_{\epsilon}(\mtt x (\tau))\, \mtt 1\{ \tau \le t\} \big]
\,\ge\, \min_{x'\in \{b_l,b_r\}} \zeta_{\epsilon}(x')\,\,
\bb P^\epsilon_{x}\big[\,\tau \le t \big]\,.
\end{equation*}
We estimate $\zeta_{\epsilon}(b_r)$, the same argument applies to
$b_l$. Fix $\delta >0$. There exists $0<\eta< (b_r-a)/2$ such that
$|S(y) - S(b_r)|\le \delta/2$ for all $|y-b_r|\le \eta$,
$|S(z) - S(a)|\le \delta/2$ for all $|z-a|\le \eta$.
By the explicit formula \eqref{40} for $\zeta_{\epsilon} (\cdot)$,
\begin{equation*}
\zeta_{\epsilon}(b_r) \,\ge\, \frac{c_0}{\epsilon}\,
\int_{b_r-\eta}^{b_r} dy\int_a^{a+\eta} e^{[S(y)-S(z)]/\epsilon}\,
dz \,\ge\, \frac{c_0}{\epsilon}\, \eta^2\,
e^{[S(b_r)-S(a) - \delta]/\epsilon}
\end{equation*}
for some positive constant $c_0$ depending only on $\mtt a(\cdot)$. 
To complete the proof, it remains to recollect all previous
estimates. 
\end{proof}

Fix $2\le p\le \mf q$, $k\in \bb Z$, $m\in \ms M_p(k)$. Denote 
by $m_{p-1, \mtt l}$, $m_{p-1, \mtt r}$ the first element to left,
right of $m$ which belongs to $\ms M_{p-1}\setminus \ms M_p$,
respectively:
\begin{equation*}
{\color{blue} m_{p-1, \mtt r}} \,:=\,
\min\big\{m' \in \ms M_{p-1}\setminus \ms M_p :
m'>m \big\}\,, \quad 
{\color{blue} m_{p-1, \mtt r}} \,:=\,
\max \big\{m' \in \ms M_{p-1}\setminus \ms M_p : m'<m \big \} \;.
\end{equation*}

\begin{lemma}
\label{l34b}
Fix $2\le p\le \mf q$.  There exists $\kappa>0$ such that
\begin{equation}
\label{44}
\Lambda (m, m_{p-1, \mtt r}) \,-\, S(\ms M_p(k)) \,> \, \mf h_{p-1}
\,+\, \kappa \;,
\quad
\Lambda (m, m_{p-1, \mtt l}) \,-\, S(\ms M_p(k)) \, > \, \mf h_{p-1}
\,+\, \kappa
\end{equation}
for all $k\in\bb Z$.
\end{lemma}

\begin{proof}
We prove the assertion for $m_{p-1, \mtt r}$, the same argument apply
to $m_{p-1, \mtt l}$. By definition,
$m_{p-1, \mtt r}\in \ms M_{p-1}(\ell)$ for some $\ell\in\bb Z$. On the
other hand, as $m \in \ms M_p$, and $m< m_{p-1, \mtt r}$,
$m\in \ms M_{p-1}(j)$ for some $j\le \ell$. Note that $j$ cannot be equal
to $\ell$ because $m\in \ms M_p$ and
$m_{p-1, \mtt r} \not\in \ms M_p$. Moreover, by definition of
$m_{p-1, \mtt r}$, $\ms M_{p-1}(i) \subset \ms M_p$ for all
$j\le i<\ell$. Recall that $m\in \ms M_p(k)$.

\smallskip
\noindent{{\bf Case A}: Suppose that
$\ms M_{p-1}(\ell-1) \subset \ms M_p(k')$ for some $k'\neq k$.}
\smallskip

In this case, as the sets $\ms M_{p}(i)$ are ordered, $k<k'$. By
postulate $\mc P_6(p)$,
$\Lambda (\ms M_p(k), \ms M_p(k+1)) - S(\ms M_p(k)) \ge \mf h_p$.
Since $k' \ge k +1$ and $\ms M_{p-1}(\ell-1) \subset \ms M_p(k')$,
$\Lambda (\ms M_p(k), \ms M_p(k+1)) \le \Lambda (\ms M_p(k), \ms
M_p(k')) \le \Lambda (\ms M_p(k), \ms M_{p-1}(\ell-1))$. As
$\ell-1<\ell$ and $m_{p-1, \mtt r} \in \ms M_{p-1}(\ell)$,
$\Lambda (\ms M_p(k), \ms M_{p-1}(\ell-1)) \le \Lambda (\ms M_p(k),
m_{p-1, \mtt r})$. Finally, as $m\in \ms M_p(k)$, $\Lambda (\ms M_p(k),
m_{p-1, \mtt r}) \le \Lambda (m, m_{p-1, \mtt
r})$. Collecting all previous estimates yields  that
\begin{equation*}
\Lambda (m, m_{p-1, \mtt r})
\,-\, S(\ms M_p(k)) \, \ge \, \mf h_{p}\;,
\end{equation*}
as claimed since $\mf h_{p} > \mf h_{p-1}$.

\smallskip
\noindent{{\bf Case B}: Suppose that
$\ms M_{p-1}(\ell-1) \subset \ms M_p(k)$.}
\smallskip

As $\ms M_{p-1}(\ell-1) \subset \ms M_p$,
$\ms M_{p-1}(\ell) \cap \ms M_p = \varnothing$, $\ms M_{p-1}(\ell-1)$
is $\bb X_{p-1}$-recurrent and $\ms M_{p-1}(\ell)$ is $\bb X_{p-1}$
transient. Hence,
$R_{p-1} (\ms M_{p-1}(\ell-1), \ms M_{p-1}(\ell))=0$. Thus, by
postulate $\mc P_6(p-1)$ and $\mc P_7(p-1)$,
\begin{equation*}
\Lambda (\ms M_{p-1}(\ell-1), \ms M_{p-1}(\ell))
\,-\, S(\ms M_{p-1}(\ell-1)) \,> \, \mf h_{p-1}\;.
\end{equation*}
As $m\in \ms M_{p-1}(j)$, $j\le \ell -1$, and $m_{p-1, \mtt r}\in \ms
M_{p-1}(\ell)$, $\Lambda (\ms M_{p-1}(\ell-1), \ms M_{p-1}(\ell)) \le
\Lambda (m, m_{p-1, \mtt r})$, so that
\begin{equation*}
\Lambda (m, m_{p-1, \mtt r})
\,-\, S(\ms M_{p-1}(\ell-1)) \,> \, \mf h_{p-1}\;.
\end{equation*}
Since $\ms M_{p-1}(\ell-1)$ and $\{m\}$ are contained in
$\ms M_{p}(k)$, by postulate $\mc P_4(p)$,
$S(\ms M_{p-1}(\ell-1)) = S(m) = S(\ms M_{p}(k))$. Since
$\mtt a(\cdot)$, $\mtt b(\cdot)$ are periodic,
$\Lambda (m^+_{p,k}, m_r) \,-\, S(\ms M_p(k)) $ takes only a finite
number of distinct values. Therefore, in the previous inequality we
may substitute $\mf h_{p-1}$ by $\mf h_{p-1}+\kappa$ for some
$\kappa>0$, as claimed.
\end{proof}

\begin{lemma}
\label{l35}
Fix $2\le p\le \mf q$. There exists $\kappa>0$ satisfying the
following property.  For all $k\in \bb Z$, $m\in \ms M_p(k)$, there
exist $m_{p-1, \mtt l} < \sigma_l < m < \sigma_r < m_{p-1, \mtt r}$
such that
\begin{equation}
\label{46}
\begin{gathered}
S(m) \,\le\, S(x) \, < \,  S(\sigma_l)  \,=\, S(\sigma_r)
\;\;
\text{for all}\;\; x\in (\sigma_l,\sigma_r) \;,
\\
\text{and}\;\; 
S(\sigma_r) \,>\, S(m) \,+\, \mf h_{p-1} \,+\, \kappa \;.
\end{gathered}
\end{equation}
\end{lemma}

\begin{proof}
To prove this assertion, let $\sigma_r$, $\sigma_l$ be the first
$S$-local maxima on the right, left of $m$ which attain
$\Lambda (m_{p-1, \mtt l} , m)$, $\Lambda (m, m_{p-1, \mtt r})$,
respectively:
\begin{gather*}
\sigma_r \,=\, \min\big\{\sigma' \in \mc W :
\sigma' > m \,,\, S(\sigma') = \Lambda (m, m_{p-1, \mtt r})\, 
\big\}\,,
\\
\sigma_l \,=\, \max \big\{ \sigma' \in \mc W :
\sigma' < m \,,\, S(\sigma') = \Lambda (m_{p-1, \mtt l},m)\, 
\big \} \;.
\end{gather*}
By definition,
$m_{p-1, \mtt l} < \sigma_l < m < \sigma_r < m_{p-1, \mtt r}$.  By
Lemma \ref{l34b}, the second line of \eqref{46} holds for some
$\kappa>0$.

Recall from \eqref{64} the definition of $\sigma_{j,j+1}^{p,\pm}$,
$j\in\bb Z$. If $\sigma_r> \sigma_{k,k+1}^{p,-}$, replace $\sigma_r$
by $\sigma_{k,k+1}^{p,-}$. By postulate $\mc P_6(p)$ and the
definition of $\sigma_{k,k+1}^{p,-}$,
$S(\sigma_{k,k+1}^{p,-}) - S(m) = S(\sigma_{k,k+1}^{p,-}) - S(\ms
M_p(k)) \ge \mf h_p > \mf h_{p-1}$. Thus, if
$\sigma_r> \sigma_{k,k+1}^{p,-}$ and we perform the replacement, the
second line of \eqref{46} still holds for some $\kappa>0$, and
$\sigma_r \le \sigma_{k,k+1}^{p,-}$. A similar procedure has to be
carried out for $\sigma_l$.

Since
$\sigma_{k-1,k}^+ \le \sigma_l < m < \sigma_r \le \sigma_{k,k+1}^-$,
and $m\in\ms M_p(k)$, by Proposition \ref{l13}, and the previous
bounds,
\begin{equation}
\label{45}
S(m) \,\le\, S(x)\;\; \text{for all}
\;\; x\in (\sigma_l,\sigma_r)\,, \;\;
\text{and}\;\; 
S(\sigma_l) \wedge S(\sigma_r)
\,>\, S(m) \,+\, \mf h_{p-1} \,+\, \kappa
\end{equation}
for some $\kappa>0$. If $S(x) \ge S(\sigma_l)\wedge S(\sigma_r)$ for
some $m\le x< \sigma_r$, replace $\sigma_r$ by the first point
$\sigma'$ to the right of $m$ such that
$S(\sigma') = S(\sigma_l)\wedge S(\sigma_r)$. Do the same thing for
$\sigma_l$. Clearly \eqref{45} still holds. Moreover,
$S(\sigma_l) = S(\sigma_r)$ and $S(x) < S(\sigma_l)$ for all
$x\in (\sigma_l,\sigma_r)$. The constant $\kappa$ can be taken
independent of $k$ due to the periodicity of $\mtt a(\cdot)$,
$\mtt b(\cdot)$. This completes the proof of the lemma.
\end{proof}

We have now all elements to show that the diffusion cannot leave a
well in a short time.  Let
${\color{blue} \ms E_p} := \cup_{k\in\bb Z} \ms E(\ms M_p(k))$, where
$\ms E (\ms M_p(k))$ has been introduced in \eqref{52}.

\begin{corollary} 
\label{l32}
Fix $2\le p\le \mf q$. Then, there exist $\kappa_0>0$ and a finite
$C_0=C_0(\mtt a(\cdot), \kappa_0)$, such that
\begin{equation*}
\sup_{k\in \bb Z} \sup_{m\in \ms M_p(k)} \bb P_m^{\epsilon}
\big [\, \tau(\ms E_{p-1}\setminus \ms E_p)
\,\le\,  e^{[\mf h_{p-1}+\kappa_0]/\epsilon}  \,\big]
\,\le\, C_0 \, \epsilon \, e^{- \kappa_0/\epsilon}\;.
\end{equation*}
\end{corollary}

\begin{proof}
Fix $2\le p\le \mf q$, $k\in \bb Z$, $m\in \ms M_p(k)$. Let
$\kappa>0$, $\sigma_l$, $\sigma_r\in \bb R$ be the elements given by
Lemma \ref{l35} and set $\kappa_0 = \kappa/3$. As
$m_{p-1, \mtt l} < \sigma_l < m < \sigma_r < m_{p-1, \mtt r}$,
starting from $m$, $\tau(\sigma_l, \sigma_r) \le \tau(\ms
E_{p-1}\setminus \ms E_p)$. Hence,
\begin{equation*}
\bb P_m^{\epsilon} \big [\, \tau(\ms E_{p-1}\setminus \ms E_p) \,\le\,
e^{[\mf h_{p-1}+\kappa_0]/\epsilon} \,\big] \,\le\, \bb P_m^{\epsilon}
\big [\, \tau(\sigma_l, \sigma_r) \le e^{[\mf
h_{p-1}+\kappa_0]/\epsilon} \,\big]\;.
\end{equation*}
By Lemma \ref{l35}, the hypotheses of Lemma \ref{l31} are fulfilled,
and $S(m) = \min \{ S(x) : x\in [\sigma_l, \sigma_r]\}$. Therefore, by
Lemma \ref{l31} with $a = x = m$, $\delta = \kappa_0$,
$t=e^{[\mf h_{p-1}+\kappa_0]/\epsilon}$
\begin{equation*}
\bb P_m^{\epsilon}
\big [\, \tau(\sigma_l, \sigma_r) \le e^{[\mf h_{p-1}+\kappa_0]/\epsilon}
\,\big] \,\le\, C_0\, \epsilon\,  e^{[\mf h_{p-1}+\kappa_0]/\epsilon}
\, e^{- [S(\sigma_r) - S(m) - \kappa_0]/\epsilon}
\end{equation*}
for some finite constant $C_0=C_0(\mtt a(\cdot), \kappa_0)$, and whose
value may change from line to line. By Lemma \ref{l35}, $S(\sigma_r) -
S(m) \ge \mf h_{p-1} + \kappa = \mf h_{p-1} + 3 \kappa_0$. in view of
the previous estimates,
\begin{equation*}
\bb P_m^{\epsilon}
\big [\, \tau(\ms E_{p-1}) \,\le\, e^{[\mf h_{p-1}+\kappa_0]/\epsilon}
\,\big] \,\le\, C_0\,  \epsilon \, e^{- \kappa_0/\epsilon}\;. 
\end{equation*}
As $\mtt a(\cdot)$, $\mtt b(\cdot)$ are periodic, this estimate is
uniform over $k\in\bb Z$.
\end{proof}

The next results provide bounds on the probability of the event
$\{\tau(a) < \tau(b)\}$.

\begin{lemma}
\label{l02}
Fix $a < b $. 
Then, 
\begin{equation*}
\begin{gathered}
\bb P^{\epsilon}_{x}\big[\tau(b)\;<\; \tau(a)\big] \,\leq\,
\frac{(x-a)}{\eta}   \,
\exp\Big\{ \,\big[ \, \sup_{z\in [a,x] } S(z) - \inf_{y\in [b-\eta, b]} S(y)
\big]/\epsilon \,\Big\} \;,
\\
\bb P^{\epsilon}_{x}\big[\tau(a)\;<\;
\tau(b)\big] \,\leq\,
\frac{(b-x)}{\eta}   \,
\exp\Big\{ \,\big[ \, \sup_{z\in [x,b] } S(z) - \inf_{y\in [a,a+\eta]} S(y)
\big]/\epsilon \,\Big\} 
\end{gathered}
\end{equation*}
for all $\epsilon > 0$, $a<x<b$, $0<\eta< b-a$.
\end{lemma}

\begin{proof}
Consider the
equation
\begin{equation}
\label{hit1:eq1}
\begin{cases}
(\ms L_{\epsilon}u)(x) = 0 &\text{if}\ x\in (a,b) \\ u(a) =
0, u'(a) = 1
\end{cases}
\end{equation}
and let $\zeta_{\epsilon}$ be its unique solution. By a direct
computation, it can be verified that
\begin{equation}
\label{hit1:eq2}
\zeta_{\epsilon}(x) = \int_{a}^{x}e^{\frac{S(y)-S(a)}{\epsilon}}dy.
\end{equation}

Fix $x \in (a,b)$, $0<\eta<b-a$.  Let $\tau := \tau(a,b)$. By
\eqref{hit1:eq1},
$Y_{t} := \zeta_{\epsilon}(X_{\epsilon} (t\wedge \tau))$ is a well
defined bounded martingale. By the optional stopping theorem and
\eqref{hit1:eq2},
\begin{equation}
\label{63}
\bb P^\epsilon_{x}\big[\tau(b)< \tau(a)\big] \,=\,
\frac{\zeta_{\epsilon}(x)}{\zeta_{\epsilon}(b)} \,\le\,
\int_{a}^{x}e^{S(y)/ \epsilon}\, dy \Big/
\int_{b-\eta}^{b}e^{S(z)/ \epsilon}\, dz\;.
\end{equation}
To complete the proof of the first bound, it remains to estimate in
the last ratio $S(y)$ by $\sup_{y'\in [a,x]} S(y')$ and $S(z)$ by
$\inf_{z'\in [b-\eta,b]} S(z')$.

To obtain the second estimate, replace the boundary conditions in
\eqref{hit1:eq1} by $u(b)=0$, $u'(b)=-1$, note that the solution is
given by
\begin{equation*}
\zeta_{\epsilon}(x) \,=\, 
\int_{x}^{b} e^{\frac{S(y)-S(b)}{\epsilon}}\, dy\, ,
\end{equation*}
and repeat the previous arguments.
\end{proof}

\begin{corollary}
\label{l26}
Fix $a<b$ such that $\max\{S(a), S(b)\} < \max_{x\in [a,b]}
S(x)$. Denote by $\mf S$ the finite set where $S(\cdot)$ attains its
global maximum in the interval $[a,b]$:
\begin{equation*}
\mf S \,=\, \big\{\sigma \in (a,b) \cap \mc W : S(\sigma) = \max_{x\in
[a,b]} S(x) \big\}\;,
\end{equation*}
and enumerate $\mf S$ as $\mf S = \{\sigma_1, \dots,
\sigma_\ell\}$. Fix $x\not\in \mf S$. Then,
\begin{equation*}
\bb P^\epsilon_{x}\big[\tau(b)< \tau(a)\big] \,=\,
\frac{\sum_{i\in I_x}  1/\sqrt{- S''(\sigma_i)}}
{\sum_{1\le  j \le \ell} 1/\sqrt{- S''(\sigma_j)}} \,+\, o_\epsilon (1)\;,
\end{equation*}
where $I_x = \{i\in \bb Z : 1\le i\le\ell\,,\, \sigma_i <x\}$.
\end{corollary}

\begin{proof}
The first identity in equation \eqref{63} provides an explicit formula
for $\bb P^\epsilon_{x} [\tau(b)< \tau(a)]$. multiply the numerator
and the denominator by
$(2\pi \epsilon)^{-1/2} e^{- S(\sigma_1)/\epsilon}$. It remains to use
second order Taylor expansion to conclude.
\end{proof}

If $x\in \mf S$, the same argument shows that
\begin{equation}
\label{89}
\bb P^\epsilon_{x}\big[\tau(b)< \tau(a)\big] \,=\,
\frac{1}
{\sum_{1\le  j \le \ell} 1/\sqrt{- S''(\sigma_j)}}
\Big\{ \sum_{i\in I_x}  \frac{1}{\sqrt{- S''(\sigma_i)}} \,+\,
\frac{1}{2} \frac{1}{\sqrt{- S''(x)} } \, \Big\}
\,+\, o_\epsilon (1)\;.
\end{equation}

\section{Proof of Theorem \ref{mt4} for $p=1$}
\label{sec6}

The proof is divided in two steps. We first show that the solution of
the resolvent equation \eqref{41} is asymptotically constant on each
well and then that all limit points are solution of the reduced
resolvent equation \eqref{42}.  More precisely, in Proposition
\ref{l06}, based on local ergodic properties of the diffusion, we show
that
\begin{equation}
\label{05}
\lim_{\epsilon \to 0} \max_{k\in \bb Z}
\sup_{x,y \in \ms E(m_k)}
|\phi_{1,\epsilon}(x) - \phi_{1,\epsilon}(y)| = 0\,.
\end{equation}
Then, in Lemma \ref{l05}, we prove that
\begin{equation}
\label{06}
\lim_{\epsilon \to 0} \max_{k\in  \bb Z}
|\phi_{1,\epsilon}(m_k) - f(\ms M_1(k))| = 0\,,
\end{equation}
where $f(\cdot)$ is the solution of the reduced resolvent equation
\eqref{42}.  Theorem \ref{mt4} for $p=1$ is an easy consequence of
these two results.

\subsection*{Local ergodicity}

We start with the local ergodicity.  We show that the solution of the
resolvent equation is asymptotically constant on each well by
guaranteeing that, for any $k\in \bb Z$ and $x\in \ms E(m_{k})$, the
process $X_{\epsilon}(t)$, starting at $x$, hits the critical point
$m_{k}$, before leaving the well $\ms E(m_{k})$, in a time-scale much
smaller than the metastable scale $\theta^{(1)}_{\epsilon}$.

Let
\begin{equation}
\label{14}
{\color{blue} \delta_{\epsilon}(k)}
\,:=\,  \sqrt{\frac{2\, \epsilon \log
\epsilon^{-1}}{-\, S''(\sigma_{k})}}\;,
\quad
{\color{blue} \eta_{\epsilon}(k)}
\,:=\, \frac{\epsilon}{\delta_{\epsilon}(k)}
\;=\;
\sqrt{\frac{-\, \epsilon \, S''(\sigma_{k})}
{2 \,  \log \epsilon^{-1}}}\;,
\quad k\in \bb Z\;.
\end{equation}
Denote by $J_k$, $J^+_k$, $J^-_k$ the intervals given by
\begin{equation*}
{\color{blue} J^+_k} \,:=\,
[\, m_k, \sigma_{k+1} - \delta_{\epsilon}(k+1)\, ]\;, \quad
{\color{blue} J^-_k} \, :=\,  [\,\sigma_k + \delta_{\epsilon}(k),
m_k\, ]\;, \quad
{\color{blue} J_k} \, := \, J^-_k \cup J^+_k\;.
\end{equation*}
The next result establishes the local ergodicity. Recall the
definition of $\mf c_0$ introduced in Lemma \ref{l01}.

\begin{remark}
The result below is much stronger than what is required to prove the
local ergodicity stated in \eqref{05} since it establishes that, for
each $k\in \bb Z$, $\phi_{1,\epsilon}$ is asymptotically constant in a
slightly smaller interval than the one formed by the two local maxima
containing the local minimum $m_k$. (In particular, for $\epsilon$
small, $\phi_{1,\epsilon}$ is almost a step function). This stronger
result is needed, however, in the proof of the convergence of
$\phi_{1,\epsilon}$. More precisely, in the proof of Lemma
\ref{as1}. See below equation \eqref{19}.
\end{remark}

\begin{proposition}
\label{l06}
For each $k\in \bb Z$, 
\begin{equation*}
\sup_{x\in  J_k}
|\phi_{1,\epsilon}(x) - \phi_{1,\epsilon}(m_{k})| \,\le\,
2\, \Vert g\Vert_\infty\, \Big\{\,
\frac{\mf c_0}{\epsilon\, \theta^{(1)}_\epsilon}
\,+\, \frac{1}{\lambda} \,
\frac{2 \, \sqrt{2} \, \sqrt{\epsilon}}
{\delta_\epsilon (k) \wedge \delta_\epsilon (k+1) } \,\Big\}
\end{equation*}
for all $\epsilon$ sufficiently small. In particular, as
$\mss a(\cdot)$, $\mss b(\cdot)$ are periodic, \eqref{05} holds.
\end{proposition}

\begin{proof}
Fix $k\in \bb Z$, and assume that $x\in J^+_k$. The same argument
applies to $x\in J^-_k$. The proof relies on Proposition \ref{l04}
with $a=m_k$, $b=\sigma_{k+1}$, $F=G$ and $\vartheta =
\theta^{(1)}_\epsilon$. We need to estimate the two terms on the
right-hand side of \eqref{08}.

On the one hand, since $S$ is monotone on $J^+_k$, by Lemma
\ref{l01},
\begin{equation*}
\sup_{x\in J^+_k}
\bb E^\epsilon_{x}\big[\, \tau(m_k, \sigma_{k+1})
\,\big] \,\leq\, \frac{\mf c_0}{\epsilon}
\end{equation*}
for all $\epsilon > 0$.

On the other hand, since
\begin{equation*}
\sup_{x \in [m_k, \sigma_{k+1} - \delta_{\epsilon}(k+1)]}  S(x)
\,=\, S(\sigma_{k+1} - \delta_{\epsilon}(k+1)) \;, \quad
\inf_{x \in [\sigma_{k+1} - \eta, , \sigma_{k+1}] } S(x)
\,=\, S(\sigma_{k+1} - \eta)\;, 
\end{equation*}
for all $\eta < \sigma_{k+1} - m_k$, as $\sigma_{k+1} - m_k<1$, by
Lemma \ref{l02} with
$\eta= (1/\sqrt{2}) \, \delta_{\epsilon}(k+1)$, 
\begin{equation*}
\sup_{x\in J^+_k}
\bb P^\epsilon_{x}\big[\,\tau(\sigma_{k+1})\; < \;
\tau(m_{k})\,\big] \,\leq\,
\frac{\sqrt{2}}{\delta_\epsilon (k+1)}\, e^{-r/\epsilon}\;, 
\end{equation*}
where
$r= S(\sigma_{k+1} - \eta) - S(\sigma_{k+1} -
\delta_{\epsilon}(k+1))$.  By a Taylor expansion and the definition
\eqref{14} of $\delta_{\epsilon}(k+1)$,
$r = (1/2) (1+o(\epsilon)) \, \epsilon \, \log \epsilon^{-1} $. In
this formula and below, $\color{blue} o(\epsilon)$ represents an
expression which depends only on $\mss a(\cdot)$, $\mss b(\cdot)$ and
converges to $0$ as $\epsilon\to 0$. Thus,
\begin{equation*}
\sup_{x\in J^+_k}
\bb P_{x}^\epsilon\big[\, \tau(\sigma_{k+1})\; < \;
\tau(m_{k})\,\big] \,\leq\,
\frac{2 \sqrt{2}}{\delta_\epsilon (k+1)}\, \sqrt{\epsilon} 
\end{equation*}
for all $\epsilon$ sufficiently small.  To complete the proof of the
assertion of the lemma for $x\in J^+_k$, it remains to recall the
statement of Proposition \ref{l04}. The same argument applies to
$J^-_k$.
\end{proof}

\subsection*{Convergence to $f$}

In this subsection, we prove \eqref{06}.
Let $f_\epsilon \colon \ms S_1 \to \bb R$ be given by
\begin{equation}
\label{82}
{\color{blue} f_\epsilon (\ms M_1(k))} \,:=\, \phi_\epsilon(m_k)\;.
\end{equation}
We wish to prove that $f_\epsilon(\cdot)$ converges to $f(\cdot)$
pointwisely. By Proposition \ref{l04}, $f_\epsilon(\cdot)$ is
uniformly bounded by $(1/\lambda)\, \Vert g\Vert_\infty$. Thus, since,
by Lemma \ref{l30}, the reduced resolvent equation \eqref{42} has a
unique bounded solution, it is enough to show that all limit points of
the sequence $f_{\epsilon}(\cdot)$ solve the resolvent equation.

\begin{proposition}
\label{l05}
For all $k\in \bb Z$,
\begin{equation*}
\lim_{\epsilon \to 0} 
|\, f_\epsilon(\ms M_1(k))  - f(\ms M_1(k))\, | \,=\,  0\,,
\end{equation*}
where $f(\cdot)$ is the unique bounded solution of the reduced
resolvent equation \eqref{42}.
\end{proposition}

To prove the convergence of the sequence $f_\epsilon(\cdot)$, we
construct a test function $Q_{\epsilon}^{1,k}$ with support contained
in an interval $(a,b)$. As $\phi_{1,\epsilon}$ is a weak solution of the
resolvent equation \eqref{41},
\begin{equation}
\label{17}
\begin{aligned}
& \lambda  \, \int_{\bb R} Q_{\epsilon}^{1,k} (x)\,\phi_{1,\epsilon} (x)  \, 
e^{-S(x)/\epsilon}   \, dx
\,+\, \epsilon \, \theta^{(1)}_{\epsilon}\,
\int_{\bb R} [\, \partial_x \, (\mss a \, Q_{\epsilon}^{1,k}
)\, ] (x)
\, (\partial_x \phi_{1,\epsilon}) (x) \,  e^{-S(x)/\epsilon}   \, dx
\\
& \quad \,=\, \int_{\bb R}   Q_{\epsilon}^{1,k}(x)\,  G(x)  \,
e^{-S(x)/\epsilon}    \, dx\;.
\end{aligned}
\end{equation}
Taking the limit as $\epsilon\to 0$ in this identity, we show that any
limit point of the sequence $f_\epsilon(\cdot)$ solves the reduced
resolvent equation \eqref{42}. 

\smallskip\noindent {\it A test function}.  Let
$\color{blue} h_{k,\epsilon}\colon \bb R \to [0,1]$, $k\in \bb Z$, be
the equilibrium potential defined as the solution of the Poisson
equation
\begin{equation*}
\begin{cases}
\ms L_{\epsilon}u = 0, &\text{if } x\not\in \ms E \,
\\
u(x) = 1 &\text{if } x \in \ms E(m_k)\,, 
\\
u(x) = 0 &\text{if } x \in \ms E \setminus  \ms E (m_k) \;,
\end{cases}
\quad\text{where} \quad {\color{blue} \ms E}\,:=\, \bigcup_{k\in \bb Z} \ms
E(m_k)\;. 
\end{equation*}
Clearly the equilibrium potential $h_{k,\epsilon}$ has a stochastic
representation as
$h_{k,\epsilon} (x) = \bb P^\epsilon_x [ \tau_\epsilon (\ms E(m_k)) <
\tau_\epsilon ( \ms E \setminus \ms E(m_k))\,]$.  The test function
$Q_{\epsilon}^{1,k}$ is an approximation of $h_{k,\epsilon}
(\cdot)$. Its precise definition requires some notation.

Recall from \eqref{60} the definition of the length
$\eta_k(\epsilon)$. Denote by $C_{\epsilon}(k)$,
$\Lambda_{\epsilon}(k)$, $A_{\epsilon}(k)$ the \emph{core} sets given
by 
\begin{equation}
\label{60}
\begin{gathered}
{\color{blue} C_{\epsilon}(k)} \, :=\,  [\sigma_{k} +
\delta_{\epsilon}(k) + \eta_{\epsilon}(k) \,,\, \sigma_{k+1} -
\eta_{\epsilon}(k+1) - \delta_{\epsilon}(k+1) ],
\\
{\color{blue}  \Lambda_{\epsilon}(k)} :=
[\sigma_{k} - \delta_{\epsilon}(k),
\sigma_{k}+\delta_{\epsilon}(k)]\;, 
\quad
{\color{blue}  A_{\epsilon}(k) }:= \Lambda_ {\epsilon}(k) \cup
C_{\epsilon}(k) \cup \Lambda_ {\epsilon}(k+1)\;;
\end{gathered}
\end{equation} 
the \emph{binding} sets 
\begin{equation}
\label{gluing:interv} 
\begin{gathered}
{\color{blue} B_ {\epsilon}^{-}(k)} \,:=\,  (\sigma_{k}
- \delta_{\epsilon}(k) - \eta_{\epsilon}(k)
\,,\, \sigma_{k}-\delta_{\epsilon}(k)), \quad
{\color{blue} B_{\epsilon}^{+}(k)} \,:=\,  (\sigma_{k}+\delta_{\epsilon}(k),
\sigma_{k}+ \delta_{\epsilon}(k) + \eta_{\epsilon}(k) ) 
\\
{\color{blue}  B_ {\epsilon}(k)} \, := \,  B_ {\epsilon}^{-}(k) \cup
B_ {\epsilon}^{+}(k) \cup B_ {\epsilon}^{-}(k+1) \cup B_ {\epsilon}^{+}(k+1)  
\end{gathered}
\end{equation}
and the $m_{k}$-\emph{neighborhood}
\begin{equation}
\label{well_neighb}
{\color{blue}  \Omega_{\epsilon}(m_{k})} \, := \, A_{\epsilon}(k)\cup
B_{\epsilon}(k) \;.
\end{equation} 

This decomposition separates $\Omega_{\epsilon}(m_{k})$ into three
types of of sets, see Figure \ref{fig-f2}.  (i) A slowly growing
sequence of intervals, denoted by $C_\epsilon(k)$, that properly
contains $\ms E(m_{k})$, on which we know the asymptotic behaviour of
$\phi_\epsilon$ thanks to the local ergodicity results of the previous
subsection. (ii) Shrinking neighbourhoods of the unstable equilibrium
points $\sigma_{k}$, denoted by $\Lambda_ {\epsilon}(k)$. We have no
information on the behaviour of $\phi_{\epsilon}$ in the interior of
these sets. In these intervals, $\phi_\epsilon$ changes from values
close to $0$ to values close to $1$. (iii) Shrinking sets, denoted by
$B^\pm_{\epsilon}(k)$, needed to connect the definition of the
test function in the two previous class of sets.

\begin{figure}
\centering
\begin{tikzpicture}[scale=0.8]
\draw[rounded corners] (1,3) .. controls (2.5,5) .. (4,3);  
\draw[rounded corners] (4,3) .. controls (5.5, 1) .. (7, -1);
\draw[rounded corners] (7,-1) .. controls (8.5, -3) .. (10,-1);
\draw[rounded corners] (10,-1) -- (11.5,1);
\draw[rounded corners] (11.5,1) .. controls (13,2.8) .. (14.5, 1);

\fill(2.5,4.5)node[above]{$\sigma_{k}$};
\fill(13,2.4)node[above]{$\sigma_{k+1}$};
\fill(8.5,-2.6)node[below]{$m_k$};

\draw[dashed, purple](1,3) -- (1,-4);
\draw[dashed, purple](1.3,3.3) -- (1.3,-4);
\draw[dashed, purple](3.7,3.3) -- (3.7,-4);
\draw[dashed, purple](4,3) -- (4,-4);

\draw[dashed, purple](11.5,1)--(11.5,-4);
\draw[dashed,purple](11.8,1.3)--(11.8,-4);
\draw[dashed, purple](14.2,1.3)--(14.2,-4);
\draw[dashed,purple](14.5,1)--(14.5,-4);

\draw[dashed, teal](7,-1)--(7,-4);
\draw[dashed, teal](10,-1)--(10,-4);

\draw[solid, thick, black](1,-4)--(14.5,-4);
\draw[line width=0.5mm, purple](1,-4.2)--(1.3,-4.2) node[midway, below, purple, font=\tiny]{$B_\epsilon^-(k)$};
\draw[line width=0.5mm, brown](1.3,-3.8)--(3.7,-3.8) node[midway, above, brown, font=\tiny]{$\Lambda_\epsilon(k)$};
\draw[line width = 0.5mm, purple](3.7,-4.2)--(4,-4.2) node[midway, below, purple, font=\tiny]{$B_\epsilon^+(k)$};
\draw[line width = 0.5mm, cyan](4,-3.8)--(11.5,-3.8) node[midway, above, cyan, font=\tiny]{$C_\epsilon(k)$};
\draw[line width=0.5mm, teal](7,-4.2)--(10,-4.2) node[midway, below, teal, font=\tiny]{$\ms E(m_k)$};
\draw[line width=0.5mm, purple](11.5, -4.2)--(11.8,-4.2) node[midway, below, purple, font=\tiny]{$B_\epsilon^-(k+1)$};
\draw[line width=0.5mm, brown](11.8,-3.8)--(14.2,-3.8) node[midway, above, brown, font=\tiny]{$\Lambda_\epsilon(k+1)$};
\draw[line width=0.5mm, purple](14.2,-4.2)--(14.5,-4.2) node[midway, below, purple, font=\tiny]{$B_\epsilon^+(k+1)$};

\end{tikzpicture}
\caption{The core and the binding sets.}
\label{fig-f2}
\end{figure}
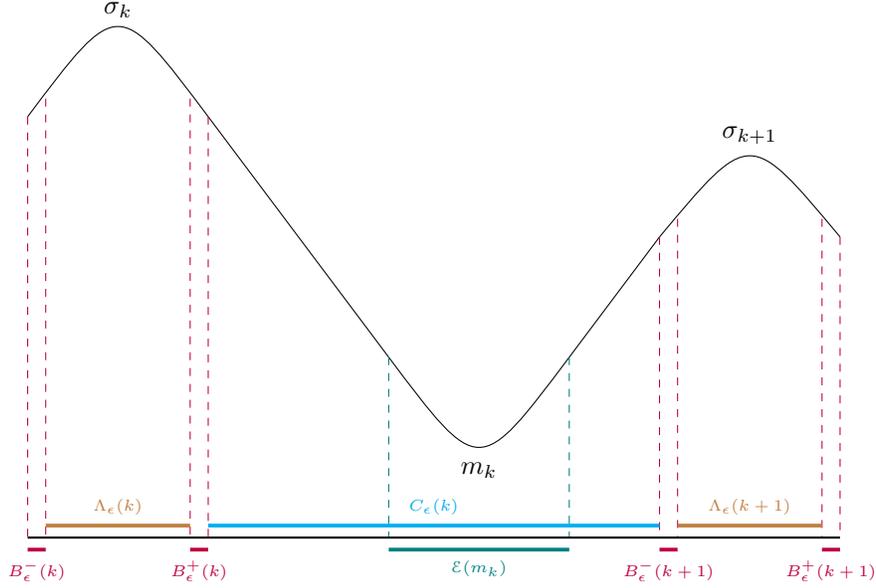

Fix $\sigma \in \mc W$, and denote by $\ms L_{\sigma, \epsilon}$ the
linearization of the generator $\ms L_\epsilon$ around $\sigma$
(centered at the origin and acting on $C^2(\bb R)$):
\begin{equation}
\label{18}
(\ms L_{\sigma, \epsilon} \psi) (x)\,=\, 
x \, \mss b'(\sigma) \, \psi '(x)
+ \epsilon \, \mss a(\sigma) \, \psi ''(x)\;.
\end{equation}
Let $\rho_{\sigma,\epsilon}\colon \bb R \to [0,1]$ be the solution of
the Poisson equation
\begin{equation*}
\ms L_{\sigma, \epsilon} u =0\,,\quad u(-\infty)=0\,,\quad
u(+\infty)=1\;.
\end{equation*}
A simple computation yields that
\begin{equation}
\label{13}
\rho_{\sigma, \epsilon}(x) \,=\,
\sqrt{\frac{-\, S''(\sigma)}{2\pi\epsilon}}
\int_{-\infty}^{x}e^{S''(\sigma) y^{2}/2 \epsilon } \,dy\,.
\end{equation}

The test function $Q_{\epsilon}^{1,k}$ is defined as
$Q_{\epsilon}^{1,k} = J_{\epsilon}^{1,k}/\mss a$, where
$ J_{\epsilon}^{1,k}$ is equal to $\rho_{\sigma_k, \epsilon}$ in the
neighbourhood $\Lambda_ {\epsilon}(k)$ of $\sigma_k$ and as
$1 - \rho_{\sigma_{k+1}, \epsilon}$ in $\Lambda_ {\epsilon}(k+1)$. More
precisely, recall that $B_{\epsilon}(k)$ stands for the binding
set. We first define the test function on
$\bb R \setminus B_{\epsilon}(k)$. Let
$J_{\epsilon}^{1,k}\colon \bb R \setminus B_{\epsilon}(k) \to \bb R$
be given by
\begin{equation}
\label{12}
{\color{blue} J_{\epsilon}^{1,k}(x) } \;:=\;
\begin{cases} \rho_{\sigma_k,\epsilon} (x-\sigma_k),
&\text{if } x\in \Lambda_{\epsilon}(k)
\\
1 , &\text{if } x\in C_{\epsilon}(k)
\\
(1 - \rho_{\sigma_{k+1},\epsilon})  (x- \sigma_{k+1}), &\text{if }
x\in \Lambda_ {\epsilon}(k+1)
\\
0, &\text{ otherwise. }
\end{cases}
\end{equation} 
We extend the definition of $J_{\epsilon}^{1,k}$ to $B_{\epsilon}(k)$
by linear interpolation. 

The test function $Q_{\epsilon}^{1,k} \colon \bb R \to \bb R$ is given
by $\color{blue} Q_{\epsilon}^{1,k} := J_{\epsilon}^{1,k}/\mss a$.  It
has compact support, it is continuous, piecewise smooth, and its
derivative is discontinuous at the boundary of $B_{\epsilon}(k)$.

Since $Q_{\epsilon}^{1,k}$ is piecewise smooth and continuous,
equation \eqref{17} holds. Multiply both sides of this equation
by $\exp\{S(m_k) /\epsilon\}/\sqrt{2\pi\, \epsilon}$, 
to get that
\begin{equation}
\label{20}
\begin{aligned}
& \frac{\lambda}{\sqrt{2\pi\, \epsilon}}\, \int_{\bb R} J_{\epsilon}^{1,k} (x)\,
\frac{1}{\mss a(x)} \, \phi_{\epsilon} (x) \, e^{-[S(x) - S(m_k)] /\epsilon} \, dx
\\
&\quad 
\,+\, \epsilon \, \theta^{(1)}_{\epsilon}\,
\frac{1}{\sqrt{2\pi\, \epsilon}}\, \int_{\bb R} (\, \partial_x \,
J_{\epsilon}^{1,k} )(x) \, (\partial_x \phi_{\epsilon}) (x) \,
e^{-[S(x)-S(m_k)]/\epsilon} \, dx
\\
& \quad \,=\, \frac{1}{\sqrt{2\pi\, \epsilon}}\,
\int_{\bb R} J_{\epsilon}^{1,k}(x)\, \frac{1}{\mss a(x)} \,  G(x) \,
e^{-[S(x)-S(m_k)] /\epsilon} \, dx\;.
\end{aligned}
\end{equation}

\subsection*{The resolvent equation}

We estimate each term of \eqref{20} separately.  By the definition of
$Q^{1,k}_{\epsilon}$, and since $\ms E(m_k) \subset C_\epsilon(k)$,
\begin{equation}
\label{32}
\begin{aligned}
& \frac{1}{\sqrt{2\pi\, \epsilon}}\, \int_{\bb R}
J_{\epsilon}^{1,k}(x)\, \frac{1}{\mss a(x)} \,  G(x) \,
e^{-[S(x)-S(m_k)] /\epsilon} \, dx
\\
&\quad \;=\; \frac{g(\ms M_1(k))}{\sqrt{2\pi\, \epsilon}}\, \int_{\ms E(m_k)}
\frac{1}{\mss a(x)} \,  
e^{-[S(x)-S(m_k)] /\epsilon} \, dx \,=\,  g(\ms M_1(k))
\, \frac{1}{\sqrt{2\pi}}\, \pi_1(m_k)
\, (1+ o_\epsilon(1)) \,,
\end{aligned}
\end{equation}
where $\pi_1(\cdot)$ has been introduced in \eqref{04}.  On the other
hand, since $\phi_{\epsilon}(\cdot)$ is uniformly bounded and
$J_{\epsilon}^{1,k}$ vanishes on $\Omega_\epsilon (k)^c$, by
Proposition \ref{l06},
\begin{equation}
\label{22b}
\begin{aligned}
& \frac{\lambda}{\sqrt{2\pi\, \epsilon}}\, \int_{\bb R} J_{\epsilon}^{1,k} (x)\,
\frac{1}{\mss a(x)} \, \phi_{\epsilon} (x) \, e^{-[S(x) - S(m_k)] /\epsilon} \, dx
\\
&\quad \;=\;
\frac{\lambda}{\sqrt{2\pi\, \epsilon}}\, \int_{\Omega_\epsilon (k)}  J_{\epsilon}^{1,k} (x)\,
\frac{1}{\mss a(x)} \, \phi_{\epsilon} (x) \, e^{-[S(x) - S(m_k)]
/\epsilon} \, dx
\\ & \quad 
\,=\, \lambda\, f_{\epsilon}(\ms M_1(k))\, \frac{1}{\sqrt{2\pi}}\,
\pi_1(m_k) \, + \, o_\epsilon(1) \;.
\end{aligned}
\end{equation}

We turn to the second term on the left-hand side of \eqref{20}. Recall
the definition of $\sigma_1(j,j+1)$, $j\in \bb Z$, introduced in
\eqref{04}. Since the test function $J_{\epsilon}^{1,k}$ vanishes on
$\Omega_\epsilon (k)^c$, the integral can be taken over the set
$\Omega_{\epsilon} (m_{k})$.

\begin{lemma}
\label{l03}
For each $k\in \bb Z$,
\begin{equation*}
\begin{aligned}
& \epsilon \, \theta^{(1)}_{\epsilon}\, \frac{1}{\sqrt{2\pi\,\epsilon}}\,
\int_{\Omega_{\epsilon} (m_{k})} (\, \partial_x \,
J_{\epsilon}^{1,k} )(x) \, (\partial_x \phi_{\epsilon}) (x) \,
e^{-[S(x)-S(m_k)]/\epsilon} \, dx
\\
&\quad \,=\, -\, \frac{1}{\sqrt{2\pi}}\, \sum_{i=\pm 1}
\frac{1}{\sigma_1(k,k+i)}\, \mtt 1 \{h^i_k = \mf h_1\}
\, [\, f_{\epsilon}(\ms M_1(k+i))- f_{\epsilon}(\ms M_1(k))\,]
\,+\, o_\epsilon(1)  \;. 
\end{aligned}
\end{equation*}
\end{lemma}

The proof is carried out estimating the integral in each region
forming $\Omega_{\epsilon} (m_{k})$. By definition,
$\partial_x \, J_{\epsilon}^{1,k}$ vanishes on $C_\epsilon
(k)$. Lemma \ref{as2} takes care of the integral on the set
$B_\epsilon(k)$, where $J^{1,k}_{\epsilon}$ is linear, and Lemma
\ref{as1} on the set
$\Lambda_{\epsilon}(k) \cup \Lambda_{\epsilon}(k+1)$.

\begin{lemma}
\label{as2}
For all $k\in \bb Z$,
\begin{equation*}
\lim_{\epsilon\to 0}
\Big|\, \frac{\epsilon\, \theta_{\epsilon}^{(1)}}{\sqrt{2\pi\, \epsilon}}\,
\int_{B_\epsilon  (k)} (\, \partial_x \,
J_{\epsilon}^{1,k} )(x) \, (\partial_x \phi_{\epsilon}) (x) \,
e^{-[S(x)-S(m_k)]/\epsilon} \, dx  \, \Big| \,=\, 0\;.
\end{equation*}
\end{lemma}

\begin{proof}
To fix ideas, we estimate the integral on the set
$B^-_\epsilon(k)$. By definition of $J^{1,k}_\epsilon$,
\eqref{13}, \eqref{14}, and an elementary Gaussian estimate,
\begin{equation*}
J^{1,k}_{\epsilon} (\sigma_k - \delta_\epsilon(k)) \,=\, 
\rho_{\sigma_k,\epsilon}(- \delta_\epsilon(k)) \,\le\,
\sqrt{\frac{-\, \epsilon}{2\pi S''(\sigma_k)} }\,
\frac{1}{\delta_\epsilon(k)} \, 
e^{ S''(\sigma_k) \delta_\epsilon( k)^2/2 \epsilon }\;.
\end{equation*}
Therefore, on $B^-_\epsilon(k)$, since $J^{1,k}_\epsilon$ has been
defined by linear interpolation, and the length of the binding
interval $B^-_\epsilon(k)$ is $\eta_\epsilon(k)$, $\partial_x \,
J^{1,k}_{\epsilon}$ is constant and 
\begin{equation*}
\big|\, (\partial_x \,  J^{1,k}_{\epsilon}) (x) \,\big|\,\le\,
\frac{1}{\eta_\epsilon(k)}\, 
\sqrt{\frac{- \, \epsilon}{2\pi S''(\sigma_k)}} \,
\frac{1}{\delta_\epsilon(k)}\, 
e^{S''(\sigma_k) \, \delta_\epsilon(k)^2/2 \epsilon }\;.
\end{equation*}

On the other hand, as $\phi_\epsilon$ belongs to $\mc H^1_{\rm loc}$,
an integration by parts yields that
\begin{equation*}
\begin{aligned}
\int_{B_\epsilon  (k)} (\partial_x \phi_{\epsilon}) (x) \,
e^{-[S(x)-S(m_k)]/\epsilon} \, dx
\,= & \, \big[\, \phi_{\epsilon} (\mtt r) - \phi_{\epsilon} (\mtt l)\,\big]
\, e^{-[S(\mtt r)-S(m_k)]/\epsilon}
\\
\,+ & \, \frac{1}{\epsilon}\, 
\int_{B_\epsilon  (k)} \big[ \, \phi_{\epsilon} (x) - \phi_{\epsilon}
(\mtt l)\, \big] \,
S'(x)\, e^{-[S(x)-S(m_k)]/\epsilon} \, dx
\;,
\end{aligned}
\end{equation*}
where $\mtt l$, $\mtt r$ represent the endpoints of the interval
$B_\epsilon (k)$.  By Proposition \ref{l06}, as the interval
$B_\epsilon (k)$ is contained in the set $J_k$,
$\sup_{x\in B_\epsilon (k)} |\phi_{\epsilon} (x) -
\phi_{\epsilon} (\mtt l)| \le C_0 \,
\sqrt{\epsilon}/\delta_\epsilon(k)$ for some finite constant $C_0$
which depends on $g(\cdot)$, $\lambda$, $\mss a(\cdot)$,
$\mss b(\cdot)$, but not on $\epsilon$, and which may change from line
to line. Since
$|S'(x)| = | S'(x) - S'(\sigma_k)|\le C_0\, |x-\sigma_k| \le C_0\,
(\delta_\epsilon(k) + \eta_\epsilon(k))$, the absolute value of the
previous expression is bounded by
\begin{equation*}
C_0 \, \frac{\sqrt{\epsilon}}{\delta_\epsilon(k)} \, \Big\{\, 1  \,+\,
\frac{1}{\epsilon}\, \big[\, \delta_\epsilon(k) + \eta_\epsilon(k)\, \big]\,
\eta_\epsilon(k) \Big\} \, \sup_{x\in B_\epsilon  (k)}
e^{-[S(x)-S(m_k)]/\epsilon}\,.
\end{equation*}

Recall that we have a term
$\exp\{S''(\sigma_k) \, \delta_\epsilon(k)^2/2 \epsilon \}$ coming
from the derivative of $J^{1,k}_\epsilon$.  Since
$\mf h_1 \le S(\sigma_k) - S(m_k)$,
\begin{equation*}
\theta_{\epsilon}^{(1)}\,
e^{S''(\sigma_k) \, \delta_\epsilon(k)^2/2 \epsilon } \,
\sup_{x\in B_\epsilon  (k)} e^{-[S(x)-S(m_k)]/\epsilon}
\,\le\,
e^{S''(\sigma_k) \, \delta_\epsilon(k)^2/2 \epsilon } \,
\sup_{x\in B_\epsilon  (k)} e^{-[S(x)-S(\sigma_k)]/\epsilon}
\end{equation*}
As the first derivatives of the functions $\mss a(\cdot)$,
$\mss b(\cdot)$ are Lipschitz continuous, and
$|x-\sigma_k|\le \delta_\epsilon(k) + \eta_\epsilon(k)$ for all
$x\in B^-_\epsilon(k)$, by a Taylor expansion,
\begin{equation*}
S''(\sigma_k) \, \frac{\delta_\epsilon(k)^2}{2 \epsilon}
\,+\, \frac{1}{\epsilon}\, \big\{\, S(\sigma_k) - S(x)\,\big\}
\;\le\; \frac{C_0}{\epsilon} \, \delta_\epsilon(k) \,
\{\eta_\epsilon(k)  + \delta_\epsilon(k)^2\}
\end{equation*}
for all $x\in B^-_\epsilon(k)$. By definition of
$\delta_\epsilon(k)$, $\eta_\epsilon(k)$, this quantity is bounded
uniformly in $\epsilon$.

Since $\eta_\epsilon(k) \le \delta_\epsilon(k)$, recollecting all
previous estimates, yields that expression appearing in the statement
of the lemma is bounded by
\begin{equation*}
C_0\, \sqrt{\epsilon}\; \frac{\sqrt{\epsilon}}
{\delta_\epsilon(k) \, \eta_\epsilon(k)}  \;
\frac{\sqrt{\epsilon}}{\delta_\epsilon(k)} \, \Big\{\, 1  \,+\,
\frac{1}{\epsilon}\, \delta_\epsilon(k)\,
\eta_\epsilon(k) \Big\} \,\le\,
C_0\, \frac{\sqrt{\epsilon}}{\delta_\epsilon(k)}\;,
\end{equation*}
where we used the definition of $\eta_\epsilon(k)$.  By definition of
$\delta_\epsilon(k)$, this quantity vanishes as $\epsilon\to 0$, which
completes the proof of the lemma.
\end{proof}

\begin{remark}
One may obtain a much better estimate for
$\sup_{x\in B_\epsilon (k)} |\phi_{\epsilon} (x) - \phi_{\epsilon}
(\mtt l)|$. It is enough to repeat the proof of Proposition \ref{l06}
with $m_k$, $a$, $b$ replaced by $\mtt l$, $\mtt l$, $\sigma_k$,
respectively, and setting
$\eta = \sqrt{1-\kappa} \, \delta_\epsilon(k)$ for some $0<\kappa<1$.
The prof of \cite[Theorem 6.3.1]{le} also provides bounds for the
$L^2(B_\epsilon(k))$ norm of $\partial_x \phi_\epsilon$ which can
replace the bound for
$\sup_{x\in B_\epsilon (k)} |\phi_{\epsilon} (x) - \phi_{\epsilon}
(\mtt l)|$ used in the previous proof.
\end{remark}

We turn to the set $\Lambda_\epsilon (k) \cup \Lambda_\epsilon
(k+1)$. 

\begin{lemma}
\label{as1}
For all $k\in \bb Z$, 
\begin{equation*}
\begin{aligned}
& \frac{\epsilon\, \theta_{\epsilon}^{(1)}}{\sqrt{2\pi\,\epsilon}}\,
\int_{\Lambda_{\epsilon}(k) \cup \Lambda_{\epsilon}(k+1)} (\, \partial_x \,
J_{\epsilon}^{1,k} )(x) \, (\partial_x \phi_{\epsilon}) (x) \,
e^{-[S(x)-S(m_k)]/\epsilon} \, dx
\\
&\quad \,=\, -\,  \frac{1}{\sqrt{2\pi}} \sum_{i=\pm 1} 
\frac{1}{\sigma_1(k,k+i)}\, \mtt 1 \{h^i_k = \mf h_1\}
\, \big[\, f_{\epsilon}(\ms M_1(k+i))-f_{\epsilon}(\ms M_1 (k))\,\big]
\,+\, o_\epsilon(1)  \;.
\end{aligned}
\end{equation*}
\end{lemma}

\begin{proof}
Consider the interval $\Lambda_{\epsilon}(k+1)$. A similar argument
applies to $\Lambda_{\epsilon}(k)$.  Recall from \eqref{13} the
definition of $\rho_{\sigma_k,\epsilon}$, and that
$J_{\epsilon}^{1,k} (x)= 1-\rho_{\sigma_{k+1},\epsilon}
(x-\sigma_{k+1})$ on $\Lambda_{\epsilon}(k+1)$. Integrate by
parts. The contribution of the boundary terms is given by
\begin{equation*}
-\,  \frac{ \sqrt{ -\, S''(\sigma_{k+1})}}  {2\pi} \,
\theta_{\epsilon}^{(1)} \,
\sum_{i = \pm 1} i\,  \phi_{\epsilon}(x_i) \, e^{(1/2\epsilon)\, 
S''(\sigma_{k+1}) \,  (\sigma_{k+1}-x_i)^2}
e^{(1/\epsilon) (S(m_{k})-S(x_i))}\,,
\end{equation*}
where $x_i = \sigma_{k+1} + i \delta_{\epsilon}(k+1)$. If
$\mf h_1 < S(\sigma_{k+1})-S(m_{k})$, this expression vanishes
exponentially fast as $\epsilon\to 0$. On the other hand, if
$\mf h_1 = S(\sigma_{k+1})-S(m_{k})$, by the definition of
$\delta_{\epsilon}(k+1)$ given in \eqref{14} and a Taylor expansion,
\begin{equation}
\label{19}
\Big|\, 
\frac{1}{\epsilon} \,
\big[\, S(\sigma_{k+1})-S(\sigma_{k+1} \pm \delta_{\epsilon}(k+1)) \,\big]
\,+\, \frac{1}{2\epsilon} \, S''(\sigma_{k+1}) \,
\delta_{\epsilon}(k+1)^2 \,\Big| \,\le \, C_0\,
\frac{\delta_{\epsilon}(k+1)^3}{\epsilon} \,=\, o_\epsilon(1)\;.
\end{equation} 
Hence, by Proposition \ref{l06} which establishes that
$\phi_{\epsilon}(\sigma_{k+1} + \delta_{\epsilon}(k+1))$,
$\phi_{\epsilon}(\sigma_{k+1} - \delta_{\epsilon}(k+1))$ are close to
$\phi_{\epsilon}(m_{k+1})$, $\phi_{\epsilon}(m_{k})$, respectively,
and by the formula for the jump rates presented in \eqref{73},
the expression in the penultimate displayed equation is equal to
\begin{equation*}
\begin{aligned}
& -\, (1+o_\epsilon(1))\,
\frac{ \sqrt{ -\, S''(\sigma_{k+1}) }}  {2\pi} \, \mtt 1 \{h^+_k = \mf h_1\}
(\phi_{\epsilon}(m_{k+1})-\phi_{\epsilon}(m_{k})) \\
&\quad
\;=\; -\, (1+o_\epsilon(1))\, \frac{1}{\sqrt{2\pi}}\,
\frac{1}{\sigma_1(k,k+1)}\,  \mtt 1 \{h^+_k = \mf h_1\}\, 
\big[\, f_{\epsilon}(\ms M_1(k+1))-f_{\epsilon}(\ms M_1(k))\,\big] \;. 
\end{aligned}
\end{equation*} 

We turn to the bulk part in the integration by parts formula. It is
given by
\begin{equation}
\label{31}
-\, \frac{\theta_{\epsilon}^{(1)}}{\sqrt{2\pi\,\epsilon}}\,
\int_{\sigma_{k+1}-\delta_{\epsilon}(k+1)}^{\sigma_{k+1}+\delta_{\epsilon}(k+1)}
\phi_{\epsilon}(x)
\big\{\, \epsilon\, (\partial^2_x \, J_{\epsilon}^{1,k} )(x)
\,+\,  \frac{\mss b (x)}{\mss a (x)}
\, (\partial_x \, J_{\epsilon}^{1,k}) (x)\, \big\}\,
e^{-[S(x)-S(m_k)]/\epsilon} \, dx\;.
\end{equation}
This equation explains the construction of the test function
$J_{\epsilon}^{1,k}$, as we show below that the expression inside
braces is small. We estimate individually each term of the previous
integral.

By \eqref{54}, the solution of the resolvent equation
$\phi_{\epsilon}$ is bounded by a constant, uniformly in
$\epsilon$. By the definition \eqref{12} of
$J_{\epsilon}^{1,k} (\cdot)$, and since $S'(\sigma_{k+1})=0$,
the expression inside braces is equal
to
\begin{equation*}
\sqrt{\frac{-\, S''(\sigma_{k+1})}{2\pi\epsilon}}
\, \Big\{\, S'(x) - S'(\sigma_{k+1})
\,-\, S''(\sigma_{k+1}) (x-\sigma_{k+1})\,\Big\}\,
e^{S''(\sigma_{k+1}) (x-\sigma_{k+1})^{2}/2 \epsilon }\;.
\end{equation*}
Since $S''(\cdot)$ is Lipschitz-continuous and
$|x-\sigma_{k+1}| \le \delta_\epsilon (k+1)$, the expression inside
braces is bounded by $C_0 \delta_\epsilon (k+1)^2$.

Since $\theta_{\epsilon}^{(1)} \le \exp\{ [S(\sigma_{k+1}) -
S(m_{k})]/\epsilon \}$, on $\Lambda_\epsilon(k+1)$,
the Taylor expansion performed in \eqref{19} yields that
\begin{equation*}
\theta_{\epsilon}^{(1)} \, e^{S''(\sigma_{k+1}) (x-\sigma_{k+1})^{2}/2 \epsilon }\,
e^{-[S(x)-S(m_k)]/\epsilon} 
\;\le\;  e^{C_0 \delta_\epsilon (k+1)^3/\epsilon}\,\le\, C_0\,.
\end{equation*}

Since the length of the interval $\Lambda_\epsilon(k+1)$ is
$\delta_\epsilon (k+1)$, putting together the previous estimates,
yields that the absolute value of the integral \eqref{31} is bounded
by
\begin{equation*}
C_0 \, \frac{1}{\epsilon}\ \, \delta_\epsilon (k+1)^3 \;. 
\end{equation*}
By definition of $\delta_\epsilon (k+1)$, this expression vanishes as
$\epsilon\to 0$.  This completes the proof of the assertion.
\end{proof}

We have all the elements to prove the convergence of the sequence
$f_{\epsilon} (\cdot)$. 

\begin{proof}[Proof of Proposition \ref{l05}.] 

On the one hand, by \eqref{54}, the sequence
$(f_{\epsilon})_{\epsilon>0}$ is uniformly bounded. On the other hand,
by \eqref{20}, \eqref{32}, \eqref{22b}, Lemma \ref{l03}, and the
definition \eqref{72b} of the generator $\bb L_1$, any pointwise limit
of the sequence $f_\epsilon(\cdot)$ is a bounded solution of the
reduced resolvent equation \eqref{42}. To complete the proof, it
remains to recall the uniqueness of bounded solutions stated in Lemma
\ref{l30}.
\end{proof}

\begin{proof}[Proof of Theorem \ref{mt4} for $p=1$.]
The proof is a straightforward consequence of Propositions \ref{l06}
and \ref{l05}.
\end{proof}

\subsection*{Reduced resolvent equation}
\label{sec-ap2}

As we did not find a reference, we prove in this subsection the
reduced resolvent equation solution uniqueness.

\begin{lemma}
\label{l30}
Fix $1\le p\le \mf q$, and a bounded function
$g\colon \ms S_p \to \bb R$. For every $\lambda>0$, the reduced
resolvent equation \eqref{42} has a unique bounded solution.
\end{lemma}

\begin{proof}
We consider three cases: $1\le p<\mf q$, $p=\mf q$ and $S(0)=S(1)$,
$p=\mf q$ and $S(0)\neq S(1)$.  Assume that $1\le p<\mf q$.  Recall from
\eqref{83} the definition of the generator $\bb L_p$, and that the
$\bb X_p(\cdot)$-closed irreducible classes are the sets
$\ms M_{p+1}(j)$, $j\in \bb Z$.

Fix $j\in \bb Z$. Denote by $\pi_{p,j}$ the probability measure
defined on the set
$\{ \ms M_p(k) : \ms M_p(k) \subset \ms M_{p+1}(j)\}$ by
$\pi_{p,j} (\ms M_p(k)) = \pi(\ms M_p(k)) / \pi(\ms M_{p+1}(j))$,
where $\pi(\cdot)$ has been introduced in \eqref{104}.  By the
definition of the jump rates $R_p$, $\pi_{p,j}$ satisfies the detailed
balance conditions. Since $\ms M_{p+1}(j)$ is a closed irreducible
class the resolvent equation \eqref{42} is autonomous on this set.  As
the set $\ms M_{p+1}(j)$ is finite, it is immediate to show that the
solution restricted to $\ms M_{p+1}(j)$ is unique. This characterise
the solution $f$ on the closed irreducible classes.

By \eqref{50} jumps are allowed only to nearest-neighbor sets. As each
transient state is surrounded by closed irreducible classes, which are
at finite distance, an elementary recursive argument permits to extend
the solution uniqueness to all $\bb X_p$-transient equivalent classes
between two irreducible classes.

Suppose that $p=\mf q$ and $S(1) = S(0)$. By Claim 2.C, stated just
before Theorem \ref{mt5},
${\color{blue} R_{\mf q} (k,k\pm 1)} := R_{\mf q} (\ms M_{\mf q} (k),
\ms M_{\mf q} (k\pm 1)) > 0$ for all $k\in \bb Z$.  We have to show
that the constant function equal to $0$ is the unique bounded solution
to \eqref{42} with $g=0$. Denote by $f$ a solution. Suppose that
$f(\ms M_{\mf q} (k)) \neq 0$. Assume, without loss of generality,
that it is positive. Since $f$ is the solution of the resolvent
equation with $g=0$,
\begin{equation}
\label{105}
R_{\mf q} (k,k + 1) \, \{F(k+1) - F(k)\}
\,+\, R_{\mf q} (k,k - 1) \, \{F(k-1) - F(k)\}
\,=\, \lambda\, F(k) \;,
\end{equation}
where $F(k) = f(\ms M_{\mf q} (k))$,
$R_{\mf q} (k,k \pm 1) = R_{\mf q} (\ms M_{\mf q} (k), \ms M_{\mf q}
(k\pm 1))$, $k\in \bb Z$.

Suppose that $F(k-1) \le F(k)$. In this case, since $F(k)>0$, and
$\pi_{\mf q} (\cdot)$ is a reversible measure,
\begin{equation*}
F(k+1) - F(k) \,>\,
\frac{R_{\mf q} (k,k - 1)}{R_{\mf q} (k,k + 1)}
\, \{F(k) - F(k-1)\} \,=\,
\frac{\pi_{\mf q} (k) \, R_{\mf q} (k,k - 1)}
{\pi_{\mf q} (k+1) \, R_{\mf q} (k+1,k)}
\, \{F(k) - F(k-1)\} \;.
\end{equation*}
In particular, $F(k+1) - F(k)>0$. Mind the inequality has been changed
to a strict inequality. Iterating this argument yields that
\begin{equation*}
F(k+j+1) - F(k+j) \,>\,
\frac{\pi_{\mf q} (k) \,  R_{\mf q} (k+1,k)}
{\pi_{\mf q} (k+j+1) \,  R_{\mf q} (k+j+1,k + j)}
\, \{F(k+1) - F(k)\} \,>\, 0 
\end{equation*}
for all $j\ge 0$. Since the jump rates and the measure $\pi_{\mf q}
(\cdot)$ are bounded above by a finite constant and below by a
positive constant, $F(k+j)\to \infty$ as $j\to\infty$, which
contradicts the assumption that $f(\cdot)$ is bounded.

If $F(k-1) > F(k)$, the same argument yields that $F(k-j)\to \infty$
as $j\to\infty$, which completes the proof of the lemma in the case
$p=\mf q$ and $S(1) = S(0)$.  

Suppose that $p=\mf q$ and $S(1) \neq S(0)$. To fix ideas, assume
without loss of generality that $S(1)<S(0)$. By Claims 2.A and 2.B stated
just before Theorem \ref{mt5}, 
$R_{\mf q} (\ms M_{\mf q} (k), \ms M_{\mf q} (k+1)) > 0$ for all
$k\in \bb Z$ and
$R_{\mf q} (\ms M_{\mf q} (j), \ms M_{\mf q} (j-1)) = 0$ for some
$j\in\bb Z$ (and therefore, for all $j+n\mf u_{\mf q}$, $n\in \bb Z$).

Fix $k_0 \in \bb Z$ such that
$R_{\mf q} (\ms M_{\mf q} (k_0), \ms M_{\mf q} (k_0-1)) = 0$. Assume
that $F(k_0) \neq 0$, say $F(k_0) > 0$. Since $f(\cdot)$ is the
solution of the resolvent equation,
\begin{equation}
\label{106}
R_{\mf q} (k_0,k_0 + 1) \, \{F(k_0+1) - F(k_0)\}
\,=\, \lambda\, F(k_0) \,>\, 0\;,
\end{equation}
so that $F(k_0+1) > F(k_0)$.
Let $k_1$ the next integer to the right of $k_0$ such that
$R_{\mf q} (\ms M_{\mf q} (k_1), \ms M_{\mf q} (k_1-1)) = 0$.
By the resolvent equation \eqref{105} applied inductively, $F(k)$ is
seen to be increasing in the interval $k_0\le k\le k_1$. Hence, as
$R_{\mf q} (\ms M_{\mf q} (k_1), \ms M_{\mf q} (k_1-1)) = 0$, by
the bound \eqref{106} for $k_1$ instead of $k_0$, 
\begin{equation*}
R_{\mf q} (k_1,k_1 + 1) \, \{F(k_1+1) - F(k_1)\}
\,=\, \lambda\, F(k_1) \,\ge\, \lambda\, F(k_0)\;.
\end{equation*}
Adding the estimates, as $k_1\ge k_0+1$, yields that
\begin{equation*}
F(k_1+1) - F(k_0) \,\ge\, \lambda\,
\Big\{ \frac{1}{R_{\mf q} (k_1,k_1 + 1)} +
\frac{1}{R_{\mf q} (k_0,k_0 + 1)} \,\Big\} F(k_0)\;.
\end{equation*}
Thus, $F(k_0+j) \to\infty$ as $j\to\infty$, in contradiction with the
boundedness of $f(\cdot)$.

If $F(k_0) = 0$ for all $k_0 \in \bb Z$ such that
$R_{\mf q} (\ms M_{\mf q} (k_0), \ms M_{\mf q} (k_0-1)) = 0$. By
\eqref{106}, $F(k_0+1) = 0$. Then, by \eqref{105} for $k=k_0+1$,
$F(k_0+2) = 0$. By induction, one concludes that $F(k) = 0$ for all
$k\ge k_0$. By Postulate $\mc P_9(\mf q)$, there exists a sequence
$(k_j : j\ge 1)$ such that $k_j\to-\infty$,
$R_{\mf q} (\ms M_{\mf q} (k_j), \ms M_{\mf q} (k_j-1)) = 0$. This
completes the proof of the lemma.
\end{proof}

\section{Proof of Theorem \ref{mt4} for $p>1$}
\label{sec5}

Following the approach for the first time scale presented in the
previous section, we prove Theorem \ref{mt4} in two steps. We first
show that the solution of the resolvent equation \eqref{41} is
asymptotically constant in each well and then that all limit points
are solution of the reduced resolvent equation \eqref{42}.

The proof is similar to the one presented in the previous section,
with one major difference. While for $p=1$, as soon as the diffusion
crossed the saddle point, we very high probability it lands in a new
well. This is not the case in general for $p>1$, since there might be
$S$-local minima separating two wells. Hence, crossing a saddle
point, the diffusion may land in a local minima which does not belong
to a level $p$ well. In particular, it may return to the original well
before hitting a new well. This possibility requires new estimates to
compute the magnitude of the transition times between two wells, see
Lemma \ref{l28} and Figure \ref{fig-f3}. 

% na figura 3 precisamos ilustrar este minimo.

Throughout this section, $p\in \{2, \dots, \mf q\}$ and a bounded
function $g\colon\ms S_p \to \bb R$ are fixed. Recall that we denote
by $\phi_{p, \epsilon}(\cdot)$ the solution of the resolvent equation
\eqref{41}.  By \eqref{54} and the definition of $G$,
\begin{equation}
\label{54b}
\sup_{x\in \bb R} |\phi_{p, \epsilon}(x)| \,\le\, \frac{1}{\lambda}\,
\max_{k\in \bb Z}  \big|\, g(\ms M_{p}(k)) \, \big| \;.
\end{equation}

\subsection*{Local Ergodicity}
 
Recall from \eqref{64} the definition of $\sigma^{p,\pm}_{j, j+1} $,
$j\in \bb Z$. For $k\in \bb Z$, let
\begin{equation}
\label{21}
{\color{blue} \delta^{p,-}_{k,k+1}}\,: = \, \delta^{p,-}_{k,k+1}
(\epsilon) \,=\,
\sqrt{\frac{2 \,\epsilon \log
\epsilon^{-1}}{-S''(\sigma^{p,-}_{k,k+1})}}\;,
\hspace{5mm}
{\color{blue} \delta^{p,+}_{k,k+1} } \,:=\, \delta^{p,+}_{k,k+1}
(\epsilon) \,=\,
\sqrt{\frac{2 \, \epsilon \log
\epsilon^{-1}}{- S''(\sigma^{p,+}_{k,k+1})}} \;\cdot
\end{equation}
Recall from \eqref{49b} that we denote by $m^{p,-}_{k}$ and
$m^{p,+}_{k}$ the leftmost and rightmost minima of $\ms M_{p}(k)$,
respectively.  Denote by $J^{p}_{k}, J^{p,+}_{k}, J^{p,-}_{k}$ the
subintervals of $\bb R$ given by
\begin{gather*}
{\color{blue}   J^{p,-}_{k} } \,:= \,
[\sigma^{p,+}_{k-1,k} + \delta^{p,+}_{k-1,k} \,,\,
m^{+}_{p,k}]\;, \quad 
{\color{blue}    J^{p,+}_{k}} \, := \, [m^{-}_{p,k} \,,\, \sigma^{p,-}_{k,k+1} -
\delta^{p,-}_{k,k+1}]\;, \quad
{\color{blue}   J^{p}_{k}} \,:= \, J^{p,-}_{k} \cup J^{p,+}_{k}\;.
\end{gather*}

\begin{lemma}
\label{l14}
There exists a finite constants $C_0$, $\epsilon_0>0$ depending only
on $S(\cdot)$, such that
\begin{equation*}
\sup_{x\in  J^p_k}
|\phi_{p,\epsilon}(x) - \phi_{p,\epsilon}(m)| \,\le\,
2\, \Vert g \Vert_\infty\, \Big\{\,
\frac{\mf c_0 \, \theta_\epsilon^{p-1}}{\epsilon\, \theta_\epsilon^p}
\,+\, \frac{C_0}{\lambda} \,
\frac{\sqrt{\epsilon}}
{\delta^{p,+}_{k-1,k}  \wedge \delta^{p,-}_{k,k+1}  }\,
\,\Big\}
\end{equation*}
for all $k\in\bb Z$, $\epsilon\le \epsilon_0$. In this formula,
$\mf c_0$ is the constant introduced in Lemma \ref{l01}.
\end{lemma}

\begin{proof}
Fix $k\in \bb Z$, and assume that $x\in J^{p,+}_k$. The same argument
applies to $x\in J^{p,-}_k$. The proof relies on Proposition \ref{l04}
with $a=m^-_{p,k}$, $b=\sigma^{p,-}_{k,k+1}$, $F=G$ and
$\vartheta = \theta^p_\epsilon$. We need to estimate the two terms on
the right-hand side of \eqref{08}.

We claim that there exists a finite constant $C_0$, independent of
$\epsilon$, and $\epsilon_0>0$. such that
\begin{equation}
\label{56}
\sup_{x\in J^{p,+}_{k}}
\bb P^\epsilon_{x} \big[\, \tau (\sigma^{p,-}_{k, k+1})\; < \;
\tau (m^{-}_{p,k})\,\big] \,\le \,
C_0 \, 
\frac{\sqrt{\epsilon}}{ \delta^{p,-}_{k,k+1}} 
\end{equation}
for all $\epsilon \le \epsilon_0$. 

To prove this assertion, recall from Lemma \ref{l02} that
\begin{equation*}
\sup_{x\in J^{p,+}_{k}}
\bb P^\epsilon_{x}\big[\,\tau (\sigma^{p,-}_{k,k+1})\; < \;
\tau (m^{-}_{p,k})\,\big] \,\leq\,
\frac{1}{\eta}   \,
\exp\Big\{ \,\big[ \,
\sup_{z\in [m^-_{p,k},x] } S(z) -
\inf_{y\in [\sigma^{p,-}_{k,k+1}-\eta, \sigma^{p,-}_{k,k+1}]} S(y)
\big]/\epsilon \,\Big\}
\end{equation*}
for all $x\in J^{p,+}_k$ because $x-a\le 1$. Here,
$\eta= (1/\sqrt{2}) \, \delta^{p,-}_{k,k+1}$.  By Lemma
\ref{l15}, for $\epsilon$ sufficiently small (so that
$\delta^{p,-}_{k,k+1}$ is small),
\begin{equation*}
\sup_{z\in [m^-_{p,k} \,,\, \sigma^{p,-}_{k,k+1} -
\delta^{p,-}_{k,k+1}]} S(z) \, = \,
S(\sigma^{p,-}_{k,k+1} - \delta^{p,-}_{k,k+1})\;.
\end{equation*}
As $S(\cdot)$ in increasing on
$[\sigma^{p,-}_{k,k+1}-\eta, \sigma^{p,-}_{k,k+1}]$, for $\epsilon$
sufficiently small, the left-hand side of the penultimate displayed
equation is less than or equal to
\begin{equation*}
\frac{1}{\eta}   \,
\exp\Big\{\,\Big[  S(\sigma^{p,-}_{k,k+1} - \delta^{p,-}_{k,k+1}) \,-\,
S(\sigma^{p,-}_{k,k+1} - \eta ) \Big]/\epsilon\, \Big\}\;.
\end{equation*}
By definition of $\eta$, and a second order Taylor expansion, as $S''$
is Lipschitz-continuous and periodic, the previous expression is less
than or equal to
\begin{equation*}
\frac{\sqrt{2}}{\delta^{p,-}_{k,k+1}}  \,
\exp\Big\{\, \Big[\,
S''(\sigma^{p,-}_{k,k+1}) \,
(\delta^{p,-}_{k,k+1})^2 \,+\, C_0 \,
(\delta^{p,-}_{k,k+1})^3 \,\Big] /4\epsilon\, \Big\}
\end{equation*}
for some finite constant $C_0$ independent of $\epsilon$, whose value
may change from line to line. By definition of $\delta^{p,-}_{k,k+1}$,
this expression is bounded by
$C_0 \, \sqrt{\epsilon} / \delta^{p,-}_{k,k+1}$.  This proves
assertion \eqref{56}.

We claim that
\begin{equation}
\label{57}
\frac{1}{\theta^p_\epsilon}\, \sup_{x\in J^{p,+}_{k}} 
\bb E^\epsilon_{x} [\, \tau (m^{-}_{p,k}, \sigma^{p,-}_{k,k+1}) \, ]
\,\le\, \frac{\mf c_0 \, \theta^{p-1}_\epsilon}{\epsilon\,
\theta^p_\epsilon} \;,
\end{equation}
where $\mf c_0$ is the constant appearing in Lemma \ref{l01}.

By Lemma \ref{l01}, the left-hand side of \eqref{57} is bounded by
\begin{equation*}
\frac{\mf c_0 }{\epsilon\, \theta^p_\epsilon}\,
\exp\Big\{ \,
\sup_{m^{-}_{p,k} \le x \le y \le \sigma^{p,-}_{k,k+1}}
[S(x)-S(y)]/\epsilon\,\Big\}\;.
\end{equation*}
Let $m^*$ be the rightmost local minima of $S(\cdot)$ smaller than
$\sigma^{p,-}_{k,k+1}$ so that $S(\cdot)$ is monotone on the interval
$[m^*, \sigma^{p,-}_{k,k+1}]$. We may restrict the supremum in $x$ to
values in the interval $[m^{-}_{p,k}, m^*]$ because for
$x\in [m^*, \sigma^{p,-}_{k,k+1}]$, $S(x)-S(y)\le 0$ if
$x\le y\le \sigma^{p,-}_{k,k+1}$.

Since $S(\cdot)$ is monotone between its extremal points, the supremum
appearing in the previous formula is achieved on the extremal points
of $S(\cdot)$:
\begin{equation*}
\sup_{m^{-}_{p,k} \le x \le y \le \sigma^{p,-}_{k,k+1}}
[S(x)-S(y)] \,=\, \max_{\sigma \in [m^{-}_{p,k}, m^*] \cap \mc W}
\;\; \max_{m\in [\sigma, m^*] \cap \mc M} [S(\sigma)-S(m)]\;.
\end{equation*}
By Lemma \ref{l17}, this expression is bounded by $\mf h_{p-1}$. This
proves \eqref{57} since
$\theta^{p-1}_\epsilon = e^{\mf h_{p-1} /\epsilon}$, and completes the
proof of the proposition.
\end{proof}

By the previous result and the definition of $\delta^{p,-}_{k,k+1}$,
$\delta^{p,+}_{k-1,k}$, 
\begin{equation}
\label{05p}
\lim_{\epsilon \to 0} \max_{k\in \bb Z}
\sup_{x,y \in \ms E(\ms M_p(k))}
|\phi_{p,\epsilon}(x) - \phi_{p,\epsilon}(y)| = 0\,.
\end{equation}

\subsection*{Convergence to $f_{p}$.}

Let $f_{p,\epsilon} \colon \ms S_p \to \bb R$ be given by
\begin{equation}
\label{70}
{\color{blue} f_{p,\epsilon} (\ms M_p(k))} \,:=\,
\phi_{p,\epsilon}(m^{-}_{p,k})\;.
\end{equation}
The main result of this section reads as follows.

\begin{proposition}
\label{l24}
For all $k\in \bb Z$, 
\begin{equation*}
\lim_{\epsilon \to 0} \big|\, f_{p,\epsilon} (\ms M_p(k))
\,-\, f_p(\ms M_{p}(k))\,\big| \,=\,  0\;,
\end{equation*}
where $f_p(\cdot)$ is the unique bounded solution of the 
reduced resolvent equation \eqref{42}.
\end{proposition}

Recall the notation in introduced in \eqref{64}, and the definition of
$\delta^{p,\pm}_{k,k+1}$ presented in \eqref{21}. Let
\begin{equation*}
{\color{blue} \eta^{p,\pm}_{k,k+1}}
\,:=\, \frac{\epsilon}{\delta^{p,+}_{k,k+1}}
\;=\;
\sqrt{\frac{-\, \epsilon \, S''(\sigma^{p,\pm}_{k,k+1})}
{2  \,  \log \epsilon^{-1}}}\;,
\quad k\in \bb Z\;.
\end{equation*}
Define the \emph{core} sets $C_{p,k}$, $\Lambda^{p,\pm}_{k,k+1}$,
$A_{p,k}$, $k\in\bb Z$, [whose dependence on $\epsilon$ has been
omitted from the notation] by
\begin{gather*} 
{\color{blue} C_{p,k}}
\,:=\, \big[\, \sigma^{p,+}_{k-1,k} + \delta^{p,+}_{k-1,k} + \eta^{p,+}_{k-1,k}  \,,\,
\sigma^{p,-}_{k,k+1} -  \delta^{p,-}_{k,k+1} - \eta^{p,-}_{k,k+1} \,\big] \;,
\\
{\color{blue} \Lambda^{p,\pm}_{k,k+1}}
\,:=\, \big[\,  \sigma^{p,\pm}_{k,k+1}-
\delta^{p,\pm}_{k,k+1}  \,,\,
\sigma^{p,\pm}_{k,k+1}+
\delta^{p,\pm}_{k,k+1}  \,\big] \;,
\quad
{\color{blue} \Lambda_{p,k}}
\,:=\, \Lambda^{p,+}_{k-1,k} \cup \Lambda^{p,-}_{k,k+1} \,,
\\
{\color{blue}  A_{p,k} }  \; := \;
\Lambda^{p,+}_{k-1,k} \,\cup\, C_{p,k} \,\cup\,
\Lambda^{p,-}_{k,k+1}  \;;
\end{gather*}
the \emph{binding} sets $B_l(\sigma^{p,\pm}_{k,k+1})$,
$B_r(\sigma^{p,\pm}_{k,k+1})$, 
$B_{p,k}$, by
\begin{gather*}
{\color{blue} B_l(\sigma^{p,\pm}_{k,k+1})}  \,:=\,
(\sigma^{p,\pm}_{k,k+1} - \delta^{p,\pm}_{k,k+1} -
\eta^{p,\pm}_{k,k+1}  \,,\,
\sigma^{p,\pm}_{k,k+1} - \delta^{p,\pm}_{k,k+1})\;,
\\
{\color{blue} B_r(\sigma^{p,\pm}_{k,k+1})}  \,:=\,
(\sigma^{p,\pm}_{k,k+1} + \delta^{p,\pm}_{k,k+1} \,,\,
\sigma^{p,\pm}_{k,k+1} + \delta^{p,\pm}_{k,k+1} +
\eta^{p,\pm}_{k,k+1} )\;,
\\
{\color{blue} B_{p,k}} \,:=\,
B_l(\sigma^{p,+}_{k-1,k}) \,\cup\,
B_r(\sigma^{p,+}_{k-1,k} ) \,\cup\,
B_l(\sigma^{p,-}_{k,k+1}) \,\cup\,
B_r(\sigma^{p,-}_{k,k+1} ) \;;
\end{gather*}
and the $\ms M_{p}(k)$-neighbourhood $\Omega_{p,k}$ by
\begin{equation}
\label{kth_E_neigh}
{\color{blue}\Omega_{p,k}} \, :=\,
A_{p,k} \,\cup\, B_{p,k} \;.
\end{equation}
We refer to Figure \ref{fig-f2} for an illustration. 

Recall the definition of the solution of the linearised Poisson
equation $\rho_{\sigma, \epsilon}$, introduced in \eqref{13}.  The
test function $Q_{\epsilon}^{p,k}$ is defined as
$Q_{\epsilon}^{p,k} = J_{\epsilon}^{p,k}/\mss a$, where
$J_{\epsilon}^{p,k}$ is equal to
$\rho_{\sigma^{p,+}_{k-1,k}, \epsilon}$ in the neighbourhood
$\Lambda^{p,+}_{k-1,k} $ of $\sigma^{p,+}_{k-1,k}$ and as
$1 - \rho_{\sigma^{p,-}_{k,k+1}, \epsilon}$ in
$\Lambda^{p,-}_{k,k+1}$. More precisely, we first define the test
function on $\bb R \setminus B_{p,k}$. Let
$J_{\epsilon}^{p,k}\colon \bb R \setminus B_{p,k} \to \bb R$
be given by
\begin{equation}
\label{12b}
{\color{blue} J_{\epsilon}^{p,k}(x) } \;:=\;
\begin{cases} \rho_{\sigma^{p,+}_{k-1,k},\epsilon} (x-\sigma^{p,+}_{k-1,k}),
&\text{if } x\in \Lambda^{p,+}_{k-1,k} 
\\
1 , &\text{if } x\in C_{p,k}
\\
(1 - \rho_{\sigma^{p,-}_{k,k+1} ,\epsilon})  (x- \sigma^{p,-}_{k,k+1}), &\text{if }
x\in \Lambda^{p,-}_{k,k+1} 
\\
0, &\text{ otherwise. }
\end{cases}
\end{equation} 
We extend the definition of $J_{\epsilon}^{p,k}$ to $B_{p,k}$
by linear interpolation. 

The test function $Q_{\epsilon}^{p,k} \colon \bb R \to \bb R$ is given
by $\color{blue} Q_{\epsilon}^{p,k} := J_{\epsilon}^{p,k}/\mss a$.  It
is continuous, piecewise smooth, and its derivative is discontinuous
at the boundary of $B_{p,k}$.

Since $Q_{\epsilon}^{p,k}$ is piecewise smooth and continuous,
equation \eqref{17} holds with the index $1$ replaced by $p$. Multiply
both sides of this equation by
$\exp\{S(\ms M_p(k)) /\epsilon\} / \sqrt{2\pi \epsilon}$ to get that
\begin{equation}
\label{61}
\begin{aligned}
& \frac{\lambda }{\sqrt{2\pi \epsilon}}\,
\int_{\bb R}\phi_{p,\epsilon}(x)\, J_{\epsilon}^{p,k}(x)\,
\frac{1}{\mss a(x)}\, 
e^{- [S(x)-S(\ms M_p(k))]/\epsilon}\, dx
\\
&\quad 
\,+\, \epsilon \, \theta^{(p)}_{\epsilon}\,
\frac{1}{\sqrt{2\pi \epsilon}}\, \int_{\bb R} (\, \partial_x \,
J_{\epsilon}^{p,k} )(x) \, (\partial_x \phi_{p,\epsilon}) (x) \,
e^{-[S(x)-S(\ms M_p(k))]/\epsilon} \, dx
\\
& \quad \,=\, \frac{1}{\sqrt{2\pi \epsilon}}\, \int_{\bb R}
J_{\epsilon}^{p,k}(x)\, \frac{1}{\mss a(x)} \,  G (x) \,
e^{-[S(x)-S(\ms M_p(k))] /\epsilon} \, dx\;.
\end{aligned}
\end{equation}

\subsection*{The resolvent equation}

We estimate each term of \eqref{61} separately.  We start with the
right-hand side. Recall the definition of $\pi_p(k)$ introduced in
\eqref{01}. Since the support of
$J^{p,k}_{\epsilon} (\cdot) \, G(\cdot)$ is contained in
$\ms E(\ms M_p(k))$, where $G(\cdot)$ takes the value $g(\ms M_p(k))$
and $J^{p,k}_{\epsilon} (\cdot)$ is equal to $1$,
\begin{equation}
\label{62}
\begin{aligned}
& \frac{1}{\sqrt{2\pi \epsilon}}\, \int_{\bb R}
J_{\epsilon}^{p,k}(x)\, \frac{1}{\mss a(x)} \, G(x) \, e^{-[S(x)-S(\ms
M_p(k))] /\epsilon} \, dx
\\
&\quad \;=\; \frac{g(\ms M_p(k))}{\sqrt{2\pi \epsilon}}\, \int_{\ms
E(\ms M_p(k))} \frac{1}{\mss a(x)} \, e^{-[S(x)-S(\ms M_p(k))]
/\epsilon} \, dx
\\
&\quad \,=\, (1+ o_\epsilon(1)) \, g(\ms M_p(k))\,
\frac{1}{\sqrt{2\pi}}\, \pi_p(k) \,.
\end{aligned}
\end{equation}

On the other hand, since $\phi_{p,\epsilon}(\cdot)$
$J_{\epsilon}^{p,k} (\cdot)$ are uniformly bounded,
$J_{\epsilon}^{p,k}$ vanishes on $\Omega_{p,k}^c$, and the measure
$\exp \{-[S(x) - S(\ms M_p(k))] /\epsilon\} \, dx$ is concentrated on
the global minima of $S(\cdot)$ on $\Omega_{p,k}$, by
Lemma \ref{l14},

\begin{equation}
\label{22}
\begin{aligned}
& \frac{\lambda}{\sqrt{2\pi \epsilon}}\, \int_{\bb R} J_{\epsilon}^{p,k} (x)\,
\frac{1}{\mss a(x)} \, \phi_{p,\epsilon} (x) \, e^{-[S(x) - S(\ms M_p(k))] /\epsilon} \, dx
\\
&\quad \;=\;
\frac{\lambda}{\sqrt{2\pi \epsilon}}\,
\int_{\Omega_{p,k}}  J_{\epsilon}^{p,k} (x)\,
\frac{1}{\mss a(x)} \, \phi_{p,\epsilon} (x) \,
e^{-[S(x) - S(\ms M_p(k))] /\epsilon} \, dx
\\
&\quad \;=\;
(1+ o_\epsilon(1)) \, \lambda\,  f_{p,\epsilon}(\ms M_p(k)) \,
\frac{1}{\sqrt{2\pi}} \, \pi_p(k)  \,,
\end{aligned}
\end{equation}
where $f_{p,\epsilon}(\cdot)$ has been introduced in \eqref{70}.

We turn to the second term on the left-hand side of \eqref{61} which
is the most delicate one.  Since the test function
$J_{\epsilon}^{p,k}$ vanishes on $\Omega_{p,k}^c$, the integral can be
taken over the set $\Omega_{p,k}$.

\begin{lemma}
\label{l25}
For each $k\in \bb Z$,
\begin{equation*}
\begin{aligned}
& \epsilon \, \theta^{(p)}_{\epsilon}\, \frac{1}{\sqrt{2\pi \epsilon}}\,
\int_{\Omega_{p,k}} (\, \partial_x \,
J_{\epsilon}^{p,k} )(x) \, (\partial_x \phi_{p,\epsilon}) (x) \,
e^{-[S(x)-S(\ms M_p(k))]/\epsilon} \, dx
\\
&\quad \,=\, -\, \sum_{i=\pm 1} 
R_p(\ms M_p(k),\ms M_p(k+i))
\, [\, f_{p,\epsilon}(\ms M_p(k+i))- f_{p,\epsilon}(\ms M_p(k))\,]
\,+\, o_\epsilon(1)  \;. 
\end{aligned}
\end{equation*}
\end{lemma}

Mind that the right-hand side may vanish since the jump rates may
be equal to $0$.  The proof is carried out estimating the integral in each
region forming $\Omega_{p,k}$. By definition,
$\partial_x \, J_{\epsilon}^{p,k}$ vanishes on
$C_{p,k}(\epsilon)$. Lemma \ref{as4} takes care of the integral on
the set $B_{p,k}$, where $J^{p,k}_{\epsilon}$ is linear, and Lemma
\ref{as5} on the set
$\Lambda_{p,k}$.

\begin{lemma}
\label{as4}
For all $k\in \bb Z$,
\begin{equation*}
\lim_{\epsilon\to 0}
\Big|\, \frac{\epsilon\, \theta_{\epsilon}^{(p)}}{\sqrt{2\pi \epsilon}}\,
\int_{B_{p,k}} (\, \partial_x \,
J_{\epsilon}^{p,k} )(x) \, (\partial_x \phi_{p,\epsilon}) (x) \,
e^{-[S(x)-S(\ms M_p(k))]/\epsilon} \, dx  \, \Big| \,=\, 0\;.
\end{equation*}
\end{lemma}

The proof of this result is omitted being similar to the one of
Lemma \ref{as2}.  We turn to the set $\Lambda_{p,k}$. Recall the
definition of the weight $\sigma_p(j,j+1)$, $j\in \bb Z$, introduced
in \eqref{01}.  

\begin{lemma}
\label{as5}
For all $k\in \bb Z$, 
\begin{equation}
\label{65}
\begin{aligned}
& \frac{\epsilon\, \theta_{\epsilon}^{(p)}}{\sqrt{2\pi \epsilon}}\,
\int_{\Lambda_{p,k}} (\, \partial_x \,
J_{\epsilon}^{p,k} )(x) \, (\partial_x \phi_{p,\epsilon}) (x) \,
e^{-[S(x)-S(\ms M_p(k))]/\epsilon} \, dx
\\
\,=\,
& \quad -\, \frac{1}{\sqrt{2\pi}} \, \frac{1}{\sigma_p(k,k+1)} \,
\mtt 1\{\, h^{p,+}_{k} \,=\,  \mf h_{p} \,\}\,
[\,  f_{p,\epsilon} (\ms M_p(k+1))   -  f_{p,\epsilon} (\ms M_p(k)) \,]
\\
&\quad
-\, \frac{1}{\sqrt{2\pi}} \, \frac{1}{\sigma_p(k-1,k)}  \,
\mtt 1\{\, h^{p,-}_{k} \,=\,  \mf h_{p} \,\}\,
[\, f_{p,\epsilon} (\ms M_p(k-1))   - f_{p,\epsilon} (\ms M_p(k)) \,]
\,+\, o_\epsilon(1)\;.
\end{aligned}
\end{equation}
\end{lemma}

\begin{proof}
The proof of Lemma \ref{as1} yields that the left-hand side of
\eqref{65} is equal to
\begin{equation}
\label{66}
\begin{aligned}
& -\, \frac{1}{2\pi} \, 
\sqrt{-\, S''(\sigma^{p,-}_{k,k+1})}
\, \mtt 1\{\, h^{p,+}_{k} \,=\,  \mf h_{p} \,\}
\, [\, \phi_{p,\epsilon}(\sigma^{p,-}_{k,k+1}
+ \delta^{p,-}_{k,k+1}) - \phi_{p,\epsilon}(\sigma^{p,-}_{k,k+1}
- \delta^{p,-}_{k,k+1})\,]
\\
&\quad
-\, \frac{1}{2\pi} \, 
\sqrt{-\, S''(\sigma^{p,+}_{k-1,k})}
\, \mtt 1\{\, h^{p,-}_{k} \,=\,  \mf h_{p} \,\}
\, [\, \phi_{p,\epsilon}(\sigma^{p,+}_{k-1,k}
- \delta^{p,+}_{k-1,k}) - \phi_{p,\epsilon}(\sigma^{p,+}_{k-1,k}
+ \delta^{p,+}_{k-1,k})\,]
\,+\, o_\epsilon(1)\;.
\end{aligned}
\end{equation}
By Lemma \ref{l14}, we may replace in the previous formula,
$\phi_{p,\epsilon}(\sigma^{p,-}_{k,k+1} - \delta^{p,-}_{k,k+1})$ and
$\phi_{p,\epsilon}(\sigma^{p,+}_{k-1,k} + \delta^{p,+}_{k-1,k})$ by
$f_{p,\epsilon} (\ms M_p(k))$ at a cost $o_\epsilon(1)$. Hence, by Lemma
\ref{l28} below and \eqref{67}, the previous expression is equal to
\begin{equation*}
\begin{aligned}
& -\, \frac{1}{2\pi} \, 
\sqrt{-\, S''(\sigma^{p,-}_{k,k+1} )}
\, \varpi^{+}_{p,k} \,
\mtt 1\{\, h^{p,+}_{k} \,=\,  \mf h_{p} \,\}
[\,  f_{p,\epsilon} (\ms M_p(k+1))   -  f_{p,\epsilon} (\ms M_p(k)) \,]
\\
&\quad
-\, \frac{1}{2\pi} \, 
\sqrt{-\, S''(\sigma^{p,+}_{k-1,k}) }
\, \varpi^{-}_{p,k} 
\, \mtt 1\{\, h^{p,-}_{k} \,=\,  \mf h_{p} \,\}
[\, f_{p,\epsilon} (\ms M_p(k-1))   - f_{p,\epsilon} (\ms M_p(k)) \,]
\,+\, o_\epsilon(1)\;.
\end{aligned}
\end{equation*}
To complete the proof of the lemma, it remains to recall the
definition of $\varpi^{\pm}_{p,k} $.
\end{proof}

Recall from \eqref{max_br} that we denote by $\ms W^{(p)}_{j,j+1}$,
$j\in \bb Z$, the finite set where $S(\cdot)$ attains its global
maximum in the interval $[m^+_{p,j} ,, m^-_{p,j+1}]$.  In the next
lemma, we express
$\phi_{p,\epsilon}(\sigma^{p,-}_{k,k+1} + \delta^{p,-}_{k,k+1})$ as a convex
combination of $f_{p,\epsilon} (\ms M_p(k))$ and
$f_{p,\epsilon} (\ms M_p(k+1))$.  Mind that $\varpi^{+}_{p,k}$,
defined in the next lemma, is equal to $1$ if $\ms W^{(p)}_{k,k+1}$ is
a singleton,
$\ms W^{(p)}_{k,k+1} = \{\sigma^{p,-}_{k,k+1} \} =
\{\sigma^{p,+}_{k,k+1} \}$, that is if
$\sigma^{p,-}_{k,k+1} = \sigma^{p,+}_{k,k+1}$.

\begin{lemma}
\label{l28}
Fix $k\in \bb Z$, and suppose that
$S(\sigma^{p,-}_{k,k+1}) - S(\ms M_p(k)) = \mf h_p$. Then,
\begin{equation*}
\phi_{p,\epsilon}(\sigma^{p,-}_{k,k+1} + \delta^{p,-}_{k,k+1}) \,=\,
(1-\varpi^{+}_{p,k} )\, f^{(p)}_\epsilon (\ms M_p(k)) 
\,+\, \varpi^{+}_{p,k} \,   f^{(p)}_\epsilon (\ms M_{p}(k+1)) 
\,+\,o_\epsilon(1)\;,
\end{equation*}
where
\begin{equation*}
{\color{blue} \varpi^{+}_{p,k} }  \,:=\, \frac{1/\sqrt{- S''(\sigma^{p,-}_{k,k+1})}}
{\sum_{\sigma \in \ms W^{(p)}_{k,k+1}} 1/\sqrt{- S''(\sigma)}}  \;.
\end{equation*}
\end{lemma}

\begin{proof}
If $\sigma^{p,-}_{k,k+1} = \sigma^{p,+}_{k,k+1}$, the assertion
follows from Lemma \ref{l14}.  Assume, therefore that
$\sigma^{p,-}_{k,k+1} < \sigma^{p,+}_{k,k+1}$.  The proof is based on
the second assertion of Proposition \ref{l04}.  Let
$x=x_\epsilon = \sigma^{p,-}_{k,k+1} + \delta^{p,-}_{k,k+1}$,
$a=a_\epsilon = \sigma^{p,-}_{k,k+1} - \delta^{p,-}_{k,k+1}$,
$b=b_\epsilon = \sigma^{p,+}_{k,k+1} + \delta^{p,+}_{k,k+1}$. Since
$S(\sigma^{p,-}_{k,k+1}) - S(\ms M_p(k)) = \mf h_p$, by Lemmata
\ref{l01} and \ref{l27},
\begin{equation*}
\bb E^\epsilon_{x}\big[\, \tau(a,b)\, \big] / \theta^{(p)}_\epsilon
\, =\, o_\epsilon(1)\;.
\end{equation*}

By the proof of Corollary \ref{l26},
\begin{equation*}
\bb P^\epsilon_{x}\{\tau(b)< \tau(a)\} \,=\,
\varpi^{+}_{p,k} \,+\, o_\epsilon (1)\;.
\end{equation*}
Note that we can not apply directly Corollary \ref{l26} as the initial
point depends on $\epsilon$. But the proof goes through without
problem. 

It follows from the previous estimates and the second
assertion of Proposition \ref{l04} that
\begin{equation*}
\phi_{p,\epsilon}(\sigma^{p,-}_{k,k+1} + \delta^{p,-}_{k,k+1}) \,=\,
(1-\varpi^{+}_{p,k})\, \phi_{p,\epsilon}(\sigma^{p,-}_{k,k+1} - \delta^{p,-}_{k,k+1}) 
\,+\, \varpi^{+}_{p,k} \,  \phi_{p,\epsilon}(\sigma^{-}_{p,k+1} + \delta^{+}_{p,k+1})  
\,+\,o_\epsilon(1)\;.
\end{equation*}
To complete the proof of the lemma, it remains to call Lemma \ref{l14}.
\end{proof}

Similarly, if $S(\sigma^{p,+}_{k-1,k}) - S(\ms M_p(k)) = \mf
h_p$, then,
\begin{equation}
\label{67}
\phi_{p,\epsilon}(\sigma^{p,+}_{k-1,k} - \delta^{p,+}_{k-1,k}) \,=\,
(1-\varpi^{-}_{p,k})\, f^{(p)}_\epsilon(\ms M_p(k)) 
\,+\, \varpi^{-}_{p,k} \,  f^{(p)}_\epsilon(\ms M_p(k-1)) 
\,+\,o_\epsilon(1)\;,
\end{equation}
where
\begin{equation*}
{\color{blue} \varpi^{-}_{p,k} } \,: =\,
\frac{1/\sqrt{- S''(\sigma^{p,+}_{k-1,k})}}
{\sum_{\sigma \in \ms W^{(p)}_{k-1,k}} 1/\sqrt{- S''(\sigma)}}  \;.
\end{equation*}

\section{Proof of Theorem \ref{mt3}}
\label{sec8}

We follow the strategy proposed in \cite{llm} and applied in
\cite{lls1, lls2} in the context of multidimensional diffusion
processes.  We start presenting a consequence of Theorem \ref{mt4}.

\subsection*{Metastability}

Fix $1\le p\le \mf q$, and denote by $Y^{(p)}_\epsilon (\cdot)$ the
process $X_\epsilon(\cdot)$ speeded-up by $\theta^{(p)}_\epsilon$:
$\color{blue} Y^{(p)}_\epsilon (t) := X_\epsilon
(t\theta^{(p)}_\epsilon)$.  Denote by
$\color{blue} \bb P^{p, \epsilon} _{x}$, $x\in \bb R$, the probability
measure on $D(\bb R_{+}, \bb R)$ induced by the process
$Y^{(p)}_\epsilon (\cdot)$ starting from $x\in\bb R$.  Note that
$\bb P^{p, \epsilon} _{x} = \bb P^{\theta^{(p)}_\epsilon, \epsilon}
_{x}$, where $\bb P^{\theta^{(p)}_\epsilon, \epsilon} _{x}$ has been
introduced at the beginning of Section \ref{sec2}.

Fix a subset $A$ of $\bb R$ and a trajectory
$\mtt x\colon \bb R_+\to \bb R$.  Denote by
$\mss T_A(t) = \mss T_{\mtt x(\cdot), A}(t)$, $t\ge 0$, the total time
the trajectory $\mtt x (\cdot)$ spent in $A$ in the time interval
$[0,t]$, and by $\mss S_A (\cdot) := \mss S_{\mtt x(\cdot),A} (\cdot)$
the generalised inverse of $\mss T_A(\cdot)$:
\begin{equation*}
{\color{blue} \mss T_A(t)} \,:=\, \int_{0}^{t}
\chi_{_{A}}(\mtt x (s))\, ds \;, \quad
{\color{blue}\mss S_A(t)} \, := \,
\sup\{s\geq 0: \mss T_A(s)\leq t\}\;,
\end{equation*}
where, recall, $\chi_{_{B}}$, $B\subset \bb R$, represents the indicator
function of the set $B$.

Let $\color{blue} (\ms F_{t} : t\ge 0)$ be the augmentation of the
natural filtration on $\bb C(\bb R_{+}, \bb R)$ generated by
$\mtt x(\cdot)$. For each $t\ge 0$, $A\subset \bb R$, $\mss S_A(t)$ is
a stopping time with respect to this filtration. Let
$(\mtt x^{\mss T, A} (t) : t\ge 0)$ be the process
$\mtt x (\cdot)$ time-changed by
$(\mss S_A (r) : r\ge 0)$:
\begin{equation}
\label{25}
{\color{blue} \mtt x^{\mss T, A} (t)} \,: =\,
\mtt x (\mss S_A(t))\;.
\end{equation}
Clearly, $\mtt x^{\mss T, A}(\cdot)$ is an $A$-valued trajectory
adapted to the filtration
$\color{blue} \ms G^{\mss T, A}_r := \ms F_{\mss S_A (r)}$, $r\ge
0$. We call $\mtt x^{\mss T, A} (\cdot)$ the \emph{trace process} of
$\mtt x(\cdot)$ on $A$.

Recall from Corollary \ref{l32} the definition of the set $\ms E_p$,
$1\le p\le \mf q$.
Define the projection $\Psi_p\colon \ms E_p \to \ms S_p$ by
\begin{equation*}
{\color{blue} \Psi_p(x)} \,:=\, \sum_{k\in\bb Z}  \ms M_p(k) \,
\chi_{_{\ms E (\ms M_p(k))}} (x)\;.
\end{equation*}
Define the $\ms S_p$-valued \emph{order process}
$\mtt y^{{\mtt T}, p}_\epsilon \colon C(\bb R_+, \bb R) \to D(\bb R_+,
\ms S_p)$ by
\begin{equation*}
{\color{blue}  \mtt y^{{\mtt T}, p} (t)} \,: =\,
\Psi_p(\mtt x^{\mss T, \ms E_p} (t))\;.
\end{equation*}

Denote by $\color{blue} \bb P_{x}^{{\mtt T}, p, \epsilon}$,
$x\in \ms E_p$, the probability measure on $D(\bb R_+, \ms E_p)$
obtained as the push-forward of $\bb P_{x}^{p, \epsilon}$ by
$\mtt x^{\mss T, \ms E_p}$ (understood as a function from
$C(\bb R_+, \bb R)$ to $D(\bb R_+, \ms E_p)$). Let
$\color{blue} \mtt Q_{x}^{{\mtt T}, p, \epsilon}$ be the push-forward
of $\bb P_{x}^{{\mtt T}, p, \epsilon}$ by $\Psi_p$, which corresponds
to the distribution of $\mtt y^{{\mtt T}, p}(\cdot)$ under
$\bb P_{x}^{{\mtt T}, p, \epsilon}$. Recall the definition of the
measure $\mtt Q_{\ms M_p(k)}^{(p)}$, $k\in \bb Z$, introduced below
\eqref{52}.  Next result follows from Theorem \ref{mt4} and
\cite[Theorem 2.3]{lms}.

\begin{theorem}
\label{l29}
Fix $1\le p\le \mf q$. Then,
\begin{equation}
\label{23}
\lim_{\epsilon \to 0} \sup_{x\in \ms E_p}
\bb E_{x}^{p,\epsilon}\Big[ \, \int_{0}^{t}
\chi_{_{\bb R \setminus \ms E_p}}
(\mtt x (s))\,  ds \, \Big]
\,=\, 0
\end{equation}
for every $t\geq 0$. Furthermore, for all $k\in \bb Z$ and sequences
$x_\epsilon \in \ms E(\ms M_p(k))$, {\rm
${\mtt Q}_{x_\epsilon}^{{\mtt T}, p, \epsilon}$} converges to {\rm
${\mtt Q}_{\ms M_p(k)}^{(p)}$}.
\end{theorem}

Next assertion is a consequence of \eqref{23}.  We refer to
\cite[display (3.2)]{llm} for a proof.

\begin{corollary}
\label{l37}
Fix $1\le p\le \mf q$. Then,
\begin{equation*}
\lim_{\epsilon \to 0} 
\sup_{x\in \ms E_p} 
\bb P^{p,\epsilon}_{x}\big[\, \mss S_{\ms E_p} (t)
> t+\delta\,\big] \,=\, 0
\end{equation*}
for all $t>0$ and $\delta > 0$.
\end{corollary}

The next result is a consequence of the previous theorem. It can also
be obtained from Theorem \ref{mt4}, see \cite[Lemma 4.2]{lms}.

\begin{corollary}
\label{l07}
Fix $1\le p\le \mf q$. Then, for all $k\in\bb Z$,
\begin{equation*}
\lim_{a \to 0} \limsup_{\epsilon \to 0} 
\sup_{x\in \ms E(\ms M_p(k))} 
\bb P^{p,\epsilon}_{x}\big[\, \tau(\ms E_p \setminus \ms E(\ms M_p(k)))
\le  a  \,\big] \,=\, 0\;.
\end{equation*}
\end{corollary}

\begin{corollary}
\label{l40}
Fix $1\le p< \mf q$. Then,
\begin{equation*}
\lim_{t\to\infty}\limsup_{\epsilon\to 0}
\sup_{x\in \bb R} \bb P^{\epsilon}_x\big[ \, 
\tau (\ms M_{p+1}) > t \, \theta^{(p)}_\epsilon  \, \big] \,=\, 0\;.
\end{equation*}
\end{corollary}

\begin{proof}
We proceed in two steps. We first apply Lemma \ref{l33} to show that
the process $X_\epsilon(\cdot)$ reaches $\ms M_p$ in a time of order
$\theta^{(p)}_\epsilon$, and then use the convergence of the order
process to show that starting from $\ms M_p$ we reach a
$\bb X_p$-closed irreducible class in a time of order
$t\, \theta^{(p)}_\epsilon$ for $t$ large.

By Lemma \ref{l33} and Chebyshev inequality,
\begin{equation*}
\bb P^{\epsilon}_x\big[ \, 
\tau (\ms M_{p+1}) > t \, \theta^{(p)}_\epsilon  \, \big]
\,\le\, \bb P^{\epsilon}_x\big[ \, \tau (\ms M_{p}) \le
\theta^{(p)}_\epsilon \,,\,
\tau (\ms M_{p+1}) > t \, \theta^{(p)}_\epsilon  \, \big]
\,+\, R_\epsilon (x)\,,
\end{equation*}
where
\begin{equation}
\label{88}
\lim_{\epsilon\to 0} \sup_{x\in \bb R} |R_\epsilon (x) | \,=\, 0\;. 
\end{equation}
If $t>1$, by the strong Markov property, the first term on the
right-hand side is bounded by
\begin{equation*}
\sup_{m\in \ms M_{p}} \bb P^{\epsilon}_m\big[ \, 
\tau (\ms M_{p+1}) \ge (t-1) \, \theta^{(p)}_\epsilon  \, \big]
\,=\,
\sup_{m\in \ms M_{p}} \bb P^{p,\epsilon}_m\big[ \, 
\tau (\ms M_{p+1}) \ge t-1 \, \big] \;.
\end{equation*}

By Lemma \ref{l01}, we may replace in the previous equation
$\ms M_{p+1}$, $t-1$ by $\ms E_{p+1}$, $t-2$, respectively, at an
extra cost $R^{(1)}_\epsilon (m)$ satisfying \eqref{88} (with
$\sup_{m\in \ms M_{p}}$ replacing $\sup_{x\in \bb R}$). (We are using
here that once $\ms E_{p+1}$ is reached, $\ms M_{p+1}$ is attained
immediately after in view of Lemma \ref{l01}). By \eqref{23}, the
previous probability (with $\ms E_{p+1}$, $t-2$ replacing
$\ms M_{p+1}$, $t-1$, respectively) is less than or equal to
\begin{equation*}
\bb P^{p,\epsilon}_m\Big[ \, \int_0^{t-2}
\chi_{_{\bb R\setminus \ms E_p}} (\mtt x(s))\,  ds
\le 1\, ,\, 
\tau (\ms E_{p+1}) \ge t-2 \, \Big]
\,+\, R^{(2)}_\epsilon (m)\,.
\end{equation*}
where $R^{(2)}_\epsilon (m)$ fulfills \eqref{88}. By definition of the
trace process and the measure $\bb P^{\mss T, p,\epsilon}_m$,
introduced just before the statement of the Theorem \ref{l29}, the
previous probability is bounded by
\begin{equation*}
\bb P^{\mss T, p,\epsilon}_m\big[ \, \tau (\ms E_{p+1}) \ge t-3 \,
\big] \,.
\end{equation*}

Since the first time the trace process $\mtt x^{\mss T,p} (\cdot)$
hits $\ms E_{p+1}$ corresponds to the first time the order process
$\mtt y^{\mss T,p} (\cdot)$ hits a recurrent element, the previous
expression is equal to
\begin{equation*}
\mtt Q^{\mss T, p,\epsilon}_m\big[ \, \tau  (\mf R_p) \ge t-3 \,
\big] \,,
\end{equation*}
where $\mf R_p = \cup_{k\in \bb Z} \mf R_p(k)$, and, recall from
\eqref{24}, $\mf R_p(k)$ are the $\bb X_p$-recurrent classes.
By construction, $\mf R_p = \ms M_{p+1}$.

Since the set $\{\tau  (\mf R_p) \ge t-3\}$ is closed for the Skorohod
topology, by Theorem \ref{l29},
\begin{equation*}
\limsup_{\epsilon\to 0} \sup_{m\in \ms M_p}
\mtt Q^{\mss T, p,\epsilon}_m\big[ \, \tau  (\mf R_p) \ge t-3 \,
\big] \,\le\, \sup_{k\in\bb Z}
\mtt Q^{(p)}_{\ms M_p(k)} \big[ \, \tau  (\mf R_p) \ge t-3 \,
\big] \,.
\end{equation*}
As $p<\mf q$ the process $\bb X_p(\cdot)$ has recurrent classes.
Since the rates are periodic, this expression vanishes as
$t\to\infty$, uniformly in $k\in\bb Z$. 
\end{proof}

\subsection*{Finite-dimensional distributions}

The proof of Theorem \ref{mt3} is based on the next result, whose
proof is similar to the one of \cite[Proposition 2.1]{llm},
\cite[Proposition 10.2]{lls1}. Recall from the beginning of this
section the definition of the measure $\bb P^{p, \epsilon} _{x}$.

\begin{proposition}
\label{l11}
Fix $1\le p\le \mf q$, $r_0>0$ as in \eqref{71}, and $k\in \bb Z$. Then,
\begin{equation*}
\lim_{\epsilon\rightarrow0}
\mathbb{P}_{x_\epsilon}^{p,\epsilon}
\Big[\, \bigcap_{j=1}^{{\mf n}}
\big \{\, \mtt x (t_{j,\epsilon} )
\in \mc E( \ms M_p(k_j))\, \big\} \, \Big]
\,=\, \mtt Q^{(p)}_{\ms M_p(k)}
\big[\, \bigcap_{j=1}^{{\mf n}}
\big\{ \, \mtt x (t_{j}) = \ms M_p(k_j) \, \} \, \, \big]
\end{equation*}
for all ${\mf n}\ge 1$, $0<t_{1}<\cdots<t_{{\mf n}}$,
$k_1, \dots, k_{\mf n} \in \bb Z$, and sequences
$x_\epsilon \in \mc E(\ms M_p(k))$, $t_{j,\epsilon} \to t_j$.
\end{proposition}

\begin{proof}
The proof is similar to the one of \cite[Proposition 2.1]{llm}.
Assume that $\mf n=1$. The case $\mf n>1$ is obtained by induction
using the Markov property of the diffusion process
$X_\epsilon(\cdot)$.

Fix $t>0$, $j$, $k \in \bb Z$, and a sequence $x_\epsilon \in \mc E(\ms
M_p(j))$. By Theorem \ref{l29}, and the definition of the measure
$\mtt Q^{{\mtt T}, p, \epsilon}_x$,
\begin{equation*}
\begin{aligned}
{\mtt Q}_{\ms M_p(j)}^{(p)} [ \mtt x (t) = \ms M_p(k)]
\; & =\; \lim_{\delta \to 0}
{\mtt Q}_{\ms M_p(j)}^{(p)}
[ \, \mtt x (t-3\delta) = \ms M_p(k) \, ]
\\
\; & =\; \lim_{\delta \to 0}
\lim_{\epsilon \to 0}
\mtt Q^{{\mtt T}, p, \epsilon} _{x_\epsilon}
[ \, \mtt x  (t-3\delta) =  \ms M_p(k) \, ]
\\
\; & =\; \lim_{\delta \to 0}
\lim_{\epsilon \to 0}
\bb P^{{\mtt T} , p, \epsilon} _{x_\epsilon}
[ \, \mtt x  (t-3\delta) \in   \ms E(\ms M_p(k)) \, ]
\;.
\end{aligned}
\end{equation*}
Fix $t>0$ and a sequence $t_\epsilon\to t$. By Lemma \ref{l08} below,
\begin{equation}
\label{26}
{\mtt Q}_{\ms M_p(j)}^{(p)} [ \mtt x (t) = \ms M_p(k)]
\;\le \;
\liminf_{\epsilon\to 0}
\bb P^{p, \epsilon} _{x_\epsilon}
[ \, \mtt x  (t_\epsilon) \in  \ms M_p(k) \, ]\;.
\end{equation}

Suppose that the inequality is strict for some $k_0\in \bb Z$.
Then, there exists $\eta>0$ such that
\begin{equation*}
{\mtt Q}_{\ms M_p(j)}^{(p)} [ \mtt x (t) = \ms M_p(k_0)] \,+\, 2 \eta
\;\le \;
\liminf_{\epsilon\to 0}
\bb P^{p, \epsilon} _{x_\epsilon}
[ \, \mtt x  (t_\epsilon) \in  \ms M_p(k_0) \, ]\;.
\end{equation*}
Choose a finite subset $A$ of $\bb Z$ which contains $k_0$ and such
that
\begin{equation*}
1 - \eta \; \le \; \sum_{k\in A}
{\mtt Q}_{\ms M_p(j)}^{(p)} [ \mtt x (t) = \ms M_p(k)]\;.
\end{equation*}
By the previous estimates,
\begin{equation*}
1 - \eta \; \le \;
\liminf_{\epsilon\to 0} \sum_{k\in A}
\bb P^{p, \epsilon} _{x_\epsilon}
[ \, \mtt x  (t_\epsilon) \in  \ms M_p(k) \, ]
\,-\, 2\eta \;\le\; 1 - 2\eta\;,
\end{equation*}
which is a contradiction. Therefore, we have an identity in
\eqref{26} for each $k\in\bb Z$. If the $\liminf$ does not coincide
with the $\limsup$, we may repeat the same argument to obtain a
contradiction. Thus, the $\liminf$ in \eqref{26} is actually a limit
and the inequality is an identity. This completes the proof of the
proposition. 
\end{proof}

The next result is stated in \cite[Lemma 3.1]{llm} in the context of
discrete-valued Markov processes. It estimates the probability that
the trace process $Y^{{\mtt T}, p}_\epsilon (t)$ belongs to a set
$\ms E (\ms M_p(j))$ by the probability that the original process
$Y^{p}_\epsilon (t)$ belongs to this set at the price of changing the
time $t$ by $t-a$ for $a$ small.

\begin{lemma}
\label{l08}
Fix $1\le p\le\mf q$, and $j$, $k\in\bb Z$.  Then, 
\begin{equation*}
\bb P^{{\mtt T}, p, \epsilon}_{x}
\big[\, \mtt x (t-3b) \in \mathcal{E}(\ms M_p(k))\,\big] \, \le\,
\bb P^{p, \epsilon} _{x}
[ \, \mtt x  (t) \in  \ms E(\ms M_p(k)) \, ]
\,+ \, R_{\epsilon}(x,t,b)\;,
\end{equation*}
where 
\begin{equation*}
\lim_{b\rightarrow0}\,\limsup_{\epsilon\to0}\,
\sup_{x\in \ms E (\ms M_p(j))}
|\,R_{\epsilon}(x, t_\epsilon,b)\,| \,=\, 0
\end{equation*}
for all $t>0$, and sequences $t_\epsilon \to t$.
\end{lemma}

\begin{proof}
Fix $t>0$, $j$, $k\in\bb Z$, $x\in \ms E(\ms M_p(j))$, a sequence
$t_\epsilon \to t$, and $2< \gamma  < 3$. By the definition \eqref{25} of
the trace process and the trivial fact that $S_{\epsilon}(t)\ge t$,
for $b\in(0,t/3)$
\begin{equation*}
\bb P^{{\mtt T}, p, \epsilon}_{x}
\big[\, \mtt x (t-3b) \in \ms E(\ms M_p(k))\,\big] 
\,=\, \bb P_{x}^{p, \epsilon}
\big [\, \mtt x  (S_{\epsilon}(t-3b)) \in \ms E(\ms M_p(k)) \,\big]
\;\le\;
\bb P_{x}^{p, \epsilon} [\,A_{\epsilon}(t,b)\,]
\,+\, \bb P_{x}^{p, \epsilon} 
[\,B_{\epsilon}(t,b)\,]\;,
\end{equation*}
where
\begin{gather*}
A_{\epsilon}(t,\,b) \,=\,
\{\,S_{\epsilon}(t-3b)>t- \gamma b\,\}\;,
\\
B_{\epsilon}(t,b)
\,=\,\{\, \mtt x (s)\in \ms E(\ms M_p(k))  \;\;
\text{for some}\ s\in[t-3b,t- \gamma b]\,\}\ .
\end{gather*}
By Corollary \ref{l37}, as $\gamma <3$,
\begin{equation*}
\limsup_{\epsilon\rightarrow0}\,
\sup_{x \in \ms E_p}\,
\bb P_{x}^{p, \epsilon} [\,A_{\epsilon}(t,b)\,]=0\ .
\end{equation*}
On the other hand,
\begin{equation*}
\bb P_{x}^{p, \epsilon} [B_{\epsilon}(t,b)]
\,\le\, \bb P_{x}^{p, \epsilon} 
\big [\, \mtt x(t_\epsilon)\in \ms E(\ms M_p(k))  \,\big ]
\,+\, \bb P_{x}^{p, \epsilon} 
\big  [\,B_{\epsilon}(t,b) \,,\, \mtt x(t_\epsilon)
\notin \ms E(\ms M_p(k)) \,\big ]
\end{equation*}
for any $t_\epsilon\ge 0$.

It remains to prove that
\begin{equation*}
\limsup_{b\rightarrow0}\,\limsup_{\epsilon\to 0}\,
\sup_{x\in \ms E(\ms M_p(j)) } \bb P_{x}^{p, \epsilon}  
\big[\, B_{\epsilon}(t,b) \,,\, \mtt x(t_\epsilon)
\notin \ms E(\ms M_p(k)) \,\big] \,=\, 0
\end{equation*}
if $t_\epsilon$ is a sequence converging to $t$.  By Corollary
\ref{l07}, the definition of $B_{\epsilon}(t,b)$, and the strong
Markov property,
\begin{equation*}
\limsup_{b\rightarrow0}\,\limsup_{\epsilon\to0}\,
\sup_{x\in\ms E(\ms M_p(j))}\,
\bb P_{x}^{p, \epsilon}   
\big[\,B_{\epsilon}(t,b) \,,\, \mtt x(t_\epsilon)\in
\ms E_p \setminus \ms E(\ms M_p(k)) \,\big] \,=\, 0\;.
\end{equation*}
On the other hand, as $\gamma>2$, for $\epsilon$ sufficiently small,
$t_\epsilon -s \in [2b,4b]$ for all $s\in [t-3b, t-\gamma b]$. Hence,
by the strong Markov property and Lemma \ref{l09},
\begin{equation*}
\limsup_{b\rightarrow0}\, \limsup_{\epsilon\to0}\,
\sup_{x\in\ms E(\ms M_p(j))}\,
\bb P_{x}^{p, \epsilon}   
[\, B_{\epsilon}(t,b) \,,\, \mtt x(t_\epsilon)
\notin\ms E_p \, ] \;=\; 0\;.
\end{equation*}
The assertion of the lemma follows from the previous estimates.
\end{proof}

\begin{lemma}
\label{l09}
Fix $1\le p\le\mf q$, and $k\in\bb Z$.  Then,
\begin{equation*}
\limsup_{b\rightarrow0}\, \limsup_{\epsilon\to0}\,
\sup_{t\in [2b,4b]}\, \sup_{x\in\ms E(\ms M_p(k))}\,
\bb P_{x}^{p, \epsilon}    [\, \mtt x(t) \notin\ms E_p \, ] \;=\; 0\;.
\end{equation*}
\end{lemma}

\begin{proof}
We actually prove a stronger statement, in which the well
$\ms E(\ms M_p(k))$, and the time interval $[2b,4b]$ are replaced by
$\bb R$, $[c,\infty)$, respectively: for all $c>0$
\begin{equation}
\label{33}
\limsup_{\epsilon\to0}\,
\sup_{t\ge c}\, \sup_{x\in \bb R}\,
\bb P_{x}^{p, \epsilon}    [\, \mtt x(t) \notin\ms E_p \, ] \;=\; 0\;.
\end{equation}

We proceed by induction in $p$. Assume that $p=1$, and let
$\vartheta^{(1)}_\epsilon$ be a time-scale such that
$1\prec \vartheta^{(1)}_\epsilon \prec \theta^{(1)}_\epsilon$. Then,
for all $t>0$, $x\in\bb R$, by the Markov property at time
$t\theta^{(1)}_\epsilon - \vartheta^{(1)}_\epsilon$,
\begin{equation*}
\bb P_{x}^{p, \epsilon}    [\, \mtt x(t) \notin\ms E_1 \, ] \,=\,
\bb P_{x}^{\epsilon}    [\, \mtt x(t\theta^{(1)}_\epsilon)
\notin\ms E_1 \, ] \,\le\,
\sup_{y\in\bb R} \bb P_{y}^{\epsilon} [\, \mtt x(\vartheta^{(1)}_\epsilon)
\notin\ms E_1 \, ]\;.
\end{equation*}
By the strong Markov property at time $\tau (\mc M)$, for any
$y\in\bb R$, the previous probability is bounded by
\begin{equation*}
\begin{aligned}
& \bb P_{y}^{\epsilon} [\, \tau (\mc M) > \vartheta^{(1)}_\epsilon \, ]
\,+\,
\bb P_{y}^{\epsilon} [\, \tau (\mc M) \le \vartheta^{(1)}_\epsilon
\,,\, \mtt x(\vartheta^{(1)}_\epsilon) \notin\ms E_1 \, ]
\\
&\quad
\le\, \sup_{y'\in \bb R}
\bb P_{y'}^{\epsilon} [\, \tau (\mc M) > \vartheta^{(1)}_\epsilon \, ]
\,+\, \sup_{z\in \mc M}
\bb P_{z}^{\epsilon} [\, \tau (\ms E^c_1) \le \vartheta^{(1)}_\epsilon \, ] \;.
\end{aligned}
\end{equation*}
It remains to find a time-scale $\vartheta^{(1)}_\epsilon$ which turns
the previous two expressions negligible.

By Lemma \ref{l01} and Chebyshev inequality, the first one tends to
$0$ as $\epsilon\to 0$ provided
$\vartheta^{(1)}_\epsilon \succ \epsilon^{-1}$. By Lemma \ref{l31},
the second term vanishes as $\epsilon\to 0$ provided
$\vartheta^{(1)}_\epsilon \le e^{h/\epsilon}$ for some $h<\mf
h_1$. This completes the proof of \eqref{33} for $p=1$.

Fix $2\le p\le \mf q$, and assume that \eqref{33} holds for $p-1$. In
this case, as $\theta^{(p)}_\epsilon \succ \theta^{(p-1)}_\epsilon$,
\begin{equation*}
\limsup_{\epsilon\to0}\,
\sup_{t\ge c}\, \sup_{x\in \bb R}\,
\bb P_{x}^{p, \epsilon}    [\, \mtt x(t) \notin \ms E_{p-1} \, ] \;=\; 0\;.
\end{equation*}
It remains to show that
\begin{equation}
\label{35}
\limsup_{\epsilon\to0}\,
\sup_{t\ge c}\, \sup_{x\in \bb R}\,
\bb P_{x}^{p, \epsilon}    [\, \mtt x(t) \in
\ms E_{p-1}  \setminus \ms E_p \, ] \;=\; 0\;.
\end{equation}
According to the first part of the proof, it is enough to find a
time-scale $\vartheta^{(p)}_\epsilon \prec \theta^{(p)}_\epsilon$ such
that 
\begin{equation*}
\begin{aligned}
\lim_{\epsilon\to 0} \sup_{y'\in \bb R}
\bb P_{y'}^{\epsilon} [\, \tau (\ms M_p) > \vartheta^{(p)}_\epsilon \, ]
\,=\, 0\,, 
\quad 
\lim_{\epsilon\to 0} \sup_{z\in \ms M_p}
\bb P_{z}^{\epsilon} [\, \tau (\ms E_{p-1} \setminus \ms E_p)
\le \vartheta^{(p)}_\epsilon \, ] \,=\, 0 \;,
\end{aligned}
\end{equation*}
where, recall from \eqref{24},
$\ms M_p = \cup_{k\in\bb Z} \ms M_p(k)$. By Corollary \ref{l32} and
Lemma \ref{l33}, these assertions holds provided
$\epsilon^{-1} \, e^{\mf h_{p-1}/\epsilon} \prec
\vartheta^{(p)}_\epsilon \prec e^{(\mf h_{p-1} + \kappa_0)/\epsilon}
$, where $\kappa_0$ is the postive constant introduced in Corollary
\ref{l32}. This completes the proof of \eqref{35} and the lemma.
\end{proof}

\subsection*{Proof of Theorem \ref{mt3}}

Fix $\mf n\ge 1$, $0<t_1< \dots <t_{\mf n}$,
$k_1, \dots, k_{\mf n} \in \bb Z$, $x\in \bb R$.  Fix a time-scale
$\vartheta_\epsilon$ such that
$\epsilon^{-1} e^{\mf h_{p-1}/\epsilon} \prec \vartheta_\epsilon \prec
e^{\mf h_{p}/\epsilon} $.  Then, by Lemma \ref{l33},
\begin{equation}
\label{86}
\begin{aligned}
& \bb P^{p,\epsilon}_x\Big[\,
\bigcap_{j=1}^{\mf n} \big\{\, \mtt x (t_{j})
\in \ms E (\ms M_p(k_j))\,\big\}\,\Big]
\\ &\quad 
\,=\,
\bb P^\epsilon_x\Big[\, \tau(\ms M_{p-1}) < \vartheta_\epsilon \,,\,
\bigcap_{j=1}^{\mf n} \big\{\, \mtt x (\theta_{\epsilon}^{(p)}
t_{j}) \in \ms E (\ms M_p(k_j))\,\big\}\,\Big]
\,+\, R^{(1)}_\epsilon (x) \,,
\end{aligned}
\end{equation}
where \eqref{88} holds for $R=R^{(i)}$.

As $\vartheta_\epsilon \prec e^{\mf h_{p}/\epsilon} $ y the strong
Markov property, the first term on the right-hand side of the
penultimate displayed formula is equal to
\begin{equation*}
\bb E^\epsilon_x\Big[\,
\mtt 1\big\{ \tau(\ms M_{p-1}) < \vartheta_\epsilon\,\big\}  \,
\bb P^\epsilon_{\mtt x(\tau(\ms M_{p-1}))} \Big[\,
\bigcap_{j=1}^{\mf n} \Big\{\, \mtt x (\theta_{\epsilon}^{(p)}
t_{j} - \tau(\ms M_{p-1}) ) \in
\ms E (\ms M_p(k_j))\,\Big\}\,\Big]\, \Big]\,. 
\end{equation*}
Recall from \eqref{85} the definition of $\mtt l_{p-1}(x)$,
$\mtt r_{p-1}(x)$. Let
$\color{blue} \mf m(\mtt l_{p-1}(x)) = m^+_{p-1, \mtt l_{p-1}(x)}$,
$\color{blue} \mf m(\mtt r_{p-1}(x)) = m^-_{p-1, \mtt r_{p-1}(x)}$
Clearly, $\bb P^\epsilon_x$-almost surely,
$\mtt x (\tau(\ms M_{p-1})) = \mf m(\mtt l_{p-1}(x))$ or
$\mtt x (\tau(\ms M_{p-1})) = \mf m(\mtt r_{p-1}(x))$. By Proposition
\ref{l11}, since $\vartheta_\epsilon \prec e^{\mf h_{p}/\epsilon}$,
for $k\in {\color{blue} I_x  :=\{\mtt l_{p-1}(x), \mtt r_{p-1}(x) \}}$,
\begin{equation*}
\lim_{\epsilon\to 0}
\bb P^\epsilon_{\mf m (k)} \Big[\,
\bigcap_{j=1}^{n} \Big\{\, \mtt x (\theta_{\epsilon}^{(p)}
t_{j} - \tau(\ms M_{p-1}) ) \in
\ms E (\ms M_p(k_j))\,\Big\}\,\Big]\,
\,=\, \mtt Q^{(p)}_{\ms M_p(k)}
\big[\, \bigcap_{j=1}^{{\mf n}}
\big\{ \, \mtt x (t_{j}) = \ms M_p(k_j) \, \} \, \, \big]
\,. 
\end{equation*}
Thus, the first term on the right-hand side of \eqref{86} is equal to
\begin{equation*}
\sum_{k\in I_x} \bb P^\epsilon_x\big[\,
\tau(\ms M_{p-1}) < \vartheta_\epsilon \,,\,
\mtt x(\tau(\ms M_{p-1})) = \mf m(k) \,\big]\,
\mtt Q^{(p)}_{\ms M_p(k)}
\big[\, \bigcap_{j=1}^{{\mf n}}
\big\{ \, \mtt x (t_{j}) = \ms M_p(k_j) \, \} \, \, \big]
\,+\, R^{(2)}_\epsilon (x) \,,
\end{equation*}
where \eqref{88} holds for $R=R^{(2)}$. By Lemma \ref{l33}, the
previous expression is equal to
\begin{equation*}
\sum_{k\in I_x} \bb P^\epsilon_x\big[\,
\mtt x(\tau(\ms M_{p-1})) = \mf m(k) \,\big]\,
\mtt Q^{(p)}_{\ms M_p(k)}
\big[\, \bigcap_{j=1}^{{\mf n}}
\big\{ \, \mtt x (t_{j}) = \ms M_p(k_j) \, \} \, \, \big]
\,+\, R^{(3)}_\epsilon (x) \,,
\end{equation*}
where \eqref{88} holds for $R=R^{(3)}$. Recall the definition of the
measure $\mtt h_p(x, \cdot)$ introduced in \eqref{74}. By Corollary
\ref{l26}, the previous expression is equal to
\begin{equation*}
\begin{aligned}
& \sum_{k\in I_x} \mtt h_p(x, \ms M_p(k))\,
\mtt Q^{(p)}_{\ms M_p(k)}
\big[\, \bigcap_{j=1}^{{\mf n}}
\big\{ \, \mtt x (t_{j}) = \ms M_p(k_j) \, \} \, \, \big]
\,+\, R^{(4)}_\epsilon (x)
\\
& \quad 
\,=\,
\mtt Q^{(p)}_{\mtt h_p(x, \cdot) }
\big[\, \bigcap_{j=1}^{{\mf n}}
\big\{ \, \mtt x (t_{j}) = \ms M_p(k_j) \, \} \, \, \big]
\,+\, R^{(4)}_\epsilon (x)
\,,
\end{aligned}
\end{equation*}
where \eqref{88} holds for $R=R^{(4)}$. This completes the proof of
the theorem. \qed

\section{Proof of Theorem \ref{mt1} and \ref{mt2}}
\label{sec9}

The proof relies on Lemma \ref{l36} below, which refines Proposition
\ref{l11}.  The first result is a simple consequence of the formula
\eqref{50} for the jump rates of the Markov chain $\bb X_p(\cdot)$,
and the fact that the jumps are strictly positive among the
nearest-neighbours of a $\bb X_p(\cdot)$-closed irreducible class.

\begin{lemma}
\label{l39}
Fix $1\le p< \mf q$, $k\in \bb Z$. Denote by $\ms M_{p}(i)$,
$\ell_1\le i\le \ell_2$, the sets forming the $\bb X_p$-closed
irreducible class $\ms M_{p+1}(k)$.  Then, the
stationary state of the Markov chain $\bb X_p(\cdot)$ restricted to
$\ms M_{p+1}(k)$ is given by $\pi(\ms M_{p}(i))/\pi(\ms M_{p+1}(k))$,
$\ell_1\le i\le \ell_2$.
\end{lemma}

The proof of the next lemma is similar to the one of \cite[Lemma
12.8]{lls2}. We present it here for completeness.

\begin{lemma}
\label{l36}
Fix $1\le p\le \mf q$, $j$, $t>0$, $k\in\bb Z$ and $m \in \ms M_p(k)$.
Then, for any sequence $(\vartheta^{(p)}_{\epsilon} :\epsilon>0)$ such
that $\vartheta^{(p)}_{\epsilon}\prec\theta_{\epsilon}^{(p)}$,
\begin{equation*}
\lim_{\epsilon\to 0}\, \sup_{x \in\mathcal{E}(\ms M_p(j))}\,
\sup_{|s|\le\vartheta^{(p)}_\epsilon}\,\Big|\,
\mathbb{P}_{x}^{\epsilon} \big[\,\mtt x(\theta_{\epsilon}^{(p)}t+s)
\in\mathcal{E}(m)\,\big] \,-\,
\frac{\pi (m)}{\pi(\ms M_p(k))}
\,\mtt {Q}_{\ms M_p(j)}^{(p)}\big[\, \mtt
x(t)= \ms M_p(k)\,\big]\,\Big|\ =\ 0\,. 
\end{equation*}
\end{lemma}

\begin{proof}
The proof is by induction. For $p=1$, as the sets 
$\ms M_1(\ell)$ are singletons, the assertion of the lemma
corresponds to the one of Proposition \ref{l11}.

Fix $p\ge2$ and assume that the lemma holds for $1\le p'<p$.  Fix 
$k\in \bb Z$, $m \in \ms M_p(k)$ and a sequence
$(\vartheta^{(p)}_\epsilon)_{\epsilon>0}$ such that
$\vartheta^{(p)}_\epsilon\prec\theta_{\epsilon}^{(p)}$.  By
construction,
$\ms M_p(k) = \cup_{\ell_1 \le i\le \ell_2} \ms M_{p-1}(i)$ for some
$\ell_1\le\ell_2$. Denote by $i_0 \in \{\ell_1, \dots, \ell_2\}$ the
index such that $m \in \ms M_{p-1}(i_0)$.

Fix $\eta>0$. Since
$\{ \ms M_{p-1}(i) : \ell_1 \le i\le \ell_2\}$ forms a
$\bb X_{p-1}$-recurrent class, by the paragraph preceding the
statement of the lemma, there exists $s_{0}=s_{0}(\eta)>0$ such that
\begin{equation}
\label{93}
\max_{\ell_1 \le i \le \ell_2}\,\Big |\,
\mtt Q_{\ms M_{p-1}(i)}^{(p-1)} \big[\, \mtt x (s_{0}) =
\ms M_{p-1}(i_0)\,\big]
\,-\, \frac{\pi(\ms M_{p-1}(i_0))}{\pi(\ms M_p(k))}\,\Big | \,\le\,
\eta\,.
\end{equation}

Fix $x\in \ms E(\ms M_p(k))$, $|s|\le\vartheta^{(p)}_\epsilon$,
and write
\begin{equation*}
\bb P_x^{\epsilon }\big[\, \mtt x(\theta_{\epsilon}^{(p)}t+s) \in
\ms E(m) \,\big]
\,=\,  \sum_{j=1}^{3} \bb P_x^{\epsilon } \big[\, \mtt x
(\theta_{\epsilon}^{(p)}t+s) \in \ms E(m)
\,\big|\,\mathscr{A}_{j}\,\big]\,
\bb P_x^{\epsilon } [\,\mathscr{A}_{j}\,]\;,
\end{equation*}
where
\begin{gather*}
\mathscr{A}_{1}\ =\ \{\, \mtt x (t_{\epsilon}) \in \ms E(\ms M_{p}(k))
\,\} \;,\quad \mathscr{A}_{2}\ =\ \{\, \mtt x (t_{\epsilon}) \in \ms
E(\ms M_{p}) \setminus \ms E(\ms M_{p}(k)) \,\}\;, \quad
\mathscr{A}_{3}\ =\ \{\, \mtt x (t_{\epsilon})\not\in \ms E(\ms M_{p})
\,\}\;,
\end{gather*}
and $t_{\epsilon}=\theta_{\epsilon}^{(p)}t+s-\theta_{\epsilon}^{(p-1)}s_{0}$.

We claim that the term corresponding to $\mathscr{A}_{2}$ is
negligible.  Bound $\bb P_x^{\epsilon } [\,\mathscr{A}_{2}\,]$ by
$1$ and apply the strong Markov property to get that this term is less
than or equal to
\begin{equation*}
\sup_{k'\neq k}\,\sup_{y\in \ms M_p(k')}\,
\mathbb{P}_{y}^{\epsilon}\big[\,\tau_{\ms E(\ms M_p(k))}
<\theta_{\epsilon}^{(p-1)}s_{0}\,\big]\;.
\end{equation*}
By Corollary \ref{l07} this expression vanishes as $\epsilon\to 0$.

By \eqref{33}, the term corresponding to $\mathscr{A}_{3}$ is also
negligible.  We turn to the term corresponding to
$\mathscr{A}_{1}$. We estimate the two probabilities separately. For
the first probability, by the strong Markov property,
\begin{align*}
& \sup_{|s|\le\vartheta^{(p)}_\epsilon}\, \Big|\,
\bb P_x^{\epsilon } \big[\, \mtt x
(\theta_{\epsilon}^{(p)}t+s) \in \ms E(m)
\,\big|\,\mathscr{A}_1 \,\big]
\,-\, \frac{\pi(m)}{\pi(\ms M_{p}(k))}
\, \Big|
\\
&\quad \le
\sup_{y\in \ms E(\ms M_p(k))}\, \Big|\,
\bb P_y^{\epsilon } \big[\, \mtt x
(\theta_{\epsilon}^{(p-1)}s_0) \in \ms E(m) \,\big]
\,-\, \frac{\pi(m)}{\pi(\ms M_{p}(k))}
\, \Big|\;.
\end{align*}
This expression is less than or equal to
\begin{align*}
& \max_{\ell_1\le i\le \ell_2}\, \sup_{y\in \ms E(\ms M_{p-1}(i))}
\,\Big|\, \bb P_y^{\epsilon } \big[\, \mtt x
(\theta_{\epsilon}^{(p-1)}s_0) \in \ms E(m) \,\big]
\,-\, \frac{\pi(m)}{\pi(\ms M_{p-1}(i_0))} 
\, \mtt Q^{(p-1)}_{\ms M_{p-1}(i)} \big[\,
\mtt x (s_{0}) = \ms M_{p-1}(i_0) \,\big] \,\Big|
\\
& \quad +\, \max_{\ell_1\le i\le \ell_2}\, \sup_{y\in \ms E(\ms M_{p-1}(i))}
\,\Big|\, \frac{\pi(m)}{\pi(\ms M_{p-1}(i_0))} 
\, \mtt Q^{(p-1)}_{\ms M_{p-1}(i)} \big[\,
\mtt x (s_{0}) = \ms M_{p-1}(i_0) \,\big]
\,-\,  \frac{\pi(m)}{\pi(\ms M_{p}(k))} \,\Big|\;.
\end{align*}
By the induction hypothesis, the first term converges to $0$ as
$\epsilon\to0$.  Therefore, by \eqref{93},
\begin{equation*}
\limsup_{\epsilon\to0}  \sup_{x \in\mathcal{E}(\ms M_p(j))}
\sup_{|s|\le\vartheta^{(p)}_\epsilon}\, \Big|\,
\bb P_x^{\epsilon } \big[\, \mtt x
(\theta_{\epsilon}^{(p)}t+s) \in \ms E(m)
\,\big|\,\mathscr{A}_1 \,\big]
\,-\, \frac{\pi(m)}{\pi(\ms M_{p}(k))}
\, \Big| \,\le\, \eta\;.
\end{equation*}

For the second probability, by Proposition \ref{l11},
\begin{equation*}
\limsup_{\epsilon\to0}  \sup_{x \in\mathcal{E}(\ms M_p(j))}
\sup_{|s|\le\vartheta^{(p)}_\epsilon}\, \Big|\,
\bb P_x^{\epsilon } \big[\, \mtt x (t_{\epsilon}) \in \ms E(\ms
M_{p}(k))  \,\big]
\,-\, \mtt Q^{(p-1)}_{\ms M_{p}(j)} \big[\,
\mtt x (s_{0}) = \ms M_{p}(k) \,\big]
\, \Big| \,=\, 0\;.
\end{equation*}
To complete the proof of the induction step and the one of the lemma,
it remains to recollect all the previous estimates, letting $\eta\to 0$
at the end.
\end{proof}

\begin{proof}[Proof of Theorem \ref{mt1} and \ref{mt2}.{\rm (a)}]
Fix a bounded, continuous function $u_0\colon \bb R \to \bb R$,
$x\in\bb R$, $t>0$, $1\le p\le \mf q$, and $\delta>0$. Let
$\mf m_{p,\mtt l}(x)$, $\mf m_{p,\mtt r}(x)$ be the $S$-local minima
given by
\begin{equation}
\label{100}
{\color{blue} \mf m_{p,\mtt l}(x)} \,:=\, 
\max\{ m\in \ms M_p : m\le x\,\}\,,\quad
{\color{blue} \mf m_{p,\mtt r}(x)} \,:=\,
\min\{ m\in \ms M_p : m\ge x\,\}\,,
\end{equation}
and let $\color{blue} \mtt k_{p,\mtt l}(x)$,
$\color{blue} \mtt k_{p,\mtt r}(x)$ be the integers such that
$\mf m_{p,\mtt l}(x) \in \ms M_p(\mtt k_{p,\mtt l}(x))$,
$\mf m_{p,\mtt r}(x) \in \ms M_p(\mtt k_{p,\mtt r}(x))$.  In the
notation introduced right before the statement of Theorem \ref{mt2},
$\mtt k_{p,\mtt l}(x) = \mtt l_p(x)$,
$\mtt k_{p,\mtt r}(x) = \mtt r_p(x)$

By the stochastic representation of the solution of the parabolic
equation \eqref{51}, \cite[Theorem 6.5.3]{Fri},
\begin{equation}
\label{90}
u(t \theta_{\epsilon}^{(p)},x) \;=\;
\bb E^{p,\epsilon}_x\big[ \,
u_0( \mtt x(t))\,\big] \;.
\end{equation}

\smallskip\noindent{\bf Step 1: Reduction to $x\in \ms M_p$.}  Fix a
time-scale $\vartheta^{(p)}_\epsilon \prec \theta^{(p)}_\epsilon$.
Repeat the arguments presented in the proof of Theorem \ref{mt3},
introducing the restriction
$\{\tau(\ms M_p) < \vartheta^{(p)}_\epsilon\}$, and then
$\{\mtt x(\tau(\ms M_{p})) = m \}$, $m = \mf m_{p,\mtt l}(x)$,
$\mf m_{p,\mtt r}(x)$ (see equation \eqref{86}, and below), to obtain that
\eqref{90} is equal to
\begin{equation}
\label{91}
\sum_{m\in M_p(x)} \bb E^{\epsilon}_{x} \Big[ \, \mtt 1\big\{ \tau(\ms M_p) <
\vartheta^{(p)}_\epsilon\,,\,
\mtt x(\tau(\ms M_{p})) = m \big\}\,
\bb E^{\epsilon}_{m} \big[ u_0( \mtt x(t\theta^{(p)}_\epsilon -
\tau(\ms M_p))) \,\big] \,\Big] \,+\, R^{(1)}_\epsilon (x)\;,
\end{equation}
where $R^{(1)}_\epsilon (x)$ satisfies \eqref{88},
${\color{blue} M_p(x)} := \{\mf m_{p,\mtt l}(x) , \mf m_{p,\mtt
r}(x)\}$, 

\smallskip\noindent{\bf Step 2: Metastability.}  It remains to
estimate
$\bb E^{\epsilon}_{m} \big[ u_0( \mtt x(t\theta^{(p)}_\epsilon - s))
\,\big]$ uniformly on $|s|\le \vartheta^{(p)}_\epsilon$ for
$m\in M_p(x)$.

Fix $\delta>0$.  By postulate $\mc P_9(p)$, the jump rates $R_p$ are
periodic. Thus, the holding times at each state $\ms M_p(j)$ is
uniformly bounded. As the Markov chain $\bb X_p(\cdot)$ jumps only to
nearest-neighbour sites, there exists, therefore,
$\ell_0 = \ell_0(t, \bb X_p(\cdot))$ such that
\begin{equation}
\label{94}
\sum_{j: |j-k|\ge \ell_0} p^{(p)}_t (\ms M_p(k), \ms M_p(j))
\,\le\,\delta \quad \text{for all}\;\; k\in\bb Z \;. 
\end{equation}

Let $k_m = \mtt k_{p,\mtt l}(x)$, if $m= \mf m_{p,\mtt l}(x)$ and
$k_m = \mtt k_{p,\mtt r}(x)$, if $m= \mf m_{p,\mtt r}(x)$.  Recall
from \eqref{71} the definition of the wells $\ms E(m)$, and choose
$r_0$ small enough so that
\begin{equation*}
\max_{j: |j-k_m|\le \ell_0}\sup_{y\in\ms E(m_j)} |u_0(y) -
u_0(m_j)| \,\le\, \delta \;.
\end{equation*}
By the previous estimate, for each $m\in M_p(x)$,
\begin{equation*}
\bb E^{\epsilon}_{m} \big[ u_0( \mtt x(t\theta^{(p)}_\epsilon - s))
\,\big] 
\,=\, \sum_{j: |j-k_m|\le \ell_0} \sum_{m'\in \ms M_p(j)} u_0(m')\, 
\bb P^{\epsilon}_m\big[ \, 
\mtt x(t\theta^{(p)}_\epsilon - s) \in \ms E(m') \, \big]
\,+\, R^{(2)}_\epsilon \,+\, R_\delta \;,
\end{equation*}
where $|R_\delta| \le \delta$, and 
\begin{equation*}
|R^{(2)}_\epsilon| \,\le\, \Vert u_0\Vert_\infty\;
\bb P^{\epsilon}_m\Big[ \, 
\mtt x(t\theta^{(p)}_\epsilon - s) \not\in \bigcup_{j: |j-k_m| \le  \ell_0}
\ms E(\ms M_p(j)) \,\Big] \;.
\end{equation*}
By Proposition \ref{l11} and \eqref{94},
\begin{equation*}
\begin{aligned}
& \lim_{\epsilon\to 0} \sup_{|s|\le \vartheta^{(p)}_\epsilon}
\bb P^{\epsilon}_m\Big[ \, 
\mtt x(t\theta^{(p)}_\epsilon - s) \not \in \bigcup_{j: |j-k_m| \le \ell_0}
\ms E(\ms M_p(j)) \,\Big]
\\
&\quad
\,=\, \lim_{\epsilon\to 0}  \sup_{|s|\le \vartheta^{(p)}_\epsilon}
\Big\{\, 1\, -\, \bb P^{\epsilon}_m\Big[ \, 
\mtt x(t\theta^{(p)}_\epsilon - s) \in \bigcup_{j: |j-k_m| \le \ell_0}
\ms E(\ms M_p(j)) \,\Big] \,\Big\} \,\le\,\delta\;.
\end{aligned}
\end{equation*}

Therefore, by Lemma \ref{l36},
\begin{equation*}
\lim_{\epsilon\to 0} \sup_{|s|\le \vartheta^{(p)}_\epsilon}
\Big|\, \bb E^{\epsilon}_{m} \big[ u_0( \mtt x(t\theta^{(p)}_\epsilon - s))
\,\big]  -\, U_p(t,k_m, \ell_0)  \,\,\Big|
\,\le\,  \delta\, (1+\Vert u_0\Vert_\infty)\;.
\end{equation*}
where
\begin{equation}
\label{95}
U_p(t,k_m, \ell_0) \,:=\, \sum_{j: |j-k_m|\le \ell_0}
\mtt {Q}_{\ms M_p(k_m)}^{(p)}\big[\, \mtt x(t)= \ms M_p(j)\,\big]
\sum_{m'\in \ms M_p(j)} u_0(m')\, 
\frac{\pi (m')}{\pi(\ms M_p(j))}
\end{equation}
In view of \eqref{94}, paying an extra
$\delta\, \Vert u_0\Vert_\infty$, we may replace in the penultimate
formula $U_p(t,k_m, \ell_0)$ by $U_p(t,k_m)$, where this later
function is defined as in the previous displayed equation, but with
the sum running over $j\in \bb Z$.

\smallskip\noindent{\bf Step 3: Conclusion.} In view of the previous
estimate, \eqref{91} is equal to
\begin{equation*}
\sum_{m\in M_p(x)} U_p(t,k_m) \, \bb P^{\epsilon}_{x}
\big[ \, \tau(\ms M_p) <
\vartheta^{(p)}_\epsilon\,,\,
\mtt x(\tau(\ms M_{p})) = m \,\big] \,+\, R^{(3)}_{\epsilon, \delta} (x)\;,
\end{equation*}
where
\begin{equation*}
\limsup_{\epsilon\to 0} |\, R^{(3)}_{\epsilon, \delta} (x) \,| \,\le
\, (1+ 2\Vert u_0\Vert_\infty) \, \delta  \;.
\end{equation*}
By Lemma \ref{l33} we may remove the indicator of the event
$\{ \tau(\ms M_p) < \vartheta^{(p)}_\epsilon\}$, provided we choose a
sequence $\vartheta^{(p)}_\epsilon$ such that
$\vartheta^{(p)}_\epsilon \succ \epsilon^{-1}\,
\theta^{(p-1)}_\epsilon$.  This can be done since the only restriction
required up to this point is that
$\vartheta^{(p)}_\epsilon \prec \theta^{(p)}_\epsilon$.  To complete
the proof of the theorem, it remains to apply Corollary \ref{l26} and
to let $\delta\to 0$.
\end{proof}

\subsection*{Proof of Theorem \ref{mt2}.{\rm (b)}}

The Theorem \ref{mt2} for intermediate scales proof's idea is quite
simple. In a time scale of order
$\theta^{(p)}_\epsilon \prec \varrho_\epsilon \prec
\theta^{(p+1)}_\epsilon$, the process reaches a set $\ms M_{p+1}(k)$
for some $k\in\bb Z$ (actually, the closest one to the left or right
of the starting point) and is not able to leave this set because
$\varrho_\epsilon \prec \theta^{(p+1)}_\epsilon$. Since
$\varrho_\epsilon \succ \theta^{(p)}_\epsilon$, and $\ms M_{p+1}(k)$
is a $\bb X_p$-closed irreducible class, the distribution of
$X_\epsilon(\varrho_\epsilon)$ among the wells $\ms M_p(j)$ which
compose $\ms M_{p+1}(k)$ will be given by the stationary state of the
reduced Markov chain $\bb X_p(\cdot)$ restricted to the set
$\ms M_{p+1}(k)$. The same argument may be applied to the distribution
inside the well $\ms M_{p}(j)$, and an inductive argument permits to
express the distribution of $X_\epsilon(\varrho_\epsilon)$ among the
wells $\ms E(m)$, $m\in \ms M_{p+1}(k)$, in terms of the stationary
states of the Markov chains $\bb X_1(\cdot), \dots, \bb X_p(\cdot)$.

Next result follows from the previous lemma, which states that the set
$\ms M_{p+1}$ is reached in times of order $\theta^{(p)}_\epsilon$,
and Corollary \ref{l32}, which asserts that starting from
$\ms M_{p+1}$ a time-scale or order
$\theta^{(p)}_\epsilon\, e^{\kappa_0/\epsilon}$ is required to attain
the set $\ms E_p \setminus \ms E_{p+1}$.

\begin{corollary}
\label{l41}
Fix $1\le p< \mf q$, and sequences $\varrho_\epsilon$,
$\vartheta_\epsilon$ such that
$\theta^{(p)}_\epsilon \prec \varrho_\epsilon$,
$\vartheta^{(p)}_\epsilon\prec \varrho_\epsilon$.  Then,
\begin{equation*}
\lim_{\epsilon\to 0} \sup_{x\in \bb R}
\sup_{|s| \le \vartheta^{(p)}_\epsilon}
\bb P^{\epsilon}_x\big[ \, 
\mtt x(\varrho_\epsilon - s) \in \ms E_p
\setminus \ms E_{p+1} \,\big] \,=\, 0 \;.
\end{equation*}
\end{corollary}

\begin{proof}
Recall the definition of $\kappa_0$ introduced in Corollary \ref{l32},
and fix $\delta>0$. By Corollary \ref{l40}, we may choose $t>0$,
$\epsilon_0>0$ such
that
\begin{equation*}
\sup_{x\in \bb R} \bb P^{\epsilon}_x\big[ \, 
\tau (\ms M_{p+1}) > t \, \theta^{(p)}_\epsilon  \, \big] \,\le \, \delta
\end{equation*}
for all $0<\epsilon<\epsilon_0$.

Fix $|s| \le \vartheta^{(p)}_\epsilon$. If
$\varrho_\epsilon - s \le e^{\kappa_0/\epsilon} \,
\theta^{(p)}_\epsilon$ we estimate the probability appearing in the
statement of the corollary by
\begin{equation*}
\bb P^{\epsilon}_x\big[ \, \tau (\ms M_{p+1}) < t\,
\theta^{(p)}_\epsilon \,,\,
\mtt x(\varrho_\epsilon - s) \in \ms E_p
\setminus \ms E_{p+1} \,\big] \,+\, \delta \;.
\end{equation*}
By the strong Markov property, and since
$\varrho_\epsilon - s \le e^{\kappa_0/\epsilon} \,
\theta^{(p)}_\epsilon$, the probability appearing in this expression
is bounded by
\begin{equation*}
\sup_{m\in \ms M_{p+1}}
\sup_{0\le s'  \le t\theta^{(p)}_\epsilon }
\bb P^{\epsilon}_m\big[ \, 
\mtt x(\varrho_\epsilon - s - s') \in \ms E_p
\setminus \ms E_{p+1} \,\big] \,\le\,
\sup_{m\in \ms M_{p+1}} \bb P^{\epsilon}_m\big[ \, 
\tau(\ms E_p \setminus \ms E_{p+1}) \le e^{\kappa_0/\epsilon} \,
\theta^{(p)}_\epsilon \,\big] \;.
\end{equation*}
By Corollary \ref{l32}, this expression is bounded by $C_0 \,\epsilon \,
e^{-\kappa_0/\epsilon}$.

If
$\varrho_\epsilon - s > e^{\kappa_0/\epsilon} \,
\theta^{(p)}_\epsilon$ we first apply the Markov property (at time
$\varrho_\epsilon - s - e^{\kappa_0/\epsilon}$) to estimate the
probability appearing in the statement of the corollary by
\begin{equation*}
\sup_{x\in \bb R}
\bb P^{\epsilon}_x\big[ \, 
\mtt x(e^{\kappa_0/\epsilon} \, \theta^{(p)}_\epsilon) \in \ms E_p
\setminus \ms E_{p+1} \,\big]\;.
\end{equation*}
Then, we repeat the previous argument.
\end{proof}

\begin{lemma}
\label{l38}
Fix $1\le p< \mf q$, $k\in \bb Z$, and sequences $\varrho_\epsilon$,
$\vartheta_\epsilon$ such that
$\theta^{(p)}_\epsilon \prec \varrho_\epsilon \prec
\theta^{(p+1)}_\epsilon$, $\vartheta_\epsilon\prec \varrho_\epsilon$.
Then,
\begin{equation*}
\lim_{\epsilon\to 0} \sup_{|s|\le \vartheta_\epsilon} \sup_{x\in \ms E(\ms M_{p+1}(k))}
\Big|\,  \bb P^{\epsilon}_x\big[ \, 
\mtt x(\varrho_\epsilon + s) \in \ms E(m') \, \big]
\,-\, \frac{\pi(m')}{\pi(\ms M_{p+1}(k))} \,\Big| \,=\, 0
\end{equation*}
for all $m'\in \ms M_{p+1}(k)$.
\end{lemma}

\begin{proof}
The proof is similar to the one of Lemma \ref{l36}. Fix $\delta>0$,
$k\in \bb Z$, $x\in \ms E(\ms M_{p+1}(k))$. Assume that
$\ms M_{p+1}(k) = \cup_{\ell_1 \le i\le \ell_2} \ms M_{p}(i)$. By
Lemma \ref{l39}, there exists $t_{0}=t_{0}(\delta)>0$ such that
\begin{equation*}
\max_{\ell_1 \le i, j\le \ell_2}\,
\Big|\, \mtt Q^{(p)}_{\ms M_p(i)} \big[ \, \mtt x(t) = \ms M_p(j)\, \big]
\,-\, \frac{\pi (\ms M_p(j))}{\pi (\ms M_{p+1}(k))}\,
\Big|\,\le\,\delta 
\end{equation*}
for all $t\ge t_0$.

Fix $m'\in \ms M_{p+1}(k)$.  Write the probability appearing in the
lemma as
\begin{equation*}
\sum_{j=1}^{2} \bb P_x^{\epsilon } \big[\, \mtt x
(\varrho_\epsilon + s) \in \ms E(m')
\,\big|\,\mathscr{A}_{j}\,\big]\,
\bb P_x^{\epsilon } [\,\mathscr{A}_{j}\,]\;,
\end{equation*}
where
\begin{equation*}
\ms A_1 \,=\, \{ \mtt x (t_\epsilon) \in \ms E(\ms M_{p+1}(k)) \}\,,
\quad \ms A_2 \,=\, \{ \mtt x (t_\epsilon) \not \in \ms 
E( \ms M_{p+1}(k)) \}\,, 
\end{equation*}
and $t_{\epsilon}=\varrho_{\epsilon}+s-\theta_{\epsilon}^{(p)} t_{0}$.

As $\varrho_\epsilon \succ \theta_{\epsilon}^{(p)}$, by \eqref{33} (to
estimate the probability of the event
$\{\mtt x (t_\epsilon) \not \in \ms E_p\}$), and Corollary \ref{l40}
and Lemma \ref{l33} (to estimate the probability of the event
$\{\mtt x (t_\epsilon) \in \ms E_p \setminus \ms E_{p+1}\}$),
\begin{equation*}
\lim_{\epsilon\to 0} \sup_{|s|\le \vartheta_\epsilon} \sup_{x\in \ms E(\ms M_{p+1}(k))}
\bb P_x^{\epsilon } [\, \mtt x (t_\epsilon)
\not\in \ms E_{p+1}\,] \,=\, 0\;.
\end{equation*}
Hence, as $\varrho_\epsilon \prec \theta_{\epsilon}^{(p+1)}$, by
Corollary \ref{l07},
\begin{equation*}
\lim_{\epsilon\to 0} \sup_{|s|\le \vartheta_\epsilon} \sup_{x\in \ms E(\ms M_{p+1}(k))}
\bb P_x^{\epsilon } [\,\mathscr{A}_2 \,] \,=\, 0\;.
\end{equation*}
We turn to the first term. By the previous remark,
$\bb P_x^{\epsilon } [\,\mathscr{A}_1 \,]$ converges to $1$, uniformly
on $x\in \ms E(\ms M_{p+1}(k))$, $|s|\le \vartheta_\epsilon$. On the
other hand, we claim that
\begin{equation*}
\lim_{\epsilon\to 0} \sup_{x\in \ms E(\ms M_{p+1}(k))}
\sup_{|s|\le \vartheta_\epsilon}
\Big|\, \bb P_x^{\epsilon } \big[\, \mtt x
(\varrho_\epsilon + s) \in \ms E(m')
\,\big|\,\mathscr{A}_1 \,\big]
\,-\, \frac{\pi(m')}{\pi(\ms M_{p+1}(k))} \,\Big| \,=\, 0\;.
\end{equation*}

Indeed, by the Markov property, this expression is bounded by
\begin{equation*}
\max_{\ell_1\le i\le \ell_2} \sup_{y\in \ms E(\ms M_p(i))}
\Big|\, \bb P_y^{\epsilon } \big[\, \mtt x
(\theta_{\epsilon}^{(p)}\, t_{0}) \in \ms E(m')\,\big]
\,-\, \frac{\pi(m')}{\pi(\ms M_{p+1}(k))} \,\Big|\;.
\end{equation*}

Choose $\ell_1\le i_0\le \ell_2$ such that $m'\in \ms M_p(i_0)$, and bound
the previous expression by
\begin{equation*}
\begin{aligned}
& \max_{\ell_1\le i\le \ell_2} \sup_{y\in \ms E(\ms M_p(i))}
\Big|\, \bb P_y^{p,\epsilon } \big[\, \mtt x
(t_{0}) \in \ms E(m')\,\big]
\,-\, \mtt Q^p_{\ms M_p(i)} \big[\, \mtt x(t_0) = \ms M_p(i_0)\,]
\, \frac{\pi(m')}{\pi(\ms M_{p}(i_0))} \,\Big|
\\
&\quad +\,
\max_{\ell_1\le i\le \ell_2} 
\Big|\, \mtt Q^p_{\ms M_p(i)} \big[\, \mtt x(t_0) = \ms M_p(i_0)\,]
\, \frac{\pi(m')}{\pi(\ms M_{p}(i_0))} \,-\,
\frac{\pi(m')}{\pi(\ms M_{p+1}(k))} \,\Big|
\end{aligned}
\end{equation*}
By the choice of $t_0$, the second term is bounded by $\delta$.
By Lemma \ref{l36}, the first one vanishes as $\epsilon\to 0$. This
completes the proof of the lemma.
\end{proof}

\begin{proof}[Proof of Theorem \ref{mt2}.{\rm (b)}]
Fix a bounded, continuous function $u_0\colon \bb R \to \bb R$,
$x\in\bb R$, $1\le p< \mf q$, a sequence $\varrho_\epsilon$
such that
$\theta^{(p)}_\epsilon \prec \varrho_\epsilon \prec
\theta^{(p+1)}_\epsilon$, and $\delta>0$.
By the stochastic representation of the solution of the parabolic
equation \eqref{51}, \cite[Theorem 6.5.3]{Fri},
\begin{equation}
\label{90b}
u(\varrho_\epsilon,x) \;=\;
\bb E^{\epsilon}_x\big[ \,
u_0( \mtt x(\varrho_\epsilon))\,\big] \;.
\end{equation}

The proof is similar to the one of Theorem \ref{mt2}.(a).  Recall
the definition of the $S$-local minima $\mf m_{p+1,\mtt l}(x)$,
$\mf m_{p+1,\mtt r}(x)$ introduced in \eqref{100}.  Fix a time-scale
$\theta^{(p)}_\epsilon\prec \vartheta^{(p)}_\epsilon \prec
\varrho_\epsilon$.  Repeat the arguments presented in the proof of
Theorem \ref{mt3}, introducing the restriction
$\{\tau(\ms M_{p+1}) < \vartheta^{(p)}_\epsilon\}$, and then
$\{\mtt x(\tau(\ms M_{p+1})) = m \}$, $m = \mf m_{p+1,\mtt l}(x)$,
$\mf m_{p+1,\mtt r}(x)$, to obtain that \eqref{90b} is equal to
\begin{equation}
\label{91b}
\sum_{m\in M_{p+1}(x)} \bb E^{\epsilon}_{x}
\Big[ \, \mtt 1\big\{ \tau(\ms M_{p+1}) <
\vartheta^{(p)}_\epsilon\,,\,
\mtt x(\tau(\ms M_{p+1})) = m \big\}\,
\bb E^{\epsilon}_{m} \big[ u_0( \mtt x(\varrho_\epsilon -
\tau(\ms M_{p+1}))) \,\big] \,\Big] \,+\, R^{(1)}_\epsilon (x)\;,
\end{equation}
where
$M_{p+1}(x) = \{\mf m_{p+1,\mtt l}(x) , \mf m_{p+1,\mtt r}(x)\}$. By
Corollary \ref{l40}, $R^{(1)}_\epsilon (x)$ satisfies \eqref{88}.

Let $k_m = \mtt k_{p+1,\mtt l}(x)$ if $m= \mf m_{p+1,\mtt l}(x)$, and
$k_m = \mtt k_{p+1,\mtt r}(x)$ if $m= \mf m_{p+1,\mtt r}(x)$.  Recall
from \eqref{71} the definition of the wells $\ms E(m)$, and choose
$r_0$ small enough so that
\begin{equation*}
\max_{m\in \ms M_{p+1}(k_m) } \sup_{y\in\ms E(m)} |u_0(y) -
u_0(m)| \,\le\, \delta \;.
\end{equation*}
By the previous estimate, for each $m\in M_{p+1}(x)$, $|s| \le
\vartheta^{(p)}_\epsilon$, 
\begin{equation}
\label{87}
\bb E^{\epsilon}_{m} \big[ u_0( \mtt x(\varrho_\epsilon - s))
\,\big] 
\,=\, \sum_{m'\in \ms M_{p+1}(k_m)} u_0(m')\, 
\bb P^{\epsilon}_m\big[ \, 
\mtt x(\varrho_\epsilon - s) \in \ms E(m') \, \big]
\,+\, R^{(2)}_\epsilon (s)\,+\, R_\delta \;,
\end{equation}
where $|R_\delta| \le \delta$, and 
\begin{equation*}
|R^{(2)}_\epsilon (s)| \,\le\, \Vert u_0\Vert_\infty\;
\bb P^{\epsilon}_m\Big[ \, 
\mtt x(\varrho_\epsilon - s) \not\in \ms E(\ms M_{p+1}(k_m)) \,\Big] \;.
\end{equation*}

We claim that
\begin{equation}
\label{84}
\lim_{\epsilon\to 0} \sup_{|s| \le \vartheta^{(p)}_\epsilon}
\bb P^{\epsilon}_m\big[ \, 
\mtt x(\varrho_\epsilon - s) \not\in  \ms E(\ms M_{p+1}(k_m)) \,\big] \,=\, 0\;.
\end{equation}
Since
$\varrho_\epsilon \succ \vartheta^{(p)}_\epsilon \succ
\theta^{(p)}_\epsilon$, by \eqref{33} and Corollary \ref{l41},
\begin{equation*}
\lim_{\epsilon\to 0} \sup_{|s| \le \vartheta^{(p)}_\epsilon}
\bb P^{\epsilon}_m\big[ \, 
\mtt x(\varrho_\epsilon - s) \not\in \ms E_p
\,\big] \,=\, 0\;, \quad
\lim_{\epsilon\to 0} \sup_{|s| \le \vartheta^{(p)}_\epsilon}
\bb P^{\epsilon}_m\big[ \, 
\mtt x(\varrho_\epsilon - s) \in \ms E_p
\setminus \ms E_{p+1} \,\big] \,=\, 0 \;.
\end{equation*}
On the other hand, as
$\varrho_\epsilon + \vartheta^{(p)}_\epsilon \prec
\theta^{(p+1)}_\epsilon$, by Corollary \ref{l07},
\begin{equation*}
\lim_{\epsilon\to 0} \sup_{|s| \le \vartheta^{(p)}_\epsilon}
\bb P^{\epsilon}_m\big[ \, 
\mtt x(\varrho_\epsilon - s) \in \ms E_{p+1}
\setminus \ms E(\ms M_{p+1}(k_m)) \,\big] \,=\, 0 \;,
\end{equation*}
which completes the proof of \eqref{84}.

By Lemma \ref{l38}, the first term on the right-hand side of
\eqref{87} converges to
\begin{equation*}
U_{p+1} (m) \, :=\,
\sum_{m'\in \ms M_{p+1}(k_m)}  u_0(m')\,
\frac{\pi(m')}{\pi(\ms M_{p+1}(k_m))}\;\cdot
\end{equation*}

In view of the previous
estimate, \eqref{91b} is equal to
\begin{equation*}
\sum_{m\in M_p(x)} U_{p+1} (m) \, \bb P^{\epsilon}_{x}
\big[ \, \tau(\ms M_{p+1}) <
\vartheta^{(p)}_\epsilon\,,\,
\mtt x(\tau(\ms M_{p+1})) = m \,\big] \,+\, R^{(3)}_{\epsilon, \delta} (x)\;,
\end{equation*}
where
\begin{equation*}
\limsup_{\epsilon\to 0} \sup_{x\in\bb R}
|\, R^{(3)}_{\epsilon, \delta} (x) \,| \,\le
\, (1+ 2\Vert u_0\Vert_\infty) \, \delta  \;.
\end{equation*}
By Corollary \ref{l40} we may remove the indicator of the event
$\{ \tau(\ms M_{p+1}) < \vartheta^{(p)}_\epsilon\}$. To complete the proof
of the theorem, it remains to apply Corollary \ref{l26} and to let
$\delta\to 0$.
\end{proof}

\subsection*{Proof of Theorem \ref{mt2}.(c)}

The proof is identical to the one of part (b) provided we replace
$\theta^{(p)}_\epsilon$ by $\epsilon^{-1}$. There are only two
differences.

First, we apply Lemma \ref{l01}, instead of Corollary \ref{l40}, to
show that the term $R^{(1)}_\epsilon(x)$ in \eqref{91b} satisfies
\eqref{88}. This explains why we need to assume that
$\varrho_\epsilon \succ \epsilon^{-1}$.

Secondly, we replace Corollary \ref{l41} by the following result,
whose proof is analogous to the one of Corollary \ref{l41}. 
Fix  sequences $\varrho_\epsilon$,
$\vartheta_\epsilon$ such that
$\epsilon^{-1} \prec \varrho_\epsilon$,
$\vartheta_\epsilon\prec \varrho_\epsilon$.  Then,
\begin{equation*}
\lim_{\epsilon\to 0} \sup_{x\in \bb R}
\sup_{|s| \le \vartheta^{(p)}_\epsilon}
\bb P^{\epsilon}_x\big[ \, 
\mtt x(\varrho_\epsilon - s) \in \ms E_1^c \,\big] \,=\, 0 \;.
\end{equation*}

The details are left to the reader.

\qed

\subsection*{Proof of Theorem \ref{mt5}}

The proof is similar to the one of Theorem \ref{mt2}.(b).
Assume that the initial condition $u_0(\cdot)$ is $\ell$-periodic.
Fix $\delta>0$, and recall from \eqref{71} the definition of the wells
$\ms E(m)$. Since $u_0(\cdot)$ is periodic, we may assume that
$x\in [0, \ell)$, and choose $r_0$ small enough so that
\begin{equation}
\label{103}
\max_{m\in \mc M} \sup_{y\in\ms E(m)} |u_0(y) -
u_0(m)| \,\le\, \delta \;.
\end{equation}

Recall the definition of the Markov chain
$\bb X_{\mf q, \ell} (\cdot)$, whose stationary state is
represented by $\pi_{\mf q, \ell}$. As the chain is
irreducible, for every $\delta>0$, there exists $t_0>0$ such that
\begin{equation*}
\max_{0 \le i, j < \ell\, \mf u_{\mf q}}\,
\Big|\, \mtt Q^{\mf q, \ell}_{\ms M_{\mf q, \ell} (i)}
\big[ \, \mtt x(t_0) = \ms M_{\mf q, \ell} (j)\, \big]
\,-\, \pi_{\mf q, \ell} (\ms M_{\mf q, \ell} (j)) \,
\Big|\,\le\,\delta \;,
\end{equation*}
where $\mtt Q^{\mf q, \ell}_{\ms M_{\mf q, \ell} (i)}$
represents the probability measure on the path space induced by the
Markov chain $\bb X_{\mf q, \ell} (\cdot)$ starting
from $\ms M_{\mf q, \ell} (i)$.

As $\mtt a(\cdot)$ and $\mtt b(\cdot)$ are periodic, we may consider
that the diffusion $X_\epsilon(\cdot)$ evolves on the torus
${\color{blue} \bb T_\ell} := [0,\ell)$.  By the stochastic
representation \eqref{90b} of the solution, by the Markov property,
and the $\ell$-translation invariance of $u_0(\cdot)$, as
$\theta^{(\mf q)}_\epsilon \prec \varrho_\epsilon$,
\begin{equation*}
u(\varrho_\epsilon,x) \;=\;
\bb E^{\bb T_\ell, \epsilon}_{[x]_\ell} \Big[ \,
\bb E^{\bb T_\ell,\epsilon}_{\mtt x(\varrho_\epsilon - t_0 \theta^{(\mf q)}_\epsilon)}
\big[ \,
u_0( \mtt x(t_0\theta^{(\mf q)}_\epsilon))\,\big]\, \Big] \;,
\end{equation*}
where $[x]_\ell$ is the representant of $x$ in $\bb T_\ell$, and
$\color{blue} \bb P^{\bb T_\ell, \epsilon}_{y}$, $y\in \bb T_\ell$, is
the probability measure on the path space induced by the diffusion
evolving on $\bb T_\ell$ and starting from $y$. Therefore,
\begin{equation*}
\begin{aligned}
& \Big|\, u_{\epsilon}(\varrho_{\epsilon} , x)
\,-\, \sum_{k=0}^{\ell \mf u_{\mf q}- 1}
\pi_{\mf q, \ell} (\ms M_{\mf q}(k)))\,
\sum_{m \in \ms M_{\mf q} (k)}
\frac{\pi(m)}{\pi(\ms M_{\mf q}(k))} \, u_{0}(m) \, \Big|
\\
& \quad \le\;
\sup_{y\in \bb T_\ell}
\Big|\, \bb E^{\bb T_\ell,\epsilon}_{y}
\big[ \,
u_0( \mtt x(t_0 \theta^{(\mf q)}_\epsilon))\,\big]
\,-\, \sum_{k=0}^{\ell \mf u_{\mf q}- 1}
\pi_{\mf q, \ell} (\ms M_{\mf q}(k)))\,
\sum_{m \in \ms M_{\mf q} (k)}
\frac{\pi(m)}{\pi(\ms M_{\mf q}(k))} \, u_{0}(m) \, \Big|\,.
\end{aligned}
\end{equation*}

We first move the starting point from $y\in \bb T_\ell$ to
$m\in {\color{blue} \ms M_{\mf q, \ell}} := \cup_{j\in \ms S_{\mf q,
\ell}} \ms M_{\mf q, \ell} (j)$. Fix a time-scale
$\theta^{(\mf q-1)}_\epsilon \prec \vartheta_\epsilon\prec
\theta^{(\mf q)}_\epsilon$.  Introduce the restriction
$\{\tau(\ms M_{\mf q, \ell}) < \vartheta_\epsilon\}$, to obtain from
Corollary \ref{l40} that
\begin{align}
\label{91c}
& \bb E^{\bb T_\ell,\epsilon}_{y} \big[ \,
u_0( \mtt x(t_0 \theta^{(\mf q)}_\epsilon))\,\big]
\\
&\quad \,=\,
\sum_{m\in \ms M_{\mf q, \ell}} \bb E^{\bb T_\ell, \epsilon}_{y}
\Big[ \, \mtt 1\big\{ \tau(\ms M_{\mf q, \ell}) <
\vartheta_\epsilon \,,\, \mtt x(\tau(\ms M_{\mf q, \ell})) = m \}\, 
\bb E^{\bb T_\ell,\epsilon}_{m} \big[ u_0( \mtt x(t_0 \theta^{(\mf q)}_\epsilon
- \tau(\ms M_{\mf q, \ell}) )) \,\big] \,\Big]
\,+\, R^{(1)}_\epsilon (y)\;,
\nonumber
\end{align}
where $R^{(1)}_\epsilon (y)$ satisfies \eqref{88}.

By \eqref{103}, for all $0\le s \le \vartheta_\epsilon$,
$m\in \ms M_{\mf q, \ell}$,
\begin{equation}
\label{87b}
\bb E^{\bb T_\ell, \epsilon}_{m} \big[ u_0( \mtt x(t_0\theta^{(\mf
q)}_\epsilon - s))
\,\big] 
\,=\, \sum_{m'\in \ms M_{\mf q, \ell}} u_0(m')\, 
\bb P^{\bb T_\ell, \epsilon}_m\big[ \, 
\mtt x(t_0\theta^{(\mf q)}_\epsilon - s) \in \ms E(m') \, \big]
\,+\, R^{(2)}_\epsilon (s)\,+\, R_\delta \;,
\end{equation}
where $|R_\delta|\le \delta$, and
\begin{equation*}
|R^{(2)}_\epsilon (s)| \,\le\, \Vert u_0\Vert_\infty\;
\bb P^{\bb T_\ell, \epsilon}_m\Big[ \, 
\mtt x(t_0 \theta^{(q)}_\epsilon - s) \not\in
\ms E(\ms M_{\mf q, \ell}) \,\Big] \;.
\end{equation*}
By the proof of \eqref{84} with $p+1=\mf q$, $\varrho_\epsilon =
\theta^{(q)}_\epsilon$, $\vartheta^{(p)}_\epsilon = \vartheta_\epsilon$,
since $\vartheta_\epsilon \succ \theta^{(\mf q-1)}_\epsilon$
\begin{equation}
\label{84b}
\lim_{\epsilon\to 0} \sup_{|s| \le \vartheta_\epsilon}
|R^{(2)}_\epsilon (s)| \,=\, 0\;.
\end{equation}
Mind that we do not need the last part of the argument, where
Corollary \ref{l07} is used, because $\ms M_{\mf q, \ell}$
contains all the wells of the last layer.

It remains to estimate the probability appearing in \eqref{87b}.
Suppose that $m\in \ms M_{\mf q}(j)$, $m'\in \ms M_{\mf q}(k)$ for
some $0\le j, k < \ell \, \mf u_{\mf q}$.  We claim that
\begin{equation*}
\lim_{\epsilon\to 0} 
\sup_{|s|\le \vartheta_\epsilon}
\Big|\, \bb P_m^{\bb T_\ell,\epsilon } \big[\, \mtt x
(t_0 \theta^{(q)}_\epsilon  - s) \in \ms E(m') \,\big]
\,-\, \pi_{\mf q, \ell} (\ms M_{\mf q}(k)))\, \,
\frac{\pi(m')}{\pi(\ms M_{\mf q}(k))} \,\Big| \,=\, 0\;.
\end{equation*}
Bound this expression by
\begin{equation*}
\begin{aligned}
& 
\Big|\, \bb P_m^{\bb T_\ell,\epsilon } \big[\, \mtt x
(t_0 \theta^{(q)}_\epsilon  - s) \in \ms E(m')\,\big]
\,-\, \mtt Q^{\ell, \mf q}_{\ms M_{\mf q} (j)}
\big[\, \mtt x(t_0) = \ms M_{\mf q} (k)\,]
\, \frac{\pi(m')}{\pi(\ms M_{\mf q}(k))} \,\Big|
\\
&\quad +\,
\Big|\, \mtt Q^{\ell, \mf q}_{\ms M_{\mf q} (j)}
\big[\, \mtt x(t_0) = \ms M_{\mf q} (k)\,]
\, \frac{\pi(m')}{\pi(\ms M_{\mf q}(k))} \,-\,
\pi_{\mf q, \ell} (\ms M_{\mf q}(k)))\, 
\frac{\pi(m')}{\pi(\ms M_{\mf q}(k))} \,\Big|
\end{aligned}
\end{equation*}
By the choice of $t_0$, the second term is bounded by $\delta$.
By Lemma \ref{l36}, the first one vanishes as $\epsilon\to 0$. This
completes the proof of the theorem.

\appendix

\section{The hierarchical structure}
\label{sec4}

In this section, we prove the assertions of Section \ref{sec7}
regarding the construction of the hierarchy structure of the model,
and present further properties of this structure needed in the proof
of Theorems \ref{mt3}, \ref{mt4}.

\begin{proof}[Proof of Property $\mc P_{10}(1)$]
By postulate $\mc P_9(1)$, the jump rates $R_1$ are $N$-periodic.
There exists, therefore, $k\in \bb Z$ such that $\mf h_1 = h^1_k$.
Since $\pi_1(k)\, \sigma_1(k,k\pm 1)$, $k\in \bb Z$, is strictly
positive and finite,
$R_1(\ms M_{1}(k),\ms M_{1}(k+1)) + R_1(\ms M_{1}(k),\ms M_{1}(k-1))
>0$. The rest of the proof is analogous to the one of
$\mc P_{10}(p+1)$ presented in the proof of Proposition \ref{l10b}
below.
\end{proof}

\begin{proof}[Proof of Proposition \ref{l12b}]
Property $\mc P_1(p+1)$, $\mc P_2(p+1)$ have been proved right before
the statement of the proposition.  We turn to postulate
$\mc P_3(p+1)$.  By \eqref{75b}, the definition of the sets
$\ms M_{p+1} (j)$, $j\in \bb Z$, and condition $\mc P_2(p+1)$,
$1 + \ms M_{p+1} (0) = \ms M_{p+1} (k)$ for some $k\ge 0$.  The index
$k$ cannot be $0$ because we have already shown that the sets
$\ms M_{p+1} (j)$ are bounded. [If
$1 + \ms M_{p+1} (0) = \ms M_{p+1} (0)$, then
$1+m_{\mss j_{p+1}} \in 1 + \ms M_{p+1} (0) = \ms M_{p+1} (0)$, and by
induction $j+m_{\mss j_{p+1}} \in \ms M_{p+1} (0)$ for all $j\ge 1$].

By property $\mc P_2(p+1)$, $\ms M_{p+1} (0) < \ms M_{p+1} (1) <
\cdots < \ms M_{p+1} (k-1) < \ms M_{p+1} (k) = 1 + \ms M_{p+1}
(0)$. Thus $k$ represents the number of equivalent classes in $\ms
S_{p+1}$, which has been denoted by $\mf u_{p+1}$. This proves
property $\mc P_3(p+1)$.

Consider postulate $\mc P_4(p+1)$.  Fix $k\in \bb Z$.  By
construction,
$\ms M_{p+1}(k) = \cup_{j : \ms M_{p}(j) \in \mf R} \ms M_{p}(j)$ for
some $\bb X_p$-recurrent class $\mf R$.  By postulate $\mc P_4(p)$,
all elements of $\ms M_{p}(j)$ have the same depth. It remains to show
that $S(\ms M_{p}(j))$ is constant.

By postulate $\mc P_7(p)$, $\bb X_p(\cdot)$ jumps only to nearest
neighbour sets, and by condition $\mc P_2(p)$, the sets $\ms M_p(j)$
are ordered. Thus, as $\bb X^p(\cdot)$ has more than one recurrent
class, $\mf R = \{ \ms M_{p}(j) : a\le j\le b\}$ for some
$-\infty<a\le b<+\infty$. Denote $a$, $b$ by $l_{p+1}(k)$,
$r_{p+1}(k)$, respectively, so that
\begin{equation}
\label{43}
\ms M_{p+1}(k) \,:=\, \bigcup_{i=l_{p+1}(k)}^{r_{p+1}(k)} \ms M_p(j) \;.
\end{equation}
Moreover, $R_p(\ms M_{p}(i), \ms M_{p}(i+1))>0$,
$R_p(\ms M_{p}(i+1), \ms M_{p}(i))>0$ for all
$l_{p+1}(k)\le i < r_{p+1}(k)$, Thus, by postulate $\mc P_8(p)$,
$S(\ms M_{p}(j)) = S(\ms M_{p}(l_{p+1}(k)))$ for all
$l_{p+1}(k)\le j\le r_{p+1}(k)$, what proves postulate $\mc P_4(p+1)$.
\end{proof}

\begin{proof}[Proof of Proposition \ref{l11b}]
We first show that $\mf h_{p+1} > \mf h_p$. By the definition
\eqref{47b}--\eqref{38b} of $\mf h_{p+1}$, to prove this inequality,
we have to show that
\begin{equation}
\label{39}
\Lambda\, \big (\ms M_{p+1}(k+ 1) \,,\,
\ms M_{p+1}(k)\big) \,-\,
S(\ms M_{p+1}( k+i)) \,>\, \mf h_p
\end{equation}
for all $k\in \bb Z$, $i = 0$, $1$.

We prove \eqref{39} for $i=1$. The same argument applies to $i=0$. Fix
$k\in \bb Z$. Recall the notation introduced in \eqref{43}. By
definition, the set $\ms M_{p} (r_{p+1}(k)+1)$ does not belong to the
$\bb X^p(\cdot)$-recurrent class of $\ms M_{p} (r_{p+1}(k))$, so that
$R^p(\ms M_{p} (r_{p+1}(k)), \ms M_{p} (r_{p+1} (k)+1)) =0$. Hence, by
conditions $\mc P_6(p)$, $\mc P_7(p)$,
\begin{equation*}
\Lambda\, \big (\ms M_{p}(r_{p+1}(k)) \,,\,
\ms M_{p}(r_{p+1}(k)+1)\big) \,-\,
S(\ms M_{p}( r_{p+1}(k))) \,>\, \mf h_p \;.
\end{equation*}
Since $\ms M_{p} (r_{p+1}(k)) \,\subset\, \ms M_{p+1} (k)$, by definition
of $\ms M_{p+1} (k)$ and postulate $\mc P_4(p+1)$,
$S(\ms M_{p}( r_{p+1}(k))) = S(\ms M_{p+1} (k))$. On the other
hand, by definition of $l_{p+1}(k+1)$, $r_{p+1}(k)+1 \le l_{p+1}(k+1)$ so that
\begin{equation*}
\Lambda (\ms M_{p}(r_{p+1}(k)) \,,\,
\ms M_{p}(r_{p+1}(k)+1) )  \,\le\,  \Lambda\, (\ms M_{p}(r_{p+1}(k)) \,,\,
\ms M_{p}(l_{p+1}(k+1)))\;.
\end{equation*}
By definition of $r_{p+1}(k)$, $l_{p+1}(k+1)$, the right-hand side is
equal to $\Lambda\, (\ms M_{p+1}(k) \,,\, \ms M_{p+1}(k+1))$, which
completes the proof of \eqref{39}.

\smallskip We turn to the proof of postulate $\mc P_5(p+1)$.  Fix
$k\in \bb Z$, and assume that $\ms M_{p+1}(k)$ has two or more
elements. Let $m'$, $m'' \in \ms M_{p+1}(k)$. Suppose, first, that
$m'$, $m''$ belong to the same set
$\ms M_{p}(i) \subset \ms M_{p+1}(k)$. By postulates $\mc P_5(p)$,
$\mc P_{4}(p+1)$ and since $\mf h_p < \mf h_{p+1}$,
\begin{equation*}
\Lambda (m' ,m'') \le  S(\ms M_{p}(i) ) + \mf
h_{p-1}  \,=\,  S(\ms M_{p+1}(k) ) + \mf h_{p-1} <
S(\ms M_{p+1}(k) ) + \mf h_{p}\;.
\end{equation*} 

Recall the notation introduced in \eqref{43}, and suppose now that
$m' \in \ms M_{p}(i) $, $m'' \in \ms M_{p}(j) $ for some $i$,
$j \in \{l_{p+1}(k), \dots, r_{p+1} (k)\}$, $i<j$. Recall from
\eqref{49b} the definition of the rightmost, leftmost, $m^+_{p,i}$,
$m^-_{p,i}$, element of $\ms M_{p}(i)$, respectively.  Then,
\begin{equation*}
\Lambda (m',m'') \,\le \,
\max \Big\{ \Lambda (m', m^+_{p,i}) \,,\,
\Lambda (m^+_{p,i}, m^-_{p,i+1}) \,,\,
\dots \,,\,
\Lambda (m^+_{p,j-1}, m^-_{p,j}) \,,\,
\Lambda (m^-_{p,j}, m'') \, \Big\}\;.
\end{equation*}
As above, by postulates $\mc P_5(p)$, $\mc P_{4}(p+1)$, and since $\mf
h_p < \mf h_{p+1}$, 
\begin{equation*}
\Lambda (m^-_{p,o}, m^+_{p,o}) \,\le\,  S(\ms M_{p}(o) ) \,+\, \mf
h_{p-1} \,=\,  S(\ms M_{p+1}(k) ) \,+\,  \mf h_{p-1} 
\end{equation*}
for all $i<o<j$. The same argument permits to bound the first and last
term on the right-hand side of the penultimate displayed formula by
$S(\ms M_{p+1}(k) ) + \mf h_{p}$. On the other hand, as $\ms M_p(o)$,
$\ms M_p(o+1)$ belong to the same $\bb X^p(\cdot)$-recurrent class, by
definition of $m^+_{p,o}$ $m^-_{p,o+1}$, properties $\mc P_7(p)$,
$\mc P_{4}(p+1)$ 
\begin{equation*}
\Lambda (m^+_{p,o}, m^-_{p,o+1}) \,=\,
\Lambda ( \ms M _p(o), \ms M _p(o+1)) \,=\,
S(\ms M_{p}(o) ) \,+\, \mf h_p
\,=\,
S(\ms M_{p+1}(k) ) \,+\, \mf h_p \;. 
\end{equation*}
This completes the proof of postulate $\mc P_5(p+1)$.

Postulate $\mc P_6(p+1)$ follows from the definition \eqref{38b} of
$\mf h_{p+1}$.  This completes the proof of the proposition;
\end{proof}

\begin{proof}[Proof of Proposition \ref{l10b}]
Condition $\mc P_7(p+1)$ follows from the definition \eqref{50} of
the jump rates.  We turn to $\mc P_8(p+1)$.  Fix $j\in \bb Z$ and
suppose that $R_q(\ms M_{q}(j),\ms M_{q}(j + 1))>0$. The same argument
applies to $j-1$. By definition of the rate function,
$h^{p+1, +}_j = \mf h_{p+1}$ so that, by \eqref{47b},
\begin{equation*}
\Lambda\, \big (\ms M_{p+1}(j+ 1) \,,\,
\ms M_{p+1}(j)\big)-S(\ms M_{p+1}(j)) \,=\,
\mf h_{p+1}\;.
\end{equation*}
By $\mc P_6(p+1)$ (for $k=j+1$), this quantity is bounded above by
\begin{equation*}
h^{p+1, -}_{j+1} \,=\,
\Lambda\, \big (\ms M_{p+1}(j+ 1) \,,\,
\ms M_{p+1}(j)\big)-S(\ms M_{p+1}(j+1)) \;,
\end{equation*}
so that $S(\ms M_{p+1}(j+1)) \le S(\ms M_{p+1}(j))$, as required. If
$R_q(\ms M_{q}(j+1),\ms M_{q}(j))=0$, then, by definition of the jump
rates, $h^{p+1, -}_{j+1} > \mf h_{p+1}$. Thus, the penultimate
displayed equation is strictly less than the last one, and
$S(\ms M_{p+1}(j+1)) < S(\ms M_{p+1}(j))$. This completes
the proof of postulate $\mc P_8(p+1)$.

Consider postulate $\mc P_9(p+1)$. By the $1$-periodicity of
the functions $\mss a(\cdot)$, $\mss b(\cdot)$, and postulate
$\mc P_3(p+1)$, $\pi_{p+1}(k+\mf u_{p+1}) = \pi_{p+1}(k)$,
$\sigma_{p+1}(k+\mf u_{p+1}, k+ 1+ \mf u_{p+1}) = \sigma_{p+1}(k, k+1)$
for all $k\in \bb Z$. Hence, by \eqref{77}, postulate $\mc P_9(p+1)$
is in force.  

To prove claim \eqref{79}, note that the weights $\pi_{p+1}(j)$,
$\sigma_{p+1}(j,j+1)$, $j\in\bb Z$, are strictly positive. Hence,
\eqref{79} follows from the fact, mentioned just before the statement
of Proposition \ref{l11b}, that $\mf h^{p+1} = h^{p+1}_{k}$ for some
$k\in \bb Z$.

We complete the proof of the proposition showing that
$\mf u_{p+2} < \mf u_{p+1}$, which is postulate $\mc P_{10}(p+1)$.
Fix $k\in \bb Z$ for which \eqref{79} holds.  By postulate
$\mc P_9(p+1)$, there exists $0\le k'<\mf u_{p+1}$ for which
\eqref{79} holds. Since one of the jump rates is strictly positive,
$\ms M_{p+1}(k')$ is either a transient state or belongs to a
recurrent class with at least two states. In both cases
$\mf u_{p+2} < \mf u_{p+1}$.

More precisely. Suppose, first, that $\ms M_{p+1}(k')$ is a transient
state. Then, by the definition \eqref{24}, the number of sets
$\ms M^*_{p+2}(j)$ containing one of the sets $\ms M_{p+1}(i)$,
$0\le i < \mf u_{p+1}$, is at most $\mf u_{p+1} -1$, because
$\ms M_{p+1}(k')$ is a transient state. Thus,
$\mf u_{p+2} < \mf u_{p+1}$.

Suppose, next, that $\ms M_{p+1}(k')$ belongs to a recurrent class,
and assume without loss of generality that
$R_{p+1}(\ms M_{p+1}(k'),\ms M_{p+1}(k'-1)) >0$.  Since
$\ms M_{p+1}(k')$ belongs to a recurrent class,
$R_{p+1}(\ms M_{p+1}(k'+1),\ms M_{p+1}(k')) >0$ also.  Thus,
$\ms M_{p+1}(k')$ and $\ms M_{p+1}(k'-1)$ are both contained in the
same set $\ms M^*_{p+2}(j)$ so that $\mf u_{p+2} < \mf u_{p+1}$.
\end{proof}

\subsection*{Further properties of the metastable structure}
\label{sec3b}

We present in this subsection additional properties of the higher order
metastable structure. We start with simple observations. Let
$\ms T_1 = \varnothing$, and define recursively $\ms T_{p+1}$ as
\begin{equation}
\label{80}
{\color{blue}\ms T_{p+1} } \,:=\, \ms T_p \,\cup\,
\bigcup_{i\colon \ms M_p(i)  \in \mf
T_p} \ms M_p(i) \,, \; 1\le p\le \mf q\,, \;\;
\text{so that}\;\;
{\color{blue} \ms P_{q}} := \{\ms M_{q} , \ms T_{q} \} \,,\;
1\le q\le \mf q+1\,,
\end{equation}
is a partition of $\mc M$.

\begin{lemma}
\label{l18}
For $1\le p\le \mf q$, $\ms M_{p} (k) \subset (0,1)$, $1\le k<\mf
u_{p}$. In particular, by  $\mc P_2(p)$,
\begin{equation}
\label{eq:1}
\ms M_p(-1) <  0< \ms M_p(1) < \cdots < \ms M_p(\mf u_p-1) < 1\;,
\quad \text{and}\;\;
\ms M_p(0)\cap (0,1) \,\neq\, \varnothing \;.
\end{equation}
\end{lemma}

\begin{proof}
Since, by $\mc P_2(p)$, $\ms M_{p}(0)$ contains
$m_{\mss j_{p}} \in (0,1)$, and the sets $\ms M_{p}(j)$ are ordered,
it is enough to show that $\ms M_{p}(\mf u_{p} -1) \subset (0,1)$.  To
prove this relation, observe that, by $\mc P_2(p)$, $\mc P_3(p)$,
$\ms M_{p}(\mf u_{p} -1) < \ms M_{p}(\mf u_{p}) = 1 + \ms
M_{p}(0)$. Thus, $\ms M_{p}(\mf u_p-1) < 1 + m_{\mss j_{p}}$.  By
definition of $\mss j_{p}$, $m_i\not\in\ms M_{p}$ for
$0\le i<\mss j_{p}$. Hence, by $\mc P_3(p)$,
$1+m_{i} \not\in\ms M_{p}$ for $0\le i<\mss j_{p}$. We may thus
improve the upper bound of $\ms M_{p}(\mf u_p-1)$ and get that
$\ms M_{p}(\mf u_p-1) < 1 + m_{0}$ because $\ms M_{p}(\mf u_p-1)$ only
contains $S$-local minima. Since there is no $S$-local minima in the
interval $[1,1+m_0)$, $\ms M_{p}(\mf u_p-1) < 1$.
\end{proof}

Fix $2\le p\le \mf q$ and $\ms M_p(k)$, $k\in \bb Z$. The set
$\ms M_p(k)$ is composed of local minima of $S(\cdot)$.
There might be other local minima of $S$ between two elements of
$\ms M_p(k)$. The first assertion of this section states
that these local minima are ``higher'' than the ones of
$\ms M_p(k)$, which, according to postulate $\mc P_4(p)$ are
all at the same height. A precise statement requires some notation.

We refer to Figure \ref{fig-f3} for an illustration of the next
definitions. Recall from \eqref{49b} the definitions of the leftmost
$m_{p,k}^-$ and rightmost $m_{p,k}^+$ elements of $\ms M_p(k)$,
respectively.  Keep in mind that $ m_{p,k}^-$, $ m_{p,k}^+$ might
coincide if $\ms M_{p}(k)$ is a singleton.

Recall from \eqref{max_br} the definition of the set
$\ms W^{(p)}_{j,j+1}$, $j\in \bb Z$.  Let $\sigma^{p,+}_{j, j+1}$,
$\sigma^{p,-}_{j, j+1}$ be the rightmost and leftmost maxima in
$\ms W^{(p)}_{j,j+1}$, respectively:
\begin{equation}
\label{64}
{\color{blue} \sigma^{p,+}_{j, j+1}} \,:=\,
\max\{\sigma\in \ms W^{(p)}_{j,j+1} \}\;,
\quad
{\color{blue}  \sigma^{p,-}_{j,j+1}}\, := \,
\min\{\sigma\in \ms W^{(p)}_{j,j+1}\}\;.
\end{equation}
Recall from \eqref{76} that $\mc M\subset \bb R$ represents the set of
local minima of $S(\cdot)$. Denote by $\mss M_p(k)$ the set of local
minima of $S(\cdot)$ which belong to the interval
$[\sigma^{p,+}_{k-1,k} , \sigma^{p,-}_{k,k+1}]$ and are not in
$\ms M_{p}(k)$:
\begin{equation*}
{\color{blue} \mss M_p(k) } \,:=\,
\big(\, \mc M \cap [\sigma^{p,+}_{k-1,k} ,
\sigma^{p,-}_{k,k+1}] \,\big)  \setminus \ms M_{p}(k) \;.
\end{equation*}
As illustrated in Figure \ref{fig-f3}, some elements of $\mss M_p(k)$
may belong to the interval $[\sigma^{p,+}_{k-1, k}, m^-_{p,k})$ or to
$[m^+_{p,k}, \sigma^{p,-}_{k,k+1})$.

\begin{figure}
\centering
\begin{tikzpicture}[scale=0.5]
\draw [rounded corners] (0.5,-3) -- (1,-1);
\draw[rounded corners] (1,-1) .. controls (1.5,1) .. (2,3);
\draw[rounded corners] (2,3) .. controls (2.5,5) .. (3,3);  
\draw[rounded corners] (3,3) -- (3.2, 2.2);
\draw[rounded corners] (3.2,2.2) -- (3.5, 1);
\draw[rounded corners] (3.5, 1) .. controls (3.8, 0) .. (4.1, 0.5);
\draw[rounded corners] (4.1, 0.5) .. controls (4.4, 1) .. (4.7, 0.5);
\draw[rounded corners] (4.7, 0.5) .. controls (5, 0) .. (5.3, 1);
\draw[rounded corners] (5.3, 1) .. controls (5.6, 2) .. (5.9, 1);
\draw[rounded corners] (5.9, 1) -- (6.4, -1);
\draw[rounded corners] (6.4,-1) .. controls (6.9, -3) .. (7.4, -1);
\draw[rounded corners] (7.4,-1) .. controls (7.9, 1) .. (8.4, -1);
\draw[rounded corners](8.4,-1) .. controls (8.7,-2)..(9,-1);
\draw[rounded corners](9,-1).. controls (9.3,0)..(9.6,-1);
\draw[rounded corners] (9.6,-1) .. controls (10.1,-3) .. (10.6, -1);
\draw[rounded corners](10.6,-1)--(11.1,1);
\draw[rounded corners](11.1,1)..controls (11.4,2)..(11.7,1);
\draw[rounded corners](11.7,1)--(12.2,-1);
\draw[rounded corners](12.2,-1).. controls (12.7,-3)..(13.2,-1);
\draw[rounded corners](13.2,-1)..controls (13.7,1)..(14.2,3);
\draw[rounded corners](14.2,3)..controls (14.5,4)..(14.8,3);
\draw[rounded corners](14.8,3)..controls (15.3,1)..(15.8,3);
\draw[rounded corners](15.8,3)..controls (16.3,5)..(16.8,3);
\draw[rounded corners](16.8,3)..controls (17.1,2)..(17.4,2.5);
\draw[rounded corners](17.4,2.5)..controls(17.7,3)..(18,2.5);
\draw[rounded corners](18,2.5)..controls(18.3,2)..(18.6,3);
\draw[rounded corners](18.6,3)..controls(19.1,5)..(19.6,3.5);
\draw[rounded corners](19.6,3.5)..controls (20.1,2)..(20.5,3.5);
\draw[rounded corners](20.5,3.5)..controls (21,5)..(21.5,3.5);
\draw[rounded corners](21.5,3.5)--(23,-2);
\draw[rounded corners](23,-2)--(23.5,-3.5);

\draw[dashed, black](2.5,4.6)-- (16.3,4.6);

\fill(16.3,-5.7)node[below, font=\small]{$\sigma^{p,-}_{k,k+1}$};
\fill(21,-5.7)node[below, font=\small]{$\sigma^{p,+}_{k,k+1}$};
\draw[thick] (16.3,-5.7) -- (16.3,-5.3);
\draw[thick] (21,-5.7) -- (21,-5.3);

\draw[thick,fill,cyan](6.9,-5.5)circle(.2);
\draw[thick,fill,cyan](12.7,-5.5) circle(.2);
\draw[thick,fill,cyan](10.1,-5.5)circle(.2);

\fill(6.9,-5.7)node[below, font=\small]{$m^{p,-}_k$};
\fill(12.7,-5.7)node[below, font=\small]{$m^{p,+}_k$};

\draw[thick, fill, teal](3.8,-5.5) circle(.2);
\draw[thick, fill, teal](5,-5.5) circle(.2);
\draw[thick, fill, teal](8.7,-5.5)circle(.2);
\draw[thick, fill, teal](15.3,-5.5) circle(.2);

\draw[solid, thick, black](0,-5.5)--(24,-5.5);

\draw (2.5,-5.7) node[below, font=\small]{$\sigma^{p,+}_{k-1,k}$};
\draw[thick] (2.5,-5.7) -- (2.5,-5.3);

\draw[thick, fill, orange](17.1,-5.5)circle(.2);
\draw[thick, fill, orange](18.3,-5.5)circle(.2);
\draw[thick, fill, orange](20.1,-5.5)circle(.2);

\end{tikzpicture}
\caption{Here we illustrate some of the important definitions
surrounding the structure of a higher order metastable state. Fixed
$p>1$ and $k\in \bb Z$, the state $\ms M_p(k)$ is represented in the
figure by the cyan circles, while the set $M_p(k)$ is composed by the
green circles. The orange circles represent the local minima in
between the left-most and right-most maxima separating the states
$\ms M_p(k)$ and $\ms M_p(k+1)$. As stated in Proposition \ref{l13},
the minima 
contained in $M_p(k)$ are "higher" than the ones contained in
$\ms M_p(k)$, which are all of the same height. Regarding the escape
barriers of the well $\ms E(\ms M_p(k))$, while we have only one
global maxima separating $\ms E(\ms M_p(k))$ and
$\ms E(\ms M_p(k-1))$, there are three between $\ms E(\ms M_p(k))$ and
$\ms E(\ms M_p(k+1))$. Here,
$\sigma_{k-1,k}^{p,+}=\sigma_{k-1,k}^{p,-}$, while
$\sigma_{k,k+1}^{p,-}\,<\,\sigma_{k,k+1}^{p,+}$.}

\label{fig-f3}
\end{figure}
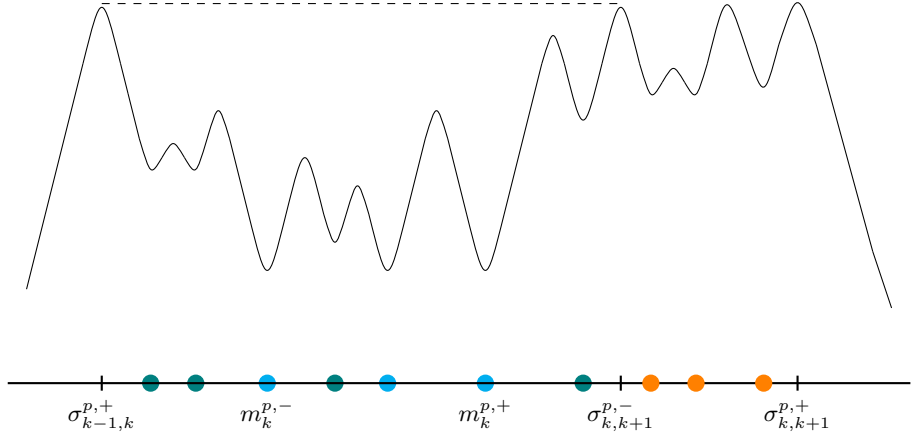

\begin{proposition}
\label{l13}
For all $k\in \bb Z$,
\begin{equation*}
\min \{ S(m) : m\in \mss M_p(k)\,\} \,>\, S(\ms M_{p}(k))\;.
\end{equation*}
\end{proposition}

We assume below that $\mss M_p(k)$ is not empty, otherwise there is
nothing to prove.  The proof of this result requires some elementary
properties of the construction. 

\begin{lemma}
\label{l16}
For each $2\le p \le \mf q$, $k\in \bb Z$, $m\in \mss M_p(k)$, there
exist $1\le q<p$, $j\in \bb Z$, such that $m \in \ms M_{q} (j)$ and
$\ms M_{q} (j)$ is a $\bb X_q$-transient set. With the notation
introduced in \eqref{80}, $\ms M_{q} (j) \in \mf T_q$.
\end{lemma}

\begin{proof}
Fix $m\in \mss M_p(k)$. By postulate $\ms P_2(p)$, $m$ does not belong
to any of the sets $\ms M _p(j)$, $j\in \bb Z$. [By definition it does
not belong to $\ms M_p(k)$. By postulate $\ms P_2(p)$ the sets
$\ms M_{p}(j)$ are ordered. By definition it is to the right of
$\ms M_p(k-1)$ and to the left of $\ms M_p(k+1)$.] Hence, by
\eqref{80}, $m\in \ms T_p$, and there exist $1\le q < p$,
$j \in \bb Z$, such that $m \in \ms M_{q} (j) \in \mf T_q$, as
claimed.
\end{proof}

For $m\in \mss M_p(k)$, denote by $\color{blue} p_m$ the index
$1\le q<p$ given by Lemma \ref{l16}.

\begin{remark}
\label{l19} 
By construction, for each $1\le q<p$, $k\in \bb Z$, there exists at
least one $\ell\in \bb Z$ such that $\ms M_{q}(\ell)$ is
$\bb X_q$-recurrent and $\ms M_{q}(\ell) \subset \ms M_{p}(k)$. 

Suppose that there exist $1\le q<p$, $j\in \bb Z$, such that
$\ms M_q(j) \cap (\sigma^{p,+}_{k-1,k} , \sigma^{p,-}_{k,k+1}) \neq
\varnothing$, and $\ms M_q(j)$ is $\bb X_q$-transient. By
construction, $\ms M_q(j) \cap \ms M_{p}(k) = \varnothing$. In this
case, by the previous paragraph, there exist at least two sets,
$\ms M_q(j)$ and $\ms M_{q}(\ell)$, whose intersection with the
interval $(\sigma^{p,+}_{k-1,k} , \sigma^{p,-}_{k.k+1})$ is not
empty. Corollary \ref{l21} asserts that $\ms M_q(j)$ is actually
contained in the interval
$(\sigma^{p,+}_{k-1,k} , \sigma^{p,-}_{k,k+1})$.
\end{remark}

Recall that $\mc W\subset \bb R$ represents the set of local maxima of
$S(\cdot)$. Next result states that the energetic barriers at
$\sigma^{p,+}_{k-1,k}$, $\sigma^{p,-}_{k,k+1}$ are higher than the ones in
the interior of the interval
$(\sigma^{p,+}_{k-1,k}\,,\, \sigma^{p,-}_{k,k+1})$.  It follows from this
result that the diffusion $X_\epsilon (t)$ is trapped in the interval
$(\sigma^{p,+}_{k-1,k}\,,\, \sigma^{p,-}_{k,k+1})$ in time-scales of smaller
order than $\theta^{(p)}_\epsilon$.

\begin{lemma}
\label{l15}
For each $2\le p \le \mf q$, $k\in \bb Z$,
\begin{gather*}
S(\sigma^{p,+}_{k-1,k}) \,>\, {\color{blue} S^-_{p,k}}
\:=\, \max \{ \, S(x) :
x\in (\sigma^{p,+}_{k-1,k} ,  m^+_{p,k}] \cap \mc W\, \big\}\;, 
\\
S(\sigma^{p,-}_{k,k+1})  
\,>\, {\color{blue} S^+_{p,k}} \,:=\,
\max \{ S(x) : x\in  [m^-_{p,k} , \sigma^{p,-}_{k,k+1} ) \cap \mc W\} \;.
\end{gather*}
\end{lemma}

\begin{proof}
We prove the first bound. The arguments extend to the second.
By definition of $\sigma^{p,+}_{k-1,k}$,
\begin{gather*}
S(\sigma^{p,+}_{k-1,k}) \,>\, \max \{ S(x) : x\in (\sigma^{p,+}_{k-1,k}
,m^-_{p,k} ]\cap \mc W\}\;,
\end{gather*}
This strict inequality can be extended to the local maxima of
$S(\cdot)$ in the interval $[m^-_{p,k}, m^+_{p,k} ]$. Indeed, suppose
that $m^-_{p,k}< m^+_{p,k}$, and fix a local maximum $\sigma$ in this
interval. By postulate $\mc P_5(p)$, and Proposition \ref{l11b},
$S(\sigma) - S(\ms M_p(k))\le \Lambda(m^-_{p,k}, m^+_{p,k}) - S(\ms
M_p(k)) \le \mf h_{p-1} < \mf h_p$. On the other hand, by condition
$\mc P_6(p)$, and since
$S(\sigma^{p,+}_{k-1,k}) = \Lambda(\ms M_p(k-1), \ms M_p(k))$,
\begin{equation}
\label{27}
S(\sigma^{p,+}_{k-1,k}) \,-\,  S(\ms M_p(k))  \,\ge\, \mf h_p  \;.
\end{equation}
Hence, $S(\sigma) < S(\sigma^{p,+}_{k-1,k})$, which completes the
proof of the lemma.
\end{proof}

\begin{corollary}
\label{l21}
Suppose that
$\ms M_{q}(j) \cap (\sigma^{p,+}_{k-1,k} \,,\,
\sigma^{p,-}_{k,k+1}) \neq \varnothing$ for some $j\in \bb Z$, $1\le
q<p$. Then,
$\ms M_{q}(j) \subset (\sigma^{p,+}_{k-1,k} \,,\,
\sigma^{p,-}_{k,k+1})$.
\end{corollary}

\begin{proof}
Fix $m\in \ms M_{q}(j) \cap (\sigma^{p,+}_{k-1,k} \,,\,
\sigma^{p,-}_{k,k+1})$ for some $j\in \bb Z$, $1\le q<p$.

By construction, a set $\ms M_{q}(\ell)$, $\ell\in \bb Z$, is either
contained in $\ms M_{p}(k)$ or in its complement. Thus, if
$\ms M_{q}(j) \cap \ms M_{p}(k) \neq \varnothing$, there is noting to
prove because
$\ms M_{p}(k) \subset (\sigma^{p,+}_{k-1,k} \,,\,
\sigma^{p,-}_{k,k+1})$.

Assume therefore that $\ms M_{q}(j) \cap \ms M_{p}(k) = \varnothing$,
and that $\ms M_{q}(j)$ is not a singleton.  Let $m'\in \ms M_{q}(j)$,
$m'\neq m$, say $m'<m$. By construction, there exists $\ell\in \bb Z$
such that $\ms M_{q}(\ell) \subset \ms M_{p}(k)$. Let
$m_*\in \ms M_{q}(\ell)$.  Since, by postulate $\mc P_2(q)$, the sets
$\ms M_{q}(i)$ are ordered, either $\ms M_{q}(j) < \ms M_{q}(\ell)$ or
$\ms M_{q}(\ell) < \ms M_{q}(j)$.  If
$\ms M_{q}(\ell) < \ms M_{q}(j)$, there is nothing to prove because
$\sigma^{p,+}_{k-1,k} < \ms M_{q}(\ell) < m' <m<
\sigma^{p,-}_{k,k+1}$.  Assume, therefore, that
$\ms M_{q}(j) < \ms M_{q}(\ell)$.

Since $m'<m$, assume by contradiction that $m'<\sigma^{p,+}_{k-1,k}$ so
that $\Lambda (m,m') \ge S(\sigma^{p,+}_{k-1,k})$. On the other hand, by
$\mc P_5(q)$ and Proposition \ref{l11b},
$\Lambda (m,m') < S(\ms M_{q}(j)) + \mf h_q$, so that
$S(\sigma^{p,+}_{k-1,k}) < S(\ms M_{q}(j)) + \mf h_q$.

Since $m\in \ms M_{q}(j)$, $m_*\in \ms M_{q}(\ell)$, and
$\sigma^{p,+}_{k-1,k} <m<m_* < m^+_{p,k}$ [because
$m'\in \ms M_{q}(j)$, $m_*\in \ms M_{q}(\ell) \subset \ms M_{p}(k) $,
$\ms M_{q}(j) < \ms M_{q}(\ell) $], by \eqref{09}, and the definition
of $S^-_{p,k}$,
$\Lambda(\ms M_{q}(j) , \ms M_{q}(\ell)) \le \Lambda (m,m_*) \le
S^-_{p,k}$.

On the other hand, as $\ms M_{q}(j) < \ms M_{q}(\ell)$, by postulates
$\mc P_2(q)$, $\mc P_6(q)$, and the definition of the barrier
$\Lambda (\cdot, \cdot)$,
$\Lambda(\ms M_{q}(j) , \ms M_{q}(\ell)) - S(\ms M_{q}(j)) \ge
\Lambda(\ms M_{q}(j) , \ms M_{q}(j+1)) - S(\ms M_{q}(j)) \ge \mf h_q$.

Putting together the previous estimates yields that
\begin{equation*}
S(\sigma^{p,+}_{k-1,k}) \,<\, S(\ms M_{q}(j)) + \mf
h_q \,\le\, \Lambda(\ms M_{q}(j) , \ms M_{q}(\ell))
\,\le\, S^-_{p,k}\;,
\end{equation*}
which contradicts the assertion of the previous lemma.
\end{proof}

Fix $\ms M_{q}(j)$ for some $1\le q<p$, $j \in \bb Z$, such that
$\ms M_{q}(j) \subset (\sigma^{p,+}_{k-1,k}, \sigma^{p,-}_{k,k+1})$,
$\ms M_{q}(j) \cap \ms M_{p}(k) = \varnothing$.  Let
$\color{blue} \ms M_{q}(j_+)$, $\color{blue} \ms M_{q}(j_-)$ the
rightmost, leftmost set which can be attained by the Markov chain
$\bb X_q(\cdot)$ starting from $\ms M_{q}(j)$, respectively. Thus,
$j_+ = j$ if $R^q(\ms M_{q} (j), \ms M_{q}(j+1))=0$, and
$j_+ := \max \{r \ge j : R^q(\ms M_{q}(i-1), \ms M_{q}(i))>0 \;,\;\;
j<i\le r\}$, otherwise. Similarly, $j_- = j$ if
$R^q(\ms M_{q}(j), \ms M_{q}(j-1))=0$, and
$j_- := \min \{r \ge j : R^q(\ms M_{q}(i+1), \ms M_{q}(i))>0 \;,\;\; r
\le i<j\}$.  Since $q<p\le \mf q$, $\bb X_q(\cdot)$ has at least two
closed irreducible classes, up to equivalences, so that $j_\pm$ are
well defined and finite.  

\begin{lemma}
\label{l22}
Fix $\ms M_{q}(j)$ for some $1\le q<p$,
$j \in \bb Z$, such that
$\ms M_{q}(j) \subset (\sigma^{p,+}_{k-1,k},
\sigma^{p,-}_{k,k+1})$,
$\ms M_{q}(j) \cap \ms M_{p}(k) =
\varnothing$. Then, with the notation introduced above,
\begin{equation*}
\ms M_{q}(j_-) \cup \ms M_{q}(j_+) \,\subset\,
(\sigma^{p,+}_{k-1,k}, \sigma^{p,-}_{k,k+1})\;.
\end{equation*}
\end{lemma}

\begin{proof}
We prove the result for $\ms M_{q}(j_-)$, as the
same argument applies to $\ms M_{q}(j_+)$.

Let $\ell \in \bb Z$ such that $\ms M_{q}(\ell) \subset \ms M_{p}(k)$
given by Remark \ref{l19}.  By postulate $\mc P_2(q)$, the sets
$\ms M_{q}(r)$, $r\in \bb Z$, are ordered. Assume that
$\ms M_{q}(j) < \ms M_{q}(\ell)$. The same argument applies for the
other case.

As $\ms M_{q}(j) < \ms M_{q}(\ell)$ and
$\ms M_{q}(\ell) \subset \ms M_{p}(k) \subset
(\sigma^{p,+}_{k-1,k}, \sigma^{p,-}_{k,k+1})$,
$\ms M_{q}(j_-) < \{\sigma^{p,-}_{k,k+1}\}$. It remains
to show that
$\ms M_{q}(j_-) > \{\sigma^{p,+}_{k-1,k}\}$.

We claim that
\begin{equation}
\label{28}
\mf h_q \,\le\, S^-_{p,k} - S(\ms M_{q}(j))\;.
\end{equation}
Indeed, by postulate $\mc P_6(q)$,
$\mf h_q \le \Lambda(\ms M_{q}(j), \ms M_{q}(j +1 ) ) - S(\ms
M_{q}(j))$. Since $\ell > j$, by the definition of the barrier height
$\Lambda (\cdot, \cdot)$, the previous expression is less than or
equal to $\Lambda(\ms M_{q}(j), \ms M_{q}(\ell)) - S(\ms
M_{q}(j))$. Since, by hypothesis, the sets $\ms M_{q}(j)$,
$\ms M_{q}(\ell)$ are contained in the interval
$(\sigma^{p,+}_{k-1,k}, m^+_{p,k}]$, by the definition of $S^-_{p,k}$,
the previous expression is bounded by $S^-_{p,k} - S(\ms M_{q}(j))$. To
complete the proof of \eqref{28}, it remains to recollect the previous
estimates.

We are now in a position to show that
$\ms M_{q}(j_-) > \{\sigma^{p,+}_{k-1,k}\}$.  
Suppose, by contradiction, that
$\ms M_{q}(j_-) < \{\sigma^{p,+}_{k-1,k}\}$.  In
particular, $j_-<j$.

As $R^q(\ms M_q(r+1), \ms M_q(r))>0$, for $j_- \le r<j$, by postulate
$\mc P_8(q)$, $S(\ms M_{q}(j-n))$ is a decreasing sequence for
$0\le n\le j-j_-$. Hence $S(\ms M_{q}(r)) \le S(\ms M_{q}(j))$ for
$j_- \le r \le j$.

Recall the assertion of Corollary \ref{l21}, and denote by $j_0$ the
index such that
$\ms M_{q}(j_0) < \{\sigma^{p,+}_{k-1,k}\} < \ms M_{q}(j_0+1)$. As
$\ms M_{q}(j_-) < \{\sigma^{p,+}_{k-1,k}\}$, $j_0\ge j_-$. In
particular, $R^q(\ms M_q(j_0+1), \ms M_q(j_0)>0$.  Thus, by postulate
$\mc P_7(q)$,
$\mf h_k = \Lambda(\ms M_{q}(j_0), \ms M_{q}(j_0+1)) - S(\ms
M_{q}(j_0+1))$. By definition of $j_0$, this quantity is larger than
or equal to $S(\sigma^{p,+}_{k-1,k}) - S(\ms M_{q}(j_0+1))$.  Since
$S(\ms M_{q}(j-n))$ is a decreasing sequence, and $j_0+1\le j$ [by
definition of $j_0$ and because
$\ms M_{q}(j) \subset (\sigma^{p,+}_{k-1,k}, \sigma^{p,-}_{k,k+1})]$,
this quantity is bounded below by
$S(\sigma^{p,+}_{k-1,k}) - S(\ms M_{q}(j))$.  Adding together the
previous estimates yields that
$\mf h_k \ge S(\sigma^{p,+}_{k-1,k}) - S(\ms M_{q}(j))$.

By the previous inequality and \eqref{28},
$S(\sigma^{p,+}_{k-1,k}) \le S^-_{p,k}$, in contradiction with Lemma
\ref{l15}. This shows that
$\ms M_{q}(j_-) > \{\sigma^{p,+}_{k-1,k}\}$, and
completes the proof of the lemma.
\end{proof}

In the next result, we keep the same notation of Lemma \ref{l22}, and
we further assume that $\ms M_{q}(j)$ is $\bb X_q$-transient.

\begin{lemma}
\label{l23}
Fix $\ms M_{q}(j)$ for some $1\le q<p$,
$j \in \bb Z$, such that
$\ms M_{q}(j) \subset (\sigma^{p,+}_{k-1,k},
\sigma^{p,-}_{k,k+1})$,
$\ms M_{q}(j) \cap \ms M_{p}(k) =
\varnothing$. Assume that $\ms M_{q}(j)$ is $\bb X_q$-transient.
Then, there exists
${\color{blue} \mf s(j)} \in \{j_-, \dots, j_+\}$ such that
$S(\ms M_{q}(\mf s(j))) < S(\ms M_{q}(j))$,
$\ms M_{q}(\mf s(j)) \subset (\sigma^{p,+}_{k-1,k},
\sigma^{p,-}_{k,k+1})$.
\end{lemma}

\begin{proof}
Recall the notation introduced above the statement of Lemma \ref{l22}.
Let $\ms M_{q}(k_+)$, $k_+\ge j$, be the rightmost set $\ms M_{q}(r)$
which can attain $\ms M_{q}(j)$ with left jumps: $k_+=j$ if
$R^q(\ms M_{q} (j+1), \ms M_{q}(j))=0$, and
$k_+ := \max \{ r \ge j : R^q(\ms M_{q} (i), \ms M_{q}(i-1))>0 \;,\;\;
j<i\le r\}$, otherwise. Define $\ms M_{q}(k_-)$, $k_-\le j$,
similarly, as the leftmost set $\ms M_{q}(r)$ which can attain
$\ms M_{q}(j)$ with right jumps. Since $q<p\le \mf q$, $\bb X_q$ has
at least two closed irreducible class, up to equivalences, so that
$k_\pm$ are finite and well defined.  We do not claim that the sets
$\ms M_{q}(k_-)$, $\ms M_{q}(k_+)$ are contained in the interval
$(\sigma^{p,+}_{k-1,k}, \sigma^{p,-}_{k,k+1})$.

Suppose that $k_-\le j_-$ and $k_+\ge j_+$. In this case,
$\ms M_{q}(j)$ belongs to a $\bb X_q$-closed irreducible class, in
contradiction with the initial assumption that it is
$\bb X_q$-transient. Hence, $k_-> j_-$ or $k_+ < j_+$.

Assume, without loss of generality, that $k_+ < j_+$. Thus, the Markov
chain $\bb X_q$ may jump back and forth between $\ms M_{q}(i)$ and
$\ms M_{q}(i+1)$ for $j\le i < k_+$, it may jump from $\ms M_{q}(k_+)$
to $\ms M_{q}(k_++1)$, and it can not jump from $\ms M_{q}(k_++1)$ to
$\ms M_{q}(k_+)$. By postulate $\mc P_8(q)$,
$S(\ms M_{q}(j)) = S(\ms M_{q}(j +1)) = \cdots = S(\ms M_{q}(k_+)) >
S(\ms M_{q}(k_+ +1))$. By Lemma \ref{l22}, since $k_+ < j_+$,
$\ms M_{q}(k_+ +1) \subset (\sigma^{p,+}_{k-1,k},
\sigma^{p,-}_{k,k+1})$. Thus, $\mf s(j) = k_+ +1$ fullfils all
requisites of the lemma.
\end{proof}

Recall from Lemma \ref{l16} that we denote by $p_m$,
$m\in \mss M_p(k)$, the index of $m$.

\begin{proof}[Proof of Proposition \ref{l13}]
Let ${\color{blue} q_1} := \min \{p_m : m\in \mss M_p(k)\}$,
$q_{i+1} := \min \{p_m > q_i : m\in \mss M_p(k)\}$, $i\ge 1$. Since
$q_i$ is an increasing sequence bounded by $p$, it has only a finite
number of terms, $q_1 < \cdots < q_{i_0} <p$. The proof is by backward
induction on $q_i$: we prove that the result holds for
$m\in \mss M_p(k)$ such that $p_m=q_{i_0}$; and that it holds for
$m\in \mss M_p(k)$ such that $p_m=q_i$ if it holds for
$m'\in \mss M_p(k)$ such that $p_{m'}=q_j$, $i<j\le i_0$.

Fix $\color{blue} q=q_i$ for some $1\le i\le i_0$, and
$m\in \mss M_p(k)$ such that $p_m=q$.  As $p_m=q$, $m\in \ms M_{q}(j)$
for some $\bb X_q$-transient set $\ms M_{q}(j)$, $j\in \bb Z$. By
Remark \ref{l19}, $\ms M_{q}(j) \cap \ms M_{p}(k) = \varnothing$, and
by Corollary \ref{l21},
$\ms M_{q}(j) \subset (\sigma^{p,+}_{k-1,k}, \sigma^{p,-}_{k,k+1})$.
Hence, by Lemma \ref{l23}, there exists $j'$ such that
\begin{equation}
\label{29}
\ms M_{q}(j') \,\subset\,
(\sigma^{p,+}_{k-1,k}, \sigma^{p,-}_{k,k+1})\,,
\quad\text{and}\quad
S(\ms M_{q}(j')) \,<\,
S(\ms M_{q}(j))\;.
\end{equation}

Suppose that $\ms M_{q}(j')$ is a $\bb X_q$-transient set. Then, by
Remark \ref{l19}, $\ms M_{q}(j') \cap \ms M_{p}(k) =
\varnothing$. Therefore, by the first assertion of \eqref{29},
$\ms M_{q}(j')$ satisfies the hypotheses of Lemma \ref{l23}.  In
particular, as the Markov chain $\bb X_q$ has a finite number of
states in the interval $(\sigma^{p,+}_{k-1,k}, \sigma^{p,-}_{k,k+1})$,
we may repeat the argument until we find an index $j'$ such that
$\ms M_{q}(j')$ belongs to a $\bb X_q$-closed irreducible class.

Assume therefore $\ms M_{q}(j')$ is $\bb X_q$-recurrent.  Fix
$m'\in \ms M_{q}(j')$. Since this later state is $\bb
X_q$-recurrent, by construction, $p_{m'}> q = p_m$. If $p_{m'} \ge p$,
by Lemma \ref{l16}, $m'\not\in \mss M_p(k)$, so that, by the first
assertion in \eqref{29}, $m'\in \ms M_{p}(k)$. Hence, by the second
assertion in \eqref{29}, and postulates $\mc P_4(q)$, $\mc P_4(p)$,
$S(\ms M_{q}(j)) > S(\ms M_{q}(j')) = S(m')=S(\ms M_{p}(k))$, as
claimed.

If $p_{m'}<p$, we may iterate the argument until we reach a point
$m''\in \mc M$ such that $p_{m''} \ge p$ and argue as in the previous
paragraph. 
\end{proof}

We conclude this section with a result needed in the proof of the
local ergodicity, and one needed in the proof of the convergence of
the solution of the resolvent equation. Note that $\sigma\le m$ in the
first estimate and $m\le \sigma$ in the second one.

\begin{lemma}
\label{l17}
For all $k\in \bb Z$,
\begin{equation*}
\begin{gathered}
\max_{\sigma \in [m^-_{p,k}, \sigma^{p,-}_{k,k+1})  \cap \mc W}
\;\; \max_{m\in [\sigma, \sigma^{p,-}_{k,k+1}] \cap \mc M}
[S(\sigma)-S(m)]
\,\le\, \mf h_{p-1} \;,
\\
\max_{\sigma \in (\sigma^{p,+}_{k-1,k} , m^{+}_{p,k}]  \cap \mc W}
\;\; \max_{m\in [\sigma^{p,+}_{k-1,k} , \sigma] \cap \mc M}
[S(\sigma)-S(m)] \,\le\, \mf h_{p-1} \;.
\end{gathered}
\end{equation*}
\end{lemma}

\begin{proof}
We prove the first estimate.  Suppose first that
$\sigma \in (m^-_{p,k}, m^+_{p,k}) \cap \mc W$, and fix
$m\in [\sigma, \sigma^{p,-}_{k,k+1}]$. In this case, by definition of
the energy barrier,
$S(\sigma)-S(m) \le \Lambda (m^-_{p,k}, m^+_{p,k}) - S(m)$.  By
postulate $\mc P_5(p)$, this expression is less than or equal to
$S(\ms M_{p}(k)) + \mf h_{p-1} - S(m)$. By Proposition \ref{l13}, this
difference is bounded by $\mf h_{p-1}$, as claimed.

Suppose that $\sigma \in (m^+_{p,k}, \sigma^{p,-}_{k,k+1}) \cap \mc W$.
Denote by $m$ a local minimum of $S(\cdot)$ in
$(\sigma, \sigma^{p,-}_{k,k+1})$ which attains the minimum value of
$S(\cdot)$ in this interval. There may be more than one, but this is
irrelevant. Let $q=p_m$. This value is smaller than $p$ because $m$ is
larger than $m^+_{p,k}$: $m>\sigma>m^+_{p,k}$.

Fix $j\in \bb Z$ so that $m\in \ms M_{q}(j)$ and recall the notation
introduced in Lemma \ref{l22}.  We claim that $\ms M_{q}(j)$,
$\ms M_{q}(j_+)$ belong to the same $\bb X_q$-equivalent class.

Indeed, by Lemma \ref{l22},
\begin{equation}
\label{53}
\ms M_{q}(j_+) \,\subset\, (\sigma^{p,+}_{k-1,k},
\sigma^{p,-}_{k,k+1})\;.
\end{equation}
If $\bb X_q$ is allowed to jump from $\ms M_{q}(\ell)$ to
$\ms M_{q}(\ell+1)$ but not from $\ms M_{q}(\ell+1)$ to
$\ms M_{q}(\ell)$ for some $j\le \ell < j_+$, then, by postulate
$\mc P_8(q)$,
\begin{equation}
\label{58}
S(\ms M_{q}(i)) \,<\, S(\ms M_{q}(j))
\end{equation}
for some $j < i \le j_+$.

By postulate $\mc P_2(q)$, the sets $\ms M_{q}(r)$, $r\in \bb Z$, are
ordered. Thus, as $j < i \le j_+$, by \eqref{53},
$\ms M_{q}(j) < \ms M_{q}(i) < \sigma^{p,-}_{k,k+1}$.  This fact and
\eqref{58} contradicts the definition of $m$.

The same argument shows that all set $\ms M_{q}(i)$, such that
$j_- \le i$ and $\sigma < m'$ for some $m'\in \ms M_{q}(i)$ belong to
the $\bb X_q$-equivalent class of $\ms M_{q}(j)$.  In particular,
$S(\ms M_{q}(\ell)) \,=\, S(\ms M_{q}(j))$ for all
$j_- \le \ell \le j_+$ such that $m''>\sigma$ for some
$m''\in \ms M_{q}(\ell)$.

Recall the definition of $\mf s(j)$ introduced in Lemma \ref{l23}. By
this result and the previous observation, $\ms M_{q}(\mf s(j))$ can
not have elements to the right of $\sigma$ because
$S(\ms M_{q}(\mf s(j))) < S(\ms M_{q}(j))$. That is,
$\ms M_{q}(\mf s(j))<\sigma$.  As $\ms M_{q}(\mf s(j))$ can be reached
from $\ms M_{q}(j)$ with left jumps, by postulates $\mc P_4(q)$,
$\mc P_7(q)$,
$S(\sigma) - S(m) = S(\sigma) - S(\ms M_{q}(j)) \le \mf h_q \le \mf
h_{p-1}$. This completes the proof of the lemma.
\end{proof}

\begin{remark}
\label{l17b}
It might be easier to understand the assertion of Lemma \ref{l17}
rewriting the inequalities as 
\begin{equation*}
\begin{gathered}
\max_{\sigma \in [m^{p,-}_{k}, \sigma^{p,-}_{k,k+1})  \cap \mc W}
\Big\{\, S(\sigma) \,-\, 
\min_{m\in [\sigma, \sigma^{p,-}_{k,k+1}] \cap \mc M} S(m)\,\Big\}
\,\le\, \mf h_{p-1} \;,
\\
\max_{\sigma \in (\sigma^{p,+}_{k-1,k} , m^{p,+}_{k}]  \cap \mc W}
\Big\{\, S(\sigma) \,-\, 
\min_{m\in [\sigma^{p,+}_{k-1,k} , \sigma] \cap \mc M} S(m)\,\Big\}
\,\le\, \mf h_{p-1} 
\end{gathered}
\end{equation*}
for all $k\in \bb Z$,
\end{remark}

Recall the definition of $\sigma^{p,\pm}_{k,k+1}$, $k\in \bb Z$,
introduced in \eqref{64}. The case $p=1$ is excluded in the next lemma
because $\sigma^{1,-}_{k,k+1} =\sigma^{1,+}_{k,k+1}$ for all
$k\in \bb Z$. In the proof of this result, the $\bb X_q$-equivalent
classes are to be understood in the sense of Markov chains and not in
the sense of the equivalence relation introduced at the beginning of
Section \ref{sec7}.

\begin{lemma}
\label{l27}
Fix $2\le p\le \mf q$, $k\in Z$, and suppose that $\sigma^{p,-}_{k,k+1}
<\sigma^{p,+}_{k,k+1}$. Then,
\begin{equation*}
\min \big\{ S(x) : x \in
[\sigma^{p,-}_{k,k+1} \,,\, \sigma^{p,+}_{k,k+1} ] \,\big\} \, >\,
S(\ms M_{p}(k)) \;.
\end{equation*}
\end{lemma}

\begin{proof}
Let $m\in [\sigma^{p,-}_{k,k+1} \,,\, \sigma^{p,+}_{k,k+1} ]$ such that
\begin{equation*}
S(m) \,=\, \min \big\{ S(x) : x \in
[\sigma^{p,-}_{k,k+1} \,,\, \sigma^{p,+}_{k,k+1} ] \,\big\} \;,
\end{equation*}
and suppose that $S(m) \le S(\ms M_{p}(k))$. We obtain a contradiction
showing that, under this condition, $m\in \ms M_{p}(i)$ for some
$i\in\bb Z$.

Fix $q<p$ and suppose that $m\in \ms M_{q}(j)$ for some $j\in\bb
Z$. Recall the definition of the indices $j_\pm$ introduced in the
proof of Lemma \ref{l22}. Since $S(m) \le S(\ms M_{p}(k))$, by the
proof of Lemma \ref{l22},
$\ms M_{q}(j_-) \cup \ms M_{q}(j_+) \subset (\sigma^{p,-}_{k,k+1}
\,,\, \sigma^{p,+}_{k,k+1})$. [Here we use the fact that
$S(\ms M_{q}(j)) = S(m) \le S(\ms M_{p}(k)) \le
S(\sigma^{p,+}_{k,k+1}) -  \mf h_p$.]

Let $\ms M_{q}(i)$, $r_-\le i\le r_+$ be the $\bb X_q$-equivalent class
of $\ms M_{q}(j)$. In particular, $j_-\le r_- \le j\le r_+\le j_+$. By
the previous paragraph, the $\bb X_q$-equivalent class of
$\ms M_{q}(j)$ is contained in
$(\sigma^{p,-}_{k,k+1} \,,\, \sigma^{p,+}_{k,k+1})$.

Thus, $m\in \ms M_{q+1}(j')$ for some $j'\in\bb Z$. Iterating this
argument yields that $m\in \ms M_{p}(j'')$ for some $j''\in\bb Z$, in
contradiction with the definition of $\sigma^{p,+}_{k,k+1}$ [as the
largest local maximum separating $\ms M_{p}(k)$ from
$\ms M_{p}(k+1)$].
\end{proof}

\section{The resolvent equation}
\label{sec-ap1}

In this section, we prove the existence and uniqueness of the
resolvent equation weak solution, as well as its stochastic
representation. We refer to \cite{le} for the definitions omitted
below.  Let $I$ be an open interval of $\bb R$, which can be equal to
$\bb R$.  Denote by $L_{loc}(I)$ the set of Lebesgue locally
integrable functions $f\colon I \to\bb R$. Let $C_c^{\infty}(I)$ be
the space smooth functions $f\colon I \to \bb R$ with compact support.

Given a positive integer $\alpha\in \bb Z_{+}$, let
$\mc H^{\alpha}(I)$ be the Sobolev space of $L^2(I)$ functions $u$
such that, for each $1\leq k \leq \alpha$, the $k$-th weak derivative
of $u$, denoted by $\partial_x^{k}u$, exists and belongs to
$L^2(I)$. For $u\in \mc H^{\alpha}(I)$, define the norm
\begin{equation*}
\lVert u \rVert^2_{\mc H^{\alpha}(I)} :=
\lVert u \rVert_{L^2(I)}^{2}
+ \sum_{k=1}^{\alpha}
\, \lVert \partial_x^{k}u \rVert_{L^2(I)}^{2} \,.
\end{equation*} 
Denote as $\mc H_{loc}^{\alpha}(I)$ the space of locally integrable
functions $u\in L_{loc}(I)$ such that, for every finite open set
$J\subset I$, the restriction of $u$ to $J$ belongs to
$\mc H^{\alpha}(J)$.

\begin{definition}[Weak solution]
Let $F\colon I \to \bb R$ be a bounded measurable function. We say that
a function $\phi \in \mc H^1_{loc}(I)$ is a weak solution of the
resolvent equation
\begin{equation}
\label{111}
(\lambda - \vartheta_\epsilon \,\ms L_\epsilon)\,\phi = F,
\end{equation}
if 
\begin{align}
\label{115}
& \lambda\,\int_{\bb R}\,\phi(x)\,\psi(x)\,dx
\,+\, \epsilon\, \vartheta_\epsilon\, \int_{\bb R}\,
\partial_x\phi(x)\,\partial_x\big[ \mss a(x) \psi(x)\big] \,dx\,
- \, \vartheta_\epsilon\,
\int_{\bb R} \mss b(x)\,\partial_x\phi(x)\,\psi(x)\, dx
\notag \\
&= \int_{\bb R}\,F(x)\,\psi(x)\,dx\,,
\end{align}
holds for all $\psi \in C_c^\infty(I)$.
\end{definition}

It is sometimes convenient to rewrite this equation as follows. Take
as test function $\psi = \hat\psi e^{-S/\epsilon} / \mss a$, for some
$\hat \psi \in C_c^{\infty}(\bb R)$, and recall the definition of $S$
to obtain that
\begin{equation}
\label{113}
\begin{aligned}
& \lambda\,\int_{\bb R}\, \frac{1}{\mss a(x)}\,
\phi(x)\,\hat  \psi(x)\, e^{-S(x)/\epsilon}\,
dx
\,+\, \vartheta_\epsilon\, \int_{\bb R}\,
\partial_x\phi(x)\,\partial_x \hat \psi(x)\,  e^{-S(x)/\epsilon}\, \,dx\,
\notag \\
&= \int_{\bb R} \frac{1}{\mss a(x)} \,F(x)\,
\hat \psi(x)\,  e^{-S(x)/\epsilon} \,dx\,.
\end{aligned}
\end{equation}

Let $I_r=(-r,r)$, $r>0$.  Fix a bounded measurable function
$G\colon I_r\to \bb R$, $\lambda>0$ and constants $\mf c_\pm$.  By
Theorem 6.2.3 with $\gamma=0$ and the Remark just after the definition
of weak solutions in Section 6.1 of \cite{le}, there exists a unique
solution $u\in \mc H^1(I_r)$  to the Dirichlet
problem
\begin{equation}
\label{109}
\left\{
\begin{aligned}
& (\lambda - \vartheta_\epsilon \ms L_\epsilon) \, u = H \quad
\text{in}\;\; I_r \,, \\
& u (\pm r) = \mf c_\pm \,.
\end{aligned}
\right.
\end{equation}
Clearly, if $H=\lambda \, \mf c$, for some constant $\mf c\in \bb R$,
and $\mf c_\pm = \mf c$, the unique solution is $u = \mf c$. This solution is
denoted by $\color{blue} u_{r, \mf c}$.

Suppose that $v\in \mc H^1(I_s)$ is a weak solution of \eqref{111} in
$I_s$ (in other words, \eqref{115} holds for all
$\psi \in C_c^{\infty}(I_s)$). Then, by \cite[Theorem 6.3.1]{le}, for
all $0<r< s$, there exists a finite constant
$C_r = C_r (\epsilon, s, \lambda, \mss a, \mss b)$ such that
\begin{equation}
\label{110}
\Vert v \Vert_{\mc H^2(I_r)} \,\le\,
C_r\, (\Vert H \Vert_{L^2(I_s)} + \Vert v\Vert_{L^2(I_s)})\;.
\end{equation}

Suppose that $H$ in \eqref{109} is uniformly H\"older continuous and
$\mf c_\pm =0$. By the stochastic representation of the solution of
elliptic equations, \cite[Theorem 6.5.1]{Fri},
\begin{equation}
\label{108}
u(x) \;=\;
\bb E^{\epsilon, \vartheta_\epsilon}_x\Big[ \, \int_0^{\tau_r}
e^{-\lambda t}\, H( \mtt x(t))\,dt\, \Big] \;,
\end{equation}
where $\tau_r$ is the hitting time of the boundary of
$(-r,r)$. \smallskip 

We turn to the existence of solutions of \eqref{41}, obtained as the
limit of solutions of Dirichlet problems. By linearity, it is enough
to prove the existence for positive functions $G$.  Denote by $u_n$
the unique solution to the Dirichlet problem \eqref{109} with $r=n$,
$\mf c_\pm =0$, $H=G$.

Approximate $G$ from above and from below by sequences of bounded
uniformly H\"older continuous functions $G_{k,+}$, $G_{k,-}$ in such a
way that $G_{k,-} \le G \le G_{k,+}$,
$\sup_{x\in \bb R} [\, G_{k,+} (x) - G_{k,-} (x) \,] \le 1/k$.  Denote
by $u_{n,k,\pm}$ the unique solution to the Dirichlet problem
\eqref{109} with $r=n$, $\mf c_\pm =0$, $H=G_{k,\pm}$.  By the maximum
principle \cite[Theorem 8.1]{GT}, $u_{n,k,-} \le u_{n} \le
u_{n,k,+}$. Hence, by \eqref{108} for $u_{n,k,\pm}$ and the dominated
convergence theorem,
\begin{equation}
\label{108b}
u_n(x) \;=\;
\bb E^{\epsilon, \vartheta_\epsilon}_x\Big[ \, \int_0^{\tau_n}
e^{-\lambda t}\, G( \mtt x(t))\,dt\, \Big] \;.
\end{equation}
Thus, the stochastic representation of the solution can be extended to
positive linear combination of sets indicators. Moreover, by this
formula, for all $n_0\ge 1$,
\begin{equation}
\label{116}
\lim_{m_0\to\infty}
\sup_{n,m\ge m_0} \sup_{x\in I_{n_0}} \big|\, u_n(x) - u_m(x)\,\big|
\,=\, 0\;.
\end{equation}

By \eqref{116}, \eqref{110}, for each $n_0\ge 1$, the sequence
$(u_n:n\ge 1)$ is Cauchy in $L^\infty(I_{n_0})$ and in
$\mc H^2(I_{n_0})$. By a standard argument one concludes the existence
of a bounded solution to \eqref{111} in $\mc H^2_{\rm loc}(\bb R)$
which can be represented as
\begin{equation}
\label{112}
u(x) \;=\;
\bb E^{\epsilon, \vartheta_\epsilon}_x\Big[ \, \int_0^{\infty}
e^{-\lambda t}\, G( \mtt x(t))\,dt\, \Big] \;.
\end{equation}
In particular,
$\Vert u \Vert_\infty \le \lambda^{-1} \Vert G\Vert_\infty$.

A bounded solution to \eqref{111} in $\mc H^2_{\rm loc}$ is unique.
Indeed, let $u_1$, $u_2$ be two bounded solutions to \eqref{111} in
$\mc H^2_{\rm loc}$, so that $v=u_1-u_2$ is a bounded solution to
\eqref{111} in $\mc H^2_{\rm loc}$ with $H=0$. Fix $n\ge 1$.  The
solution $v$ restricted to $I_n$ solves the Dirichlet problem
\eqref{109} with $H=0$ and boundary conditions $\mf c_\pm = v(\pm n)$.
By the stochastic representation of the solution of elliptic
equations, \cite[Theorem 6.5.1]{Fri},
\begin{equation*}
v(x) \;=\;
\bb E^{\epsilon, \vartheta_\epsilon}_x\Big[ \, v(\mtt x(\tau_n)) \, 
e^{-\lambda \tau_n} \, \Big] \;, \quad x\in I_n\;.
\end{equation*}
As $v$ is bounded, for each fixed $x\in \bb R$, the right-hand side
vanishes as $n\to\infty$. This proves that $v=0$ as claimed.

\subsection*{Acknowledgements}
C. L. has been partially supported by FAPERJ CNE E-26/201.117/2021, by
CNPq Bolsa de Produtividade em Pesquisa PQ 305779/2022-2.

\end{document}